\documentclass[11pt, a4paper]{article}
\usepackage[colorlinks,citecolor=blue,urlcolor=blue]{hyperref}
\usepackage{amsmath, amsthm, amssymb, amsfonts}

\usepackage{a4wide}
\usepackage{setspace}
\usepackage{natbib}

\usepackage[plain,noend]{algorithm2e}

\usepackage{subfig}
\usepackage{bigints}
\usepackage{enumitem}   
\usepackage[capitalise,noabbrev]{cleveref}
\crefname{assumption}{Assumption}{Assumptions}
\usepackage{todonotes}

\newcommand{\norm}[1]{\left\lVert#1\right\rVert}

\newcommand{\prob}[1]{\mathbb{P}\left(#1\right)}
\newcommand{\abs}[1]{\left\lvert #1 \right\rvert}
\newcommand{\s}[2]{\left\langle #1,\,#2 \right\rangle}
\newcommand{\eb}[4]{\text{EB}(#1,#2,#3,#4)}
\newcommand*\diff{\mathop{}\!\mathrm{d}}
\newcommand{\TV}{\mathsf{TV}}
\newcommand{\F}{\mathsf{F}}
\newcommand{\tr}{\mathrm{tr}}
\newcommand{\rev}{}

\providecommand{\cprime}{\ensuremath{^{\prime}}}

\theoremstyle{definition}
\newtheorem{theorem}{Theorem}
\newtheorem{proposition}{Proposition}
\newtheorem{lemma}{Lemma}
\newtheorem{assumption}{Assumption}
\newtheorem{definition}{Definition}
\newtheorem{example}{Example}

\allowdisplaybreaks
\begin{document}

\title{Weak convergence of Bayes estimators under general loss functions}

\author{Robin Requadt, Housen Li and Axel Munk}
\date{Institute of Mathematical Statistics, University of Goettingen}

\maketitle

\begin{abstract}
We investigate the asymptotic behavior of parametric Bayes estimators under a broad class of loss functions that extend beyond the classical translation-invariant setting. To this end, we develop a unified theoretical framework for loss functions exhibiting locally polynomial structure. This general theory encompasses important examples such as the squared Wasserstein distance, the Sinkhorn divergence and Stein discrepancies, which have gained prominence in modern statistical inference and machine learning. Building on the classical Bernstein--von Mises theorem, we establish sufficient conditions under which Bayes estimators inherit the posterior's asymptotic normality. As a by-product, we also derive conditions for the differentiability of Wasserstein-induced loss functions and provide new consistency results for Bayes estimators. Several examples and numerical experiments demonstrate the relevance and accuracy of the proposed methodology.
\end{abstract}

\noindent%

\section{Introduction}

In Bayesian inference (see e.g.~\citealp{Berger85,Bernardo1994,Gelman2014}), the posterior distribution plays a central role by representing the updated beliefs about unknown parameters after observing data. Through Bayes' theorem, it synthesizes prior knowledge with the observed data via the likelihood function. This unifying formulation provides a coherent framework for estimation, prediction, and decision-making under uncertainty. A foundational result in the asymptotic theory of Bayesian procedures is the Bernstein--von Mises (BvM) theorem, which traces back to the seminal insights of \citet{laplace}, and was later rigorously formalized in finite-dimensional settings by \citet{bernstein} and \citet{mises1931wahrscheinlichkeitsrechnung}, see also \cite{Nickl2016,Castillo2015,ghosal_van_der_vaart_2017,Durante2024} for a modern account and extensions to semi- and nonparametric models and non-gaussian limit distributions.  The classical BvM theorem asserts that, in a parametric model, under regularity conditions, the posterior distribution converges in total variation {distance} to a normal distribution centered at an efficient estimator {(e.g.\ the maximum likelihood estimator)} as the sample size grows. This convergence not only bridges Bayesian and frequentist paradigms but also implies that any mapping of the posterior distribution that is continuous with respect to total variation (e.g., the posterior mean or median) converges to the same mapping applied to the limiting normal distribution. {In particular, it implies that Bayesian credible sets can admit an asymptotic frequentist interpretation as confidence sets, see \cite{klein2012} for further discussion.}

Bayes estimators, defined as minimizers of expected posterior loss, can be viewed as such mappings.  The BvM theorem thus offers a powerful framework for analyzing their asymptotic properties. 
%The study of the large-sample behavior of Bayes estimators has long been a central theme in mathematical statistics.
%-new-
 {\rev The canonical example is the posterior mean,  arising under squared Euclidean loss, whose asymptotic properties are classical  \citep{Lecam1953,chao,Ibragimov1973}. Extensions (cf.~\citealp{Bickel1969}) beyond this quadratic setting typically assume that the loss $\ell$ is \emph{translation invariant},
\begin{equation}\label{e:shift-inv}
\ell(t, \vartheta) = \ell_0(t - \vartheta), \quad t, \vartheta \in \Theta,
\end{equation}
for some $\ell_0:\mathbb{R}^d\to[0,\infty)$. An exception is \citet{Bunke98}, where translation invariance is imposed only on the global upper and lower bounds of the loss $\ell$ rather than on $\ell$ itself.
A further common restriction is that the loss $\ell$ depends on a \emph{rescaled} difference of parameters, 
\begin{equation}\label{e:rescale}
\ell(t,\vartheta)=\ell_n(t,\vartheta)=\ell_0(\sqrt{n}(t - \vartheta))
\end{equation}
as considered in \citet{Ibragimov,Jeganathan1982,v,bayes_norm_loss}. This rescaled-loss formulation~\eqref{e:rescale} largely simplifies the asymptotic analysis. In particular, under \eqref{e:rescale}, the rescaled estimator $\sqrt{n}(\hat{\theta}_n - \vartheta_0)$ can be viewed as the minimizer of a posterior risk functional whose integrand no longer depends on the sample size $n$. This invariance eliminates a major technical obstacle in establishing asymptotic results. However, except in the special case where the loss is positively homogeneous, the estimator defined through \eqref{e:rescale} does not coincide with the standard Bayes estimator associated with a fixed loss function.
}

 {\rev By contrast, many loss functions of practical relevance in modern applications  
 neither satisfy the translation invariance condition \eqref{e:shift-inv} 
nor admit global upper and lower bounds that are themselves translation invariant.} Notable examples include total variation distance, Stein discrepancies, intrinsic losses, Bregman divergences and {Wasserstein-type distances {(as well as their surrogates such as the Sinkhorn divergence)}. The latter {have} been a primary motivation for this work,} as they recently receive substantial attention across disciplines owing to its ability to incorporate geometric information and to provide meaningful comparisons between probability distributions. Applications span fields such as image retrieval \citep{Rubner2000}, computer vision and image processing \citep{kolouri2018}, generative machine learning \citep{arjovsky2017}, econometrics and risk management \citep{feng2019}, natural language processing \citep{colombo2021}, {and biology \citep{SCHIEBINGER2019,Tameling2021, bunne2023}. In Bayesian statistics, the Wasserstein distance has also gained prominence, with uses ranging from predictive estimation \citep{Matsuda2021,Backhoff2022} to approximate computation \citep{Bernton2019_2,Ambrogioni2018} and sampling algorithms \citep{minsker2014scalable,srivastava2015wasp,SLD18}. However, the theoretical understanding of the resulting Bayes estimators remains rather limited. For instance, \citet{Matsuda2021} analyzed the Bayesian predictive density for location--scale families (with a one-dimensional scale parameter) under the squared $2$-Wasserstein loss, showing that it reduces to the plug-in estimator based on the posterior mean, thereby falling within the well-understood framework of translation-invariant losses. In a nonparametric setting, \citet{Backhoff2022} established consistency of Bayes estimators with respect to $p$-Wasserstein losses ($p\ge1$) under exponential moment conditions. By contrast, we consider parametric models and study not only consistency but, more importantly, the asymptotic distribution of Bayes estimators under mild moment conditions for general losses, including Wasserstein and translation-invariant ones. As an example, we show (in \Cref{ss:quad:loss}) that the Bayes estimator under the squared 2-Wasserstein loss is consistent and asymptotic normal for the Pareto family $\{P_\vartheta=\mathrm{Pareto}(1,\vartheta) \mid \vartheta\in [a,\infty) \}$ with some $a>2$. In this example, the consistency result of \citet[Theorem~4.24]{Backhoff2022} does not apply, since \(
\int_1^\infty e^{c\norm{x-x_0}^2}\diff P_{3}(x)=\infty,
\) for all $c>0$ and $x_0\ge 1$, which violates their exponential moment assumption.} 

Interestingly, the Bayes estimator under the Wasserstein loss can be interpreted as a \emph{Wasserstein barycenter} over the posterior. Let $\mathcal{W}_2$ denote the set of probability measures on a Polish space $\mathcal{X}$ with finite second moment. Given a probability measure $\Lambda$ on $\mathcal{W}_2$, the corresponding Wasserstein barycenter is the solution to
\[
\inf_{P \in \mathcal{W}_2} F_\Lambda(P), \quad \text{where} \quad F_\Lambda(P) := \int_{\mathcal{W}_2} W_2^2(P, Q)\, \diff\Lambda(Q),
\]
with $W_2$ denoting the 2-Wasserstein distance (see e.g.~\citealp{Agueh-carlier,zemel}). Now consider a parametric family of distributions $\mathcal{M} = \{P_\vartheta : \vartheta \in \Theta\} \subseteq \mathcal{W}_2$, and define a mapping $\Psi : \Theta \to \mathcal{W}_2$, $\vartheta \mapsto P_\vartheta$. If $\Lambda$ is taken as the push-forward of the posterior distribution $P_{\theta \mid {\boldsymbol X} = x}$ under $\Psi$, i.e., $\Lambda = \Psi^\# P_{\theta \mid {\boldsymbol X} = x}$, then
\[
\inf_{P \in \mathcal{M}} F_\Lambda(P) = \inf_{t \in \Theta} \int_\Theta W_2^2(P_t, P_\vartheta)\, \diff P_{\theta \mid {\boldsymbol X} = x}(\vartheta),
\]
and a Bayes estimator $\hat{\theta}_n$ under Wasserstein loss yields that $P_{\hat{\theta}_n}$ is a Wasserstein barycenter over $\mathcal{M}$. In certain cases, such as when $\mathcal{M}$ is a location-scale family, the Wasserstein barycenter uniquely exists {and is part of the parametric family $\mathcal{M}$, see \citet[Theorem~3.10]{esteban2018}.} In this case $P_{\hat{\theta}_n}$ is the population Wasserstein barycenter.

{While Wasserstein distances form a central motivating example, they are not the only important class of loss functions beyond the translation-invariant framework~\eqref{e:shift-inv}. Stein discrepancies, for instance, arise from Stein’s method and are closely related to integral probability metrics.} They are particularly useful in providing tractable and informative distances between distributions, and have been actively studied in both theoretical and applied contexts, see \citet{anastasiou2023} for an overview. {Likewise, Bregman divergences and $f$-divergences} represent a further broad and influential class of loss functions that fall outside the translation invariant setting~\eqref{e:shift-inv}. {These divergences} play a prominent role in both statistics and machine learning due to their deep connections to convexity, information geometry, and decision theory (see e.g.\ \citealp{Csiszar1991,Zhang2004,Banerjee2005}). 

Motivated by the growing importance of such non-standard loss functions, we develop a general theoretical framework for analyzing the asymptotic behavior of Bayes estimators under broad classes of losses, in particular, those that are not translation invariant. {\rev For such losses, the structural simplifications afforded by \eqref{e:shift-inv} or \eqref{e:rescale} are no longer available, and the asymptotic analysis requires substantially more delicate arguments, even though, at a high level, the overall proof strategy remains conceptually aligned with the classical works of \citet{Lecam1953,chao,Ibragimov1973}.}
Our main contributions are as follows:
\begin{enumerate}[label=(\roman*)]
\item 
We establish consistency (\Cref{consistency}) and derive asymptotic distributions (\Cref{main_thm1}) for Bayes estimators associated with general loss functions that exhibit locally polynomial behavior, under comparatively mild regularity conditions. {\rev A key structural insight is that the posterior risk, viewed as a stochastic process indexed by the local parameter (i.e.\ a suitable local rescaling of the parameter),  admits a nondegenerate limit typically characterized by the Hadamard directional derivative of the loss function (\Cref{corollary1}). This derivative naturally emerges as the effective limiting loss, and the asymptotic distribution of the Bayes estimator is therefore governed by a location functional of a Gaussian limit experiment evaluated with respect to this derivative. In this way, the Hadamard directional derivative of the loss plays a central role in linking local posterior asymptotics with the limiting optimization problem.}
\item 
We show that Bayes estimators under squared $2$-Wasserstein loss (\Cref{ss:quad:loss}) and under Stein discrepancies (\Cref{ss:stein}) are asymptotically normal and achieve asymptotic efficiency, thereby extending classical optimality results beyond the translation-invariant framework.
\item 
As an auxiliary contribution, we provide sufficient conditions (\Cref{grad_ws,grad_ws_real} in Supplement~\ref{diff_ws}) for classical differentiability of loss functions induced by the $2$-Wasserstein distance. These are of independent interest for the theoretical analysis of Wasserstein-based procedures.
\end{enumerate}

\emph{Organization of the paper.}
\Cref{section_1} lays the foundation for our results, including a review of the BvM theorem in \Cref{sec1_sub1}, and novel consistency results for Bayes estimators under general loss functions in \Cref{sec1_sub3}. In \Cref{sec2}, we present our main results on the weak convergence of Bayes estimators for losses that behave locally as polynomials. This section also contains a convergence result for the posterior expectation of loss functions, uniformly over compact sets, with its interpretation discussed in Supplement~\ref{ss:had}. \Cref{loc_q} specializes to loss functions that are locally quadratic. In \Cref{sec5}, we provide several examples, including popular loss functions both within and beyond the classical framework~\eqref{e:shift-inv}.  {Even when the conditions of our theoretical framework are not satisfied, one may still expect a distributional limit, though possibly non-Gaussian, see \Cref{ss:beyond}.} Numerical illustrations confirming the theoretical distributional limits of Bayes estimators are presented in \Cref{sec6}.  {\rev A discussion of potential extensions of the proposed framework is given in~\Cref{s:extension}.} All proofs, along with an overview of different types of differentiability, and sufficient conditions for the differentiability of Wasserstein-type losses, are deferred to the supplementary material.

\emph{Notation.} For sequences of positive real numbers $\{a_n\}$ and $\{b_n\}$, we write $a_n \lesssim b_n$ if there exists a constant $C > 0$ such that $a_n \le C b_n$ for all $n$. The Euclidean norm on $\mathbb{R}^d$ is denoted by $\norm{\cdot}$ and the total variance norm of probability measures by $\norm{\cdot}_{\TV}$. For matrices in $\mathbb{R}^{d\times d}$, $\norm{\cdot}_{\F}$ denotes the Frobenius norm, $\tr(\cdot)$ the trace operator, and $\det(\cdot)$ the determinant. For $\delta > 0$, a parameter space $\Theta \subseteq \mathbb{R}^d$ and $\vartheta_0 \in \Theta^\circ$ the interior of $\Theta$, we define the $\delta$-ball around $\vartheta_0$ by $B_\delta(\vartheta_0) := \{ \vartheta \in \Theta : \|\vartheta - \vartheta_0\| \le \delta \}$. By $\mathcal{N}_d(\mu, \Sigma)$, we denote the $d$-variate Gaussian distribution with mean vector $\mu$ and covariance matrix $\Sigma$. We write $\mathbb{E}_P[\cdot]$ to denote the expectation under the probability distribution $P$. For a set $K\subseteq\mathbb{R}^d$, we denote the space of bounded functions from $K$ to $\mathbb{R}$ as 
\(
\ell^\infty(K):=\left\{f:K\to \mathbb{R}:\sup_{x\in K}\abs{f(x)}<\infty\right\}.
\) 

\section{Framework and preparatory results}\label{section_1}

Let $(\Omega, \mathcal{F}, \mathbb{P})$ be a probability space, $(\mathcal{X}, \mathcal{A})$ a measurable space and $(\mathcal{Y}, \mathcal{T})$ a measurable, Polish  space. Consider a random variable $Z: (\Omega, \mathcal{F}, \mathbb{P}) \to (\mathcal{Y}, \mathcal{T})$ and a sequence of random variables $X_1, \dots, X_n : (\Omega, \mathcal{F}, \mathbb{P}) \to (\mathcal{X}, \mathcal{A})$ such that the joint vector $\boldsymbol{X} = (X_1, \dots, X_n)$ is distributed according to a probability measure $P_n$ on $\mathcal{X}^n$. Let $g_n : \mathcal{X}^n \to \mathcal{Y}$ be a sequence of measurable functions. We write
\(
g_n(\boldsymbol{X}) \stackrel{D}{\to}_{P_n} Z, \text{ as } n \to \infty,
\)
to denote weak convergence under $P_n$, namely, for every bounded and continuous function $f : \mathcal{Y} \to \mathbb{R}$, it holds that
\(
\int_{\mathcal{X}^n} f\bigl(g_n(x)\bigr)\diff P_n(x)\to \int_\Omega f\bigl(Z(\omega)\bigr)\diff\mathbb{P}(\omega),\text{ as } n\to \infty.
\)

We now introduce our Bayes setting. Let $\{P_\vartheta : \vartheta \in \Theta\}$ be a parametric family of probability measures on $(\mathcal{X}, \mathcal{A})$, which are absolutely continuous with respect to a common $\sigma$-finite measure. The parameter space $\Theta \subseteq \mathbb{R}^d$ is assumed to have nonempty interior, i.e., $\Theta^\circ \neq \emptyset$. We consider the Bayes model
\begin{equation}\label{e:bayes}
\begin{aligned}
	X_1, \dots, X_n \mid \theta = \vartheta \stackrel{\text{i.i.d.}}{\sim} P_\vartheta, \qquad 
	\theta \sim \Pi,
\end{aligned}
\end{equation}
where $\Pi$ is a prior probability measure on $(\Theta, \mathcal{B}(\Theta))$, with $\mathcal{B}(\Theta)$ being the Borel $\sigma$-algebra induced by the Euclidean topology on $\Theta$.

Fix a true parameter value $\vartheta_0 \in \Theta^\circ$, and define the local parameter $h_n := n^{1/2}(\theta - \vartheta_0)$ as well as the corresponding rescaled local parameter space $H_n := n^{1/2}(\Theta - \vartheta_0)$. Throughout this paper, we adopt a frequentist perspective: We study the behavior of Bayes estimators and posterior expectations under the data distribution $P_{\vartheta_0}^n := \bigotimes_{i=1}^n P_{\vartheta_0}$, corresponding to the assumption that $X_1, \dots, X_n$ are i.i.d.\ observations from $P_{\vartheta_0}$. To ensure that posterior quantities are well-defined under this data distribution, we always assume that
\(
P_{\vartheta_0}^n \ll P^{\boldsymbol{X}},
\)
where $P^{\boldsymbol{X}}$ denotes the marginal distribution of $\boldsymbol{X}$ under the Bayes model~\eqref{e:bayes}.

\subsection{Bernstein--von Mises theorem and its implications}\label{sec1_sub1}
A parametric family of probability distributions $\{P_\vartheta:\vartheta\in \Theta\subset\mathbb{R}^d\}$, dominated by a measure $\mu$ on $\mathcal{X}$, is \emph{differentiable in quadratic mean} at $\vartheta_0 \in \Theta^\circ$ if there exists a vector-valued function $\ell'_{\vartheta_0}:\mathcal{X}\to \mathbb{R}^d$ such that
\[
\int_{\mathcal{X}}\left[p_{\vartheta_0+h}^{1/2}(x)- p^{1/2}_{\vartheta_0}(x)-\frac{1}{2}\s{h}{\ell'_{\vartheta_0}(x)p^{1/2}_{\vartheta_0}(x)}\right]^2\diff \mu(x)=o(\norm{h}^2),\quad h\to 0,
\]
where $p_\vartheta=\diff P_\vartheta/\diff \mu$. We refer to $\ell'_{\vartheta_0}$ as the \emph{quadratic mean derivative} at $\vartheta_0$. In this case, the Fisher information matrix at $\vartheta_0$ is defined as
\(
I_{\vartheta_{0}}=\int_{\mathcal{X}} \ell_{\vartheta_0}'(x)\ell_{\vartheta_0}'(x)^\intercal \diff P_{\vartheta_0}(x)\in \mathbb{R}^{d\times d},
\)
which has finite components if the family $\{P_\vartheta\}_{\vartheta\in \Theta}$ is differentiable in quadratic mean at $\vartheta_0$.

%\todo[inline]{Better: Change in all places $\ell'_{\vartheta_0}$ for a different notation, as $\ell$ is reserved for the loss}

The celebrated Bernstein--von Mises (BvM) theorem, forms the foundation of our analysis of the asymptotic behavior of Bayes estimators.  

\begin{theorem}[BvM theorem; {\citealp[Theorem~10.1]{v}}]
	\label[theorem]{B.v.M:} 
	Let $\{P_\vartheta:\vartheta\in \Theta\subseteq \mathbb{R}^d\}$ be a parametric family of probability measures, each absolutely continuous with respect to a common $\sigma$-finite measure $\mu$ on $\mathcal{X}$, and assume this family is differentiable in quadratic mean at $\vartheta_{0}\in \Theta^\circ$ with derivative $\ell'_{\vartheta_{0}}$ and nonsingular Fisher information matrix $I_{\vartheta_{0}}$.  Let ${\boldsymbol X} = (X_1,\ldots, X_n)$ and suppose that, for every $\varepsilon>0$, there exists a sequence of tests $\phi_n:\mathcal{X}^n\to [0,1]$ such that, as $n\to \infty$,
\[
\mathbb{E}_{P^n_{\vartheta_{0}}}\left[\phi_n({\boldsymbol X})\right]\to 0\qquad \text{and}\qquad \sup_{\vartheta\in \Theta:\norm{\vartheta-\vartheta_0}\ge \varepsilon}\mathbb{E}_{P_\vartheta^n}\left[1-\phi_n({\boldsymbol X})\right]\to 0.
\]
Define
\(
\Delta_{n,\vartheta_0}({\boldsymbol X}):=I_{\vartheta_0}^{-1}n^{-1/2}\sum_{i=1}^{n}\ell'_{\vartheta_0}(X_i).
\)
Then, under the Bayes model~\eqref{e:bayes}, if the prior measure $\Pi$ has a Lebesgue density $\pi$ that is continuous in a neighborhood of $\vartheta_{0}$ and satisfies $\pi(\vartheta_{0})>0$, it holds for $h_n:=n^{1/2}(\theta-\vartheta_0)$ that, as $n\to \infty$, 
\[
\norm{P_{h_n\vert {\boldsymbol X} }- \mathcal{N}_d\left(\Delta_{n,\vartheta_0}({\boldsymbol X}),I_{\vartheta_0}^{-1}\right)}_{\TV}\stackrel{P_{\vartheta_0}^n}{\to}0.
\] 
\end{theorem}

\subsection{Bayes estimator and its consistency}\label{sec1_sub3}

A key concept in the theory of Bayes estimators is the loss function. We define a \emph{loss function} $\ell:\Theta\times \Theta\to [0,\infty)$ as a jointly measurable function on $\Theta\subseteq \mathbb{R}^d$  that satisfies $\ell(t,\vartheta)=0$ if and only if $t=\vartheta$. Throughout, we assume:

\begin{assumption}[{Integrability}]
 For the loss function $\ell$, we assume that for every $t\in \Theta$,
 \(
\int_{\Theta}\ell(t,\vartheta)\diff\Pi(\vartheta)<\infty,
\)
where $\Pi$ is the prior measure in the Bayes model~\eqref{e:bayes}.
\end{assumption}

This assumption ensures that posterior expectations of the loss function are well-defined.  
We define a \emph{Bayes estimator} with respect to the loss function $\ell$ as a measurable function $\hat{\theta}_n:\mathcal{X}^n\to \Theta$ such that
\[
\hat{\theta}_n(x)\in \mathop{\arg \min}_{t\in \Theta}\int_{\Theta}\ell(t,\vartheta)\diff P_{\theta\vert {\boldsymbol X}=x}(\vartheta)
\]
for $P^{\boldsymbol X}$-almost every (a.e.) $x\in \mathcal{X}^n$. We assume throughout that a Bayes estimator exists. For sufficient conditions ensuring existence, see e.g.~\citet[Proposition 4.1]{Shao_2003_book}. With slight abuse of notation, we do not distinguish between the function $\hat{\theta}_n$ and the random variable $\hat{\theta}_n({\boldsymbol X})$, and often write simply $\hat{\theta}_n$.

In the study of M-estimators, a well-separated minimum is typically assumed to guarantee consistency, see for example \citet[Section~3.2]{vanwell} or \citet[Chapter~5]{v}. Specifically, $\ell(\cdot,\vartheta)$ has a well-separated minimum at $\vartheta\in \Theta$, if for any $M>0$
 \[
 \inf_{t\in \Theta:\norm{t-\vartheta}>M}\ell(t,\vartheta)>0.
 \] 
 To establish consistency of Bayes estimators, we introduce the following strengthened {version of a well-separated-minimum} condition, which uniformly controls how the loss behaves away from the true parameter.

\begin{assumption}[Uniformly well-separated minimum] \label[assumption]{mon}
For any $M > 0$, there is a $\delta > 0$ such that
\[
\inf_{\substack{t,\vartheta \in \Theta \\ \|t - \vartheta_0\| > M, \, \|\vartheta - \vartheta_0\| \le \delta}} \ell(t, \vartheta) > 0.
\]
\end{assumption}

This condition guarantees that for any sufficiently small neighborhood around the true parameter $\vartheta_0$, all points far away from $\vartheta_0$ incur a strictly positive loss. Several natural loss functions satisfy this condition, as illustrated below.

\begin{proposition} \label[proposition]{unif_well_implied}
The following loss functions $\ell : \Theta \times \Theta \to [0, \infty)$ satisfy \Cref{mon}:
\begin{enumerate}[label=(\roman*)]
    \item \label{i:uwi:i}
    $\ell(t, \vartheta) = \rho(t, \vartheta)^p$ for some $p > 0$, where $\rho : \Theta \times \Theta \to [0, \infty)$ is a metric such that $t \mapsto \rho(t, \vartheta_0)$ is continuous at $\vartheta_0$ (with respect to the Euclidean distance) and has a well-separated minimum at $\vartheta_0$, i.e.,
    \(
    \inf_{t \in \Theta : \|t - \vartheta_0\| > M} \rho(t, \vartheta_0) > 0  \text{ for any } M > 0.
    \)
    
    \item \label{i:uwi:ii}
    $\ell(t, \vartheta) {\rev\, \ge\,}g(\|t - \vartheta\|)$, where $g : [0, \infty) \to [0, \infty)$ is a non-decreasing function satisfying $g(x) = 0$ if and only if $x = 0$.
\end{enumerate}
\end{proposition}

\Cref{unif_well_implied}~\ref{i:uwi:i} includes in particular the {squared Wasserstein loss}
\(
\ell(t, \vartheta) = W_2^2(P_t, P_\vartheta), \; t, \vartheta \in \Theta,
\)
provided that the mapping $t \mapsto W_2(P_t, P_{\vartheta_0})$ is continuous and has a well-separated minimum at~$\vartheta_0$. As a technical preparation, we introduce the following growth condition.

\begin{definition}[Exponential growth]\label[definition]{d:eg}
 A function $f : A \times A \to [0,\infty)$, $A\subseteq \mathbb{R}^d$, is said to be \emph{exponentially bounded} with parameters $c_1, c_2, c_3 \in \mathbb{R}$ relative to a subset $K \subseteq A$ if 
\[
f(t,\vartheta) \leq c_1 \exp\!\left( c_2 \|t\|^{c_3} + c_2 \|\vartheta\|^{c_3} \right)\quad \text{for all }t \in K\text{ and }\vartheta \in A.
\]
In this case, we write $f \in \eb{c_1}{c_2}{c_3}{K}$.
\end{definition}

As an alternative to \Cref{mon}, we can employ convexity of loss functions.

\begin{proposition}[Consistency]\label[proposition]{consistency}
Let the conditions of the BvM theorem (\Cref{B.v.M:}) hold, and suppose that the loss function $\ell : \Theta \times \Theta \to [0, \infty)$ is exponentially bounded (see~\Cref{d:eg}), i.e.~$\ell \in  \eb{c_1}{c_2}{c_3}{\Theta}$ with constants $c_1, c_2 > 0$ and $0<c_3 < 2$. Assume that $\vartheta \mapsto \ell(t, \vartheta)$ is continuous at $\vartheta_0 \in \Theta^\circ$ for every $t \in \Theta$, and further that one of the following holds:
\begin{enumerate}[label=(\roman*)]
    \item 
    \Cref{mon} holds. \label{i:sep}
    \item \label{i:conv}
    The set $\Theta$ is convex, and for each $\vartheta \in \Theta$, the mapping $t \mapsto \ell(t, \vartheta)$ is convex. 
%    Additionally, for every $M > 0$, 
%    \[
%    \inf_{t \in \Theta : \|t - \vartheta_0\| > M} \ell(t, \vartheta_0) > 0.
%    \]
\end{enumerate}
Then, the Bayes estimator $\hat{\theta}_n$ with respect to the loss function $\ell$ is consistent, i.e.
\[
\hat{\theta}_n \xrightarrow{P^n_{\vartheta_0}} \vartheta_0, \quad \text{as } n \to \infty.
\]
\end{proposition}

Note that a loss function satisfying condition~\ref{i:conv} always has a well-separated minimum at $\vartheta_0$, but may not satisfy \Cref{mon}.

\section{Weak convergence of Bayes estimators for locally polynomial losses}\label{sec2}
We now present our main theorem, which establishes general conditions under which Bayes estimators converge weakly when the loss function behaves locally like a polynomial. 

\begin{theorem}[Weak convergence]\label{main_thm1}
Suppose the conditions of the BvM theorem (\Cref{B.v.M:}) hold, and there exists $\delta>0$ such that the loss function $\ell:\Theta\times \Theta\to [0,\infty)$ satisfies
\(
\int_\Theta\sup_{t\in B_\delta(\vartheta_0)}\ell(t,\vartheta)\,\mathrm{d} \Pi(\vartheta)<\infty.
\)  
For some $p >0$, assume the following:
\begin{enumerate}[label=(\roman*)]
\item\label{Thm4.1a1} 
There exist $\varepsilon>0$ and some constants $c_1,c_2>0$ such that
\[
c_1\|t-\vartheta\|^p\le \ell(t,\vartheta)\le c_2\|t-\vartheta\|^p \qquad \text{for all } t,\vartheta \in B_\varepsilon(\vartheta_{0}),
\]
and moreover $\int_{\Theta}\|\vartheta\|^p\,\mathrm{d} \Pi(\vartheta)<\infty$.

\item \label{Thm4.1a3}
There exist $\delta>0$ and constants $c_3, c_4>0$ and $0<c_5<2$ such that
{$\ell \in \eb{c_3}{c_4}{c_5}{B_\delta(\vartheta_0)}$, see \Cref{d:eg}.}

\item \label{Thm4.1a2} 
For any compact set $K\subseteq \mathbb{R}^d$, we have $Z_n\stackrel{D}{\to}_{P^n_{\vartheta_{0}}} Z$ in $\ell^\infty(K)$, where
\[
Z_n(t):=\int_{ H_n}n^{p/2}\ell\!\left(\vartheta_{0}+\tfrac{t}{n^{1/2}},\,\vartheta_{0}+\tfrac{h}{n^{1/2}}\right)\mathrm{d} P_{h_n\vert {\boldsymbol X}}(h), 
\quad Z(t):=\int_{ \mathbb{R}^d}\ell_0(t,h)\,\mathrm{d} \mathcal{N}_d(Y,I^{-1}_{\vartheta_{0}})(h),
\]
with $Y\sim \mathcal{N}_d(0,I^{-1}_{\vartheta_0})$. The function $\ell_0:\mathbb{R}^d\times \mathbb{R}^d\to [0,\infty)$ is a loss function such that the mapping 
\[
t\mapsto \int_{ \mathbb{R}^d}\ell_0(t,h)\,\mathrm{d} \mathcal{N}_d(Y,I^{-1}_{\vartheta_{0}})(h)
\]
is almost surely continuous and admits a minimizer $G$, which is almost surely unique.
\end{enumerate}
If $\hat{\theta}_n$ is a Bayes estimator with respect to $\ell$ and $\hat{\theta}_n\stackrel{P^n_{\vartheta_{0}}}{\to}\vartheta_{0}$, then
\(
n^{1/2}\bigl(\hat{\theta}_n-\vartheta_{0}\bigr)\stackrel{D}{\to}_{P^n_{\vartheta_0}}G. 
\)
\end{theorem}

Conditions~\ref{Thm4.1a1}--\ref{Thm4.1a2} are weaker than the standard assumptions in the literature and accommodate a broader class of loss functions. In particular, \ref{Thm4.1a1} ensures that the loss behaves locally like a power of the Euclidean norm. Its local lower bound is needed for uniform tightness of $n^{1/2}(\hat{\theta}_n-\vartheta_0)$, but not for the {weak convergence $Z_n\stackrel{D}{\to}_{P^n_{\vartheta_{0}}} Z$ in $\ell^\infty(K)$} required in~\ref{Thm4.1a2} (see \Cref{weak_conv_prop3,corollary1}). { We note, however, that certain losses fall outside our framework. 
 For instance, $\ell(t,\vartheta)=\abs{t-\vartheta}\log(\abs{t-\vartheta}^{-1})$ for $t\neq \vartheta$ (and $0$ otherwise) with {$\Theta=[0,1)$} does not satisfy \ref{Thm4.1a1}.

An extension of our theory to this setting might be possible by incorporating a slowly varying function \citep{BiOs13} into both the upper and lower bounds in \ref{Thm4.1a1}.}

Condition~\ref{Thm4.1a3} ensures that integrals over the sets $\{h\in H_n:\|h\|>M_n\}$ are asymptotically negligible by \Cref{weak_conv_prop1} (in Supplement~\ref{technical}). This can be weakened to
\[
n^{p/2}\ell\!\left(\vartheta_{0}+\tfrac{t}{n^{1/2}},\,\vartheta_{0}+\tfrac{h}{n^{1/2}}\right)\le c_1\exp\!\big(c_2\|t\|^a+c_3\|h\|^a\big)
\]
for all sufficiently large $n$, with constants $c_1,c_2,c_3>0$ and $0<a<2$. This follows from the fact that {for some $\gamma>0$,}
\(
\ell(t,\vartheta)\lesssim \|t-\vartheta\|^p \exp\!\big(\gamma(\|t\|^a+\|\vartheta\|^a)\big), \; t,\vartheta\in \Theta,
\)
which is implied by \ref{Thm4.1a1} and \ref{Thm4.1a3}.

Condition~\ref{Thm4.1a2} is needed to apply the argmin continuous mapping theorem \cite[Corollary~5.58]{v}. Concrete conditions guaranteeing the weak convergence are given in \Cref{weak_conv_prop3} and \Cref{corollary1}, where the link between $\ell$ and $\ell_0$ is also clarified.

The assumed consistency of $\hat{\theta}_n$ ensures that it concentrates in any arbitrarily small neighborhood of $\vartheta_{0}$ with probability tending to one. Consequently, the asymptotic behavior of $n^{1/2}(\hat{\theta}_n-\vartheta_{0})$ can be related to that of the localized and rescaled process
\[
t\longmapsto \int_{n^{1/2}N}n^{p/2}\ell\!\left(\vartheta_{0}+\tfrac{t}{n^{1/2}},\,\vartheta_{0}+\tfrac{h}{n^{1/2}}\right)\mathrm{d} P_{h_n\vert {\boldsymbol X}}(h),
\]
where $N$ is a small neighborhood of the origin. 

Losses of the form $\ell(t,\vartheta)=\|t-\vartheta\|^p$ satisfy the assumptions of \Cref{main_thm1} (see \Cref{sec1_sub3}, and \Cref{weak_conv_prop1,weak_conv_prop2} in Supplement~\ref{technical}). By \citet[Lemma~8.5]{v}, in this case $G=Y$, implying asymptotic efficiency of Bayes estimators under such losses. {Further examples are given below and in the next section.

\begin{example}\label{b:integer:order}
Let $k \ge 2$ be an integer. Suppose that $t \mapsto \ell(t,\vartheta)$ is $(k+1)$-times continuously differentiable in a neighborhood of $\vartheta_0$, and denote its $j$th-order derivative at $\vartheta$  by
    \[
    D^j(\vartheta)[u]:=\sum_{i_1=1}^d\cdots\sum_{i_j = 1}^d\frac{\partial^j \ell(t,\vartheta)}{\partial t_{i_1}\cdots\partial t_{i_j} }\bigg \vert_{t=\vartheta}u_{i_1}u_{i_2} \cdots  u_{i_j},\quad u \in \mathbb{R}^d, \quad  j \in\{1,\ldots, k\}.
    \] 
Assume that at $\vartheta_0$, we have $D^j(\vartheta_0)[u]=0$ if $j<k$ and $D^k(\vartheta_0)[u]>0$, for every $u\in\mathbb{R}^d$. Then, by continuity (see Supplement~\ref{ss:detail:io} for details), there exists $\varepsilon>0$ such that for all $t, \vartheta \in B_\varepsilon(\vartheta_0)$, the upper and lower bounds in \Cref{main_thm1}~\ref{Thm4.1a1} hold with $p=k$, i.e.
\[
c_1\|t-\vartheta\|^k\le \ell(t,\vartheta)\le c_2\|t-\vartheta\|^k \qquad \text{for some constants } c_1,c_2>0.
\]
The above argument can be applied to derive local bounds for $p$-Wasserstein distances between two parametric measures. In particular, we will consider the squared 2-Wasserstein loss in \Cref{loc_q}. For general $p\ge1$, however, it remains an open problem to determine under which conditions certain powers of the $p$-Wasserstein distance {are} $(p+1)$-times continuously differentiable.

In addition, we emphasize that the $(k+1)$-times differentiability requirement is only a sufficient condition, not a necessary one. For example, if $\ell(t,\vartheta)=|t-\vartheta|^k$ with odd $k$, the $k$th-order derivative is not continuous at $t=\vartheta$, yet the displayed condition above is trivially satisfied.
\end{example}}

The next lemma provides sufficient conditions under which \Cref{main_thm1}~\ref{Thm4.1a2} holds. It extends \Cref{weak_conv_prop2} (in Supplement~\ref{technical}) to sequences of loss functions.

\begin{lemma}\label[lemma]{weak_conv_prop3}
Assume the conditions of the BvM theorem (\Cref{B.v.M:}). Let $K\subset \mathbb{R}^d$ be compact, and let $\ell_0:\mathbb{R}^d\times \mathbb{R}^d\to [0,\infty)$ and $\ell_n: H_n\times H_n\to [0,\infty)$ be loss functions satisfying:
\begin{enumerate}[label=(\roman*)]
\item \label{prop5a1}
For some constants $c_1>0$ and $0< c_2<{1}$,
\(
\int_{\Theta}\sup_{t\in K}\ell_n\!\bigl(t,n^{1/2}(\vartheta-\vartheta_{0})\bigr)\,\mathrm{d} \Pi(\vartheta) \lesssim e^{c_1n^{c_2}}.
\)
\item \label{prop5a2}
For all $h\in \mathbb{R}^d$,
\(
\sup_{t\in K}\big|\ell_n(t,h)-\ell_0(t,h)\big|\to 0  \text{ as } n\to \infty.
\)
\item \label{prop5a3}
{For some constants $c_3,c_4>0$ and $0<c_5<2$ it holds $\ell_0\in \eb{c_3}{c_4}{c_5}{\mathbb{R}^d}$ and for all sufficiently large $n$ it holds $\ell_n\in \eb{c_3}{c_4}{c_5}{K}$, see \Cref{d:eg}.}
\end{enumerate}
Then, with
\(
Z_n(t)=\int_{ H_n}\ell_n(t,h)\,\mathrm{d} P_{h_n\vert {\boldsymbol X}}(h) \text{ and } 
Z(t)=\int_{ \mathbb{R}^d}\ell_0(t,h)\,\mathrm{d} \mathcal{N}_d(Y,I^{-1}_{\vartheta_{0}})(h),
\)
where $Y\sim \mathcal{N}_d(0,I^{-1}_{\vartheta_{0}})$, we have $Z_n\stackrel{D}{\to}_{P^n_{\vartheta_{0}}}Z$ in $\ell^\infty(K)$.
\end{lemma}

In \Cref{weak_conv_prop3}, condition~\ref{prop5a2} links $\ell_n$ and $\ell_0$, ensuring convergence of $Z_n$ to $Z$. Condition~\ref{prop5a3} controls the growth of the loss functions, guarantees boundedness on compact sets, and allows the use of \Cref{weak_conv_prop1} (in Supplement~\ref{technical}). In this paper, \Cref{weak_conv_prop3} is applied with
\[
\ell_n(t,h)=n^{p/2}\,\ell\!\left(\vartheta_{0}+\tfrac{t}{n^{1/2}},\,\vartheta_{0}+\tfrac{h}{n^{1/2}}\right),
\]
for which $\ell_0(t,h)$ is typically the directional Hadamard derivative of $\ell$ at $(\vartheta_0,\vartheta_0)$ in direction $(t,h)$. In this case, the weak limit $G$ minimizes, almost surely, the expected slope for $p=1$, the expected curvature for $p=2$, and so on, of $\ell$ at $(\vartheta_0,\vartheta_0)$ in direction $(G,h)$, with $h\sim \mathcal{N}_d(Y,I^{-1}_{\vartheta_0})$ and $Y\sim \mathcal{N}_d(0,I^{-1}_{\vartheta_0})$ (see also Supplement~\ref{ss:had}).

\begin{proposition}\label[proposition]{corollary1}
Assume the conditions of the BvM theorem (\Cref{B.v.M:}). Let $K\subset \mathbb{R}^d$ be compact and $\ell,\ell_0:\Theta\times \Theta\to [0,\infty)$ be loss functions satisfying:
\begin{enumerate}[label=(\roman*)]
\item \label{cor1a1}
There exist $p>0$ and $\varepsilon>0$ sufficiently small such that, for some $c_1>0$,
\[
\ell(t,\vartheta)\le c_1\|t-\vartheta\|^p \qquad \text{for all }\, t,\vartheta\in B_\varepsilon(\vartheta_{0}).
\]
\item \label{cor1a2} 
For some $\delta>0$ sufficiently small,
\(
\int_{\Theta}\sup_{t\in B_\delta(\vartheta_0)}\ell(t,\vartheta)\,\mathrm{d} \Pi(\vartheta)<\infty.
\)
\item \label{cor1a3}
For all $h\in \mathbb{R}^d$,
\(
\sup_{t\in K}\left|\varepsilon^{-p}\ell(t\varepsilon+\vartheta_{0},h\varepsilon+\vartheta_{0})-\ell_0(t,h)\right|\to 0  \text{ as } \varepsilon\to 0,\, \varepsilon>0.
\)
\item \label{cor1a4}
There exist $c_2, c_3>0$, $0< c_4<2$ and $\delta>0$ such that $\ell\in \eb{c_2}{c_3}{c_4}{B_\delta(\vartheta_0)} $ and $\ell_0\in \eb{c_2}{c_3}{c_4}{\mathbb{R}^d} $, see \Cref{d:eg}. 
\end{enumerate}
Then, with
\[
Z_n(t)=\int_{ H_n}n^{p/2}\,\ell\!\left(\tfrac{t}{n^{1/2}}+\vartheta_{0},\tfrac{h}{n^{1/2}}+\vartheta_{0}\right)\mathrm{d} P_{h_n\vert {\boldsymbol X}}(h),\quad 
Z(t)=\int_{ \mathbb{R}^d}\ell_0(t,h)\,\mathrm{d} \mathcal{N}_d(Y,I^{-1}_{\vartheta_{0}})(h),
\]
where $Y\sim \mathcal{N}_d(0,I^{-1}_{\vartheta_{0}})$, we have $Z_n\stackrel{D}{\to}_{P^n_{\vartheta_{0}}}Z$ in $\ell^\infty(K)$.
\end{proposition}

If two loss functions $\ell_1$ and $\ell_2$ satisfy the assumptions of Proposition~\ref{corollary1} with the same $p>0$, then any positive linear combination
\(
\tilde{\ell}(t,\vartheta) =a\,\ell_1(t,\vartheta)+b\,\ell_2(t,\vartheta),\; a,b>0,
\)
also satisfies the assumptions.

\section{Asymptotic efficiency of Bayes estimators under locally quadratic losses}\label{loc_q}
We now consider the case where the loss function is locally quadratic.

\begin{assumption}\label[assumption]{assumption3}
Let $A:\Theta\to \mathbb{R}^{d\times d}$ be a matrix-valued mapping whose components are continuous on $\Theta$, and suppose $A(\vartheta_0)$ is positive definite. We assume that for each $\vartheta\in\Theta$, the mapping $t\mapsto\ell(t,\vartheta)$ admits the local expansion
\begin{equation}\label{e:quad:expan}
\ell(t,\vartheta) = \langle t-\vartheta,\, A(\vartheta)(t-\vartheta) \rangle + \xi(t,\vartheta),
\end{equation}
where the remainder $\xi(t,\vartheta)$ satisfies
\(
|\xi(t,\vartheta)| \lesssim \|t-\vartheta\|^{c_1}  \text{ for all } t,\vartheta\in B_\varepsilon(\vartheta_0),
\)
for some $\varepsilon>0$ and $c_1>2$.
Moreover, the following hold:
\begin{enumerate}[label=(\roman*)]
\item\label{i:ass3:1} There exists $\delta>0$ such that
\(
\int_\Theta \sup_{t\in B_\delta(\vartheta_0)} \ell(t,\vartheta)\,\mathrm{d}\Pi(\vartheta) < \infty.
\)
\item\label{i:ass3:2}  {\rev There exist $\delta, c_2, c_3>0$ and $0< c_4<2$ such that $\ell\in \eb{c_2}{c_3}{c_4}{B_\delta(\vartheta_0)} $, see~\Cref{d:eg}.}
\iffalse %old 
There exist $\delta>0$, $c_2>0$, and $0< a<2$ such that for all $t\in B_\delta(\vartheta_{0})$ and $\vartheta\in\Theta$,
\begin{align*} 
|\xi(t,\vartheta)| \lesssim \exp\big(c_2(\|t\|^{a}+\|\vartheta\|^{a})\big),
\text{ and } \|A(\vartheta)\|_{\F} \lesssim \exp\big(c_2\|\vartheta\|^{a}\big).
\end{align*}
\fi
\end{enumerate}
\end{assumption} 

%The growth bounds on $\xi$ and $A$ in \Cref{assumption3} are mild and are satisfied in many settings. 
{\rev The growth bounds in \Cref{assumption3}~\ref{i:ass3:1}--\ref{i:ass3:2} are rather mild and hold in many settings.} The key requirement is the expansion~\eqref{e:quad:expan}. If $t\mapsto\ell(t,\vartheta)$ is thrice continuously differentiable for every $\vartheta\in \Theta$, then \eqref{e:quad:expan} follows from the second-order Taylor expansion whenever $\Theta\subseteq\mathbb{R}^d$ is convex. In this case, \Cref{assumption3} holds with $c_1=3$ and $A(\vartheta) = \frac{1}{2} \nabla^2_t \ell(t,\vartheta)\big|_{t=\vartheta}$, provided $A(\vartheta_0)$ is positive definite.

\Cref{assumption3} ensures that the loss function $\ell$ behaves locally like a quadratic form, that is, there exists $\delta>0$ such that for all $t,\vartheta \in B_\delta(\vartheta_0)$,
\begin{equation}\label{e:up:lo:quad}
\|t-\vartheta\|^2 \lesssim \ell(t,\vartheta) \lesssim \|t-\vartheta\|^2,
\end{equation}
{\rev see \Cref{sss:technical}.} Given the local expansion \eqref{e:quad:expan}, the sequence of loss functions
\begin{align*}
\ell_n(t,h) 
&= n^{p/2} \ell\!\left(\tfrac{t}{n^{1/2}}+\vartheta_{0},\, \tfrac{h}{n^{1/2}}+\vartheta_{0}\right) \\
&= n^{{p}/{2}-1} \langle t-h,\, A(\tfrac{h}{n^{1/2}}+\vartheta_{0})(t-h) \rangle
\;+\; n^{p/2}\,\xi\!\left(\tfrac{t}{n^{1/2}}+\vartheta_{0},\, \tfrac{h}{n^{1/2}}+\vartheta_{0}\right)
\end{align*}
satisfies condition~\ref{cor1a3} in \Cref{corollary1} with $p=2$ and a quadratic loss
\(
\ell_0(t,h) = \langle t-h,\, A(\vartheta_0)(t-h) \rangle.
\)
Thus, it is sufficient for $\ell$ to be locally quadratic in order for the entire process
\(
Z_n(t) = \int_{H_n} \ell_n(t,h)\,\mathrm{d}P_{h_n\vert {\boldsymbol X}}
\)
to converge to
\(
Z(t) = \int_{\mathbb{R}^d} \langle t-h,\, A(\vartheta_0)(t-h) \rangle\,\mathrm{d}\mathcal{N}_d(Y,I^{-1}_{\vartheta_0})(h),
\)
the limiting process for the purely quadratic case. This highlights that controlling the local behavior of the loss suffices to deduce the asymptotic behavior of the entire process. This observation can be used to adapt the theory to other sequences of loss functions.

\begin{assumption}\label[assumption]{assumption2}
The conditions of the BvM theorem (\Cref{B.v.M:}) hold, and the prior has a finite second moment, i.e.\
\(
\int_\Theta \|\vartheta\|^{2}\,\mathrm{d} \Pi(\vartheta) < \infty.
\)
\end{assumption}

\begin{theorem}\label{main_thm2}
Let Assumptions~\ref{assumption3}--\ref{assumption2} hold, and suppose the corresponding Bayes estimator $\hat{\theta}_n$ is consistent in the sense that $\hat{\theta}_n \stackrel{P_{\vartheta_{0}}^n}{\to} \vartheta_0$. Then, as $n\to\infty$,
\[
n^{1/2}\big(\hat{\theta}_n({\boldsymbol X})-\vartheta_0\big) \stackrel{D}{\longrightarrow}_{P^n_{\vartheta_0}} Y \sim \mathcal{N}_d(0,I^{-1}_{\vartheta_0}),
\]
i.e.\ $\hat{\theta}_n$ is asymptotically efficient.
\end{theorem}

If $\Theta$ is convex and $\psi:\Theta\to\mathbb{R}$ is a thrice continuously differentiable convex function, then many Bregman divergence losses
\(
\ell(t,\vartheta) = d_\psi(t,\vartheta)
:= \psi(t) - \psi(\vartheta) - \langle t-\vartheta,\, \nabla_\vartheta\psi(\vartheta) \rangle
\)
satisfy the assumptions of \Cref{main_thm2}.

\section{Examples} \label{sec5}

We review here several popular loss functions that behave locally like a polynomial and thus fall within our framework.  A natural and widely used class is given by
\(
\ell(t,\vartheta)=\norm{t-\vartheta}^p, \; p > 0,
\)
which clearly satisfies the assumptions of \Cref{main_thm1}.  Special cases include the posterior mean for $p=2$ and the posterior (geometric) median for $p=1$.  The asymptotic behavior of Bayes estimators under such losses has been extensively studied, and  in fact, stronger results are known.  For instance, as shown in \citet{chao}, the difference between a Bayes estimator corresponding to these loss functions and the maximum likelihood estimator converges to zero at a rate faster than $n^{-1/2}$.  This directly implies that the two estimators share the same asymptotic distribution. 

Our focus, however, is on loss functions beyond this classical and well-understood setting. 
In particular, we study a broad class of losses induced by distances between probability measures that satisfy \Cref{assumption3} but are not encompassed by most existing results.  

\subsection{Locally quadratic loss functions}\label{ss:quad:loss}
\subsubsection{Intrinsic loss functions}
Examples of such losses include the so-called \emph{intrinsic losses} \citep{intrinsic}, which depend on the parametrization of the statistical model.  
A primary motivation for using intrinsic losses is the \emph{parametrization invariance} of the resulting Bayes estimators.

If $\hat{\theta}_n$ is a Bayes estimator with respect to a loss function $\ell$, we say that $\hat{\theta}_n$ is \emph{parametrization invariant} if, for any measurable function $g:\Theta \to \mathbb{R}^d$, the transformed estimator $g(\hat{\theta}_n)$ is a Bayes estimator for $g(\theta)$ under the same loss $\ell$.  
In general, Bayes estimators do not enjoy this property.  
For instance, under the quadratic loss $\ell(t,\vartheta) = \norm{t-\vartheta}^2$, the Bayes estimator is the posterior mean $\int_{\Theta} \vartheta \, \diff P_{\theta \mid {\boldsymbol X}}(\vartheta)$.  
The Bayes estimator of $g(\theta)$ would be $\int_{\Theta} g(\vartheta) \, \diff P_{\theta \mid {\boldsymbol X}}(\vartheta)$, which typically differs from $g\!\bigl( \int_{\Theta} \vartheta \, \diff P_{\theta \mid {\boldsymbol X}}(\vartheta) \bigr)$. 

In contrast, for an intrinsic loss $\ell$, the corresponding Bayes estimator \emph{is} parametrization invariant, as shown in \citet[Lemma~6.2]{intrinsic}.  
Formally, let $\mathcal{M}=\{P_\vartheta : \vartheta \in \Theta\}$ be a fixed parametric family of probability measures, each absolutely continuous with respect to Lebesgue measure on $\mathbb{R}^d$.  
The parametrization mapping $\Psi:\Theta \to \mathcal{M}$ is given by $\vartheta \mapsto P_\vartheta$.  
We denote by $p_t$ the Lebesgue density of $P_t$.  
A few examples of intrinsic losses that satisfy the assumptions of \Cref{main_thm2} for some suitable parametric family are:
\begin{enumerate}[label=(\Alph*)]
	\item \label{i:hellinger}
	Parametric squared Hellinger distance:
	\[
	\ell_{H}(t,\vartheta)=H^2\bigl(\Psi(t),\Psi(\vartheta)\bigr)=\frac{1}{2}\int_{ \mathcal{X}}\left(\sqrt{p_t(x)}-\sqrt{p_\vartheta(x)}\right)^2\diff x;
	\]
	\item \label{i:KL}
	Parametric Kullback--Leibler divergence:
	\[
	\ell_{\mathrm{KL}}(t,\vartheta)=\mathrm{KL}\bigl(\Psi(\vartheta)\mid \Psi(t)\bigr)=\int_{\mathcal{X}}\log\!\left(\frac{p_\vartheta(x)}{p_t(x)}\right)\diff P_\vartheta(x);
	\]
	\item \label{i:2wass}
	Parametric squared 2-Wasserstein distance:
	\[
	\ell_{W_2}(t,\vartheta)=W_2^2\bigl(\Psi(t),\Psi(\vartheta)\bigr)=\inf_{\pi\in \Gamma(P_t,P_\vartheta)}\int_{ \mathcal{X}^2}\norm{x-y}^2\diff \pi(x,y),
	\]
	where 
	\begin{multline}\label{e:plans}
	\Gamma(P,Q)=\bigl\{\pi\in \mathcal{P}(\mathcal{X}\times \mathcal{X}):P(B)=\pi(B\times \mathcal{X}), \\Q(B)=\pi(\mathcal{X}\times B)\ \text{for all measurable }B\subseteq\mathcal{X}\bigr\}
	\end{multline}
	and $\mathcal{P}(\mathcal{X}\times \mathcal{X})$ is the set of all probability measures on $\mathcal{X}\times \mathcal{X}$ with finite second moment.
\end{enumerate}
It is shown in \citet[Lemma~4.1]{intrinsic} that \ref{i:hellinger} and \ref{i:KL} behave locally like a quadratic function.

\subsubsection{2-Wasserstein distance}
Since the behavior of intrinsic loss functions depends on the underlying parametric family, we give an example of a parametric family for which $\ell_{W_2}$ in \ref{i:2wass} satisfies the conditions of \Cref{main_thm2} beyond the location–scale setting, in which case the loss function reduces to the Euclidean distance.

\begin{example}
	Consider the family of Pareto distributions with unknown shape parameter, fixed scale parameter equal to $1$, and parameter space $\Theta=[a,\infty)$ for some fixed $a>2$. Then
	\[
	\ell_{W_2}(t,\vartheta)=\frac{2(t-\vartheta)^2}{(t\vartheta-t-\vartheta)(t-2)(\vartheta-2)}=\frac{(t-\vartheta)^2A(\vartheta)}{2}+\xi(t,\vartheta),
	\]
	where
	\[
	A(\vartheta)=\frac{4}{\vartheta(\vartheta-2)^3}\quad\text{and}\quad \xi(t,\vartheta)=2(t-\vartheta)^2\cdot \frac{\vartheta(\vartheta-2)^2-(t-2)(t\vartheta-t-\vartheta)}{(t-2)(\vartheta-2)^3\vartheta(t\vartheta-t-\vartheta)}.
	\]
	For any $t,\vartheta\in \Theta$, we have \(0\le \abs{A(\vartheta)}\le 4/(a(a-2)^3)\), and for some constants $c_1,c_2>0$,
	\[
	\abs{\xi(t,\vartheta)}\le2(t-\vartheta)^2 \frac{\abs{\vartheta(\vartheta-2)^2-(t-2)(t\vartheta-t-\vartheta)}}{a^2(a-2)^5}
		\lesssim \exp(c_1t+c_2\vartheta).
	\]
	Moreover, we have
	\[
	\nabla_t\ell_{W_2}(t,\vartheta)=\frac{2(t-\vartheta)\bigl((2\vartheta-1)t-3\vartheta\bigr)}{(t-2)^2\bigl((\vartheta-1)t-\vartheta\bigr)^2}
	\begin{cases}
		>0,&\text{if }t>\vartheta,\\
		<0,&\text{if }t<\vartheta,\\
		=0,&\text{if }t=\vartheta,
	\end{cases}
	\]
	and thus the loss function has a well-separated minimum, which together with  
 the triangle inequality implies \Cref{mon}.  
	Thus, the conditions of \Cref{main_thm2} are satisfied.
\end{example}

More general differentiability conditions for loss functions induced by the 2-Wasserstein distance, together with concrete examples,  are given in Supplement~\ref{diff_ws}.

\subsubsection{Sinkhorn divergence} 
Entropy-regularized Wasserstein distances are widely used due to their computational efficiency {(see e.g.~\citealp{Cuturi2019}).}  
We give a concrete example to show that such loss functions also fit within our framework.  
For a regularization parameter $\lambda>0$ and a parametric family $\{P_\vartheta:\vartheta \in \Theta\}$ on $\mathcal{X}$ with finite second moments, define the Sinkhorn divergence $S_\lambda$ with squared Euclidean cost between $P_t$ and $P_\vartheta$ as
\[
S_\lambda\!\left(P_t,P_\vartheta\right)= \inf_{\pi \in \Gamma(P_t,P_\vartheta)}\int_{\mathcal{X}^2} \norm{x-y}^2\diff\pi(x,y)+ \lambda\, \mathrm{KL}\!\left(\pi \mid P_t\otimes P_\vartheta\right),
\] 
where $\Gamma(P_t,P_\vartheta)$ is given in \eqref{e:plans}.  
The Sinkhorn divergence is not a loss function in our sense, since in general $S_\lambda(P_t,P_t)\neq 0$.  
This can be corrected by defining the {\emph{centered}} Sinkhorn divergence
\[
d_{S,\lambda}\!\left(P_t,P_\vartheta\right)=S_\lambda\!\left(P_t,P_\vartheta\right)-\tfrac{1}{2}S_\lambda\!\left(P_t,P_t\right)-\tfrac{1}{2}S_\lambda\!\left(P_\vartheta,P_\vartheta\right).
\]
We consider the loss function $\ell(t,\vartheta):=d_{S,\lambda}\!\left(P_t,P_\vartheta\right)$.  
If the model is identifiable (i.e., $P_t=P_\vartheta$ implies $t=\vartheta$), then $\ell(t,\vartheta)\ge 0$ with equality if and only if $t=\vartheta$, see \cite{lavenant2024}.  

Let $\mathcal{A}=\{(\sigma_1,\sigma_2,\sigma_3)\in (0,\infty)^3:\sigma_1\sigma_3 > \sigma_2^2\}$, and define the bivariate normal family 
\[
\left\{P_{(\mu,\sigma_1,\sigma_2,\sigma_3)}=\mathcal{N}_2\!\left(\mu,\begin{bmatrix}
    \sigma_1 & \sigma_2\\
    \sigma_2 & \sigma_3 
\end{bmatrix} \right):(\mu,\sigma_1,\sigma_2,\sigma_3)\in \Theta=\mathbb{R}^2\times \mathcal{A} \right\}.
\] 
Let $P_1$ and $P_2$ be two such Gaussians with means $\mu$ and $\nu$ and covariance matrices
\[
\Sigma_1=\begin{bmatrix}
    \sigma_1 & \sigma_2\\
    \sigma_2 & \sigma_3 
\end{bmatrix}\quad\text{and}\quad
\Sigma_2=\begin{bmatrix}
    \tau_1 & \tau_2\\
    \tau_2 & \tau_3 
\end{bmatrix}.
\]
Define
\(
f_\lambda(\Sigma_1,\Sigma_2)= \Sigma_1^{-1/2}\left(\Sigma_1^{1/2}\Sigma_2 \Sigma_1^{1/2}+\left(\tfrac{\lambda}{4}\right)^2\mathcal{I}_2 \right)^{1/2} \Sigma_1^{-1/2}-\tfrac{\lambda}{4}\Sigma_1^{-1}.
\)
It is shown in \cite{delbarrio2020} that $S_\lambda$ admits the closed form
\begin{multline*}
    S_\lambda(P_1,P_2)=\norm{\mu-\nu}^2+\tr\!\left(\Sigma_1\right) +\tr\!\left(\Sigma_2\right) \\ -2\tr\!\left(\Sigma_1f_\lambda\left(\Sigma_1,\Sigma_2\right)\right)-\tfrac{\lambda}{2}\log \!\left((2\pi e)^4\tfrac{\lambda^2}{4}\det\!\left(\Sigma_1f_\lambda\left(\Sigma_1,\Sigma_2\right)\right)\right). 
\end{multline*}
From this, the loss $\ell((\mu,\sigma),(\nu,\tau))$ can be expressed explicitly, and one can see that it is quadratic in $\mu$ and thrice continuously differentiable in $\sigma$, with a unique minimum at $(\mu,\sigma)=(\nu,\tau)$.  
Therefore,  the loss $\ell$ satisfies the assumptions of \Cref{main_thm2}.

\subsection{Stein discrepancies}\label{ss:stein}

An additional example is \emph{Stein's loss} for a variance parameter, considered in \cite{Zhang2022}, with parameter space $\Theta=(0,\infty)$, in form of
\(
\ell(t,\vartheta)=\vartheta/t-\log\!\bigl(\vartheta/t\bigr)-1,\; t,\vartheta>0.
\)
This loss behaves locally like a quadratic function.  
In contrast to the squared error loss, which does not penalize underestimation as strongly since it remains bounded as $t\to 0$, Stein's loss diverges to infinity both as $t\to 0$ and as $t\to\infty$.

In the broader context of Stein's method, \emph{Stein discrepancies} provide a means of expressing integral probability metrics on the space of probability measures on $\mathcal{X}$ as 
\begin{equation}\label{e:stein}
d_{\mathcal{H}}(P,Q)=\sup_{h\in \mathcal{H}} \abs{ \int_{\mathcal{X}} h(x)\,\diff P(x)-\int_{\mathcal{X}} h(x)\,\diff Q(x)},
\end{equation}
where $\mathcal{H}\subset \{h:\mathcal{X}\to \mathbb{R}:h \ \text{measurable}\}$ is called a \emph{Stein class}.  
Loss functions induced by these metrics are natural candidates for satisfying the conditions of \Cref{consistency}.

\subsubsection{Total variation distance}
A special case of an integral probability metric in \eqref{e:stein} is the {total variation distance}, obtained with $\mathcal{H}=\{\mathbb{I}_A:A\in \mathcal{B}(\mathcal{X})\}$.  
In \Cref{loc_q}, we considered loss functions satisfying a local quadratic expansion and showed that condition~\ref{cor1a3} in \Cref{corollary1} holds for a quadratic loss $\ell_0$.  
We now present an example of a loss with a different local behavior, where the corresponding $\ell_0$ is not quadratic.

\begin{example}[$L^1$-differentiable family]\label[example]{L1}
Let $\{P_\vartheta:\vartheta\in \Theta\subseteq \mathbb{R}^d\}$ be a parametric family with densities $p_\vartheta=\diff P_\vartheta/\diff \mu$ with respect to a $\sigma$-finite measure $\mu$ on $\mathcal{X}$.  
Assume that the family is $L^1(\mu)$-differentiable, i.e.\ there exists $\dot p_\vartheta:\mathcal{X}\to \mathbb{R}^d$ such that
\[
\norm{p_{\vartheta+r}-p_\vartheta -\langle r,\dot p_\vartheta\rangle }_{L^1}
=\int_{\mathcal{X}}\abs{p_{\vartheta+r}(x)-p_\vartheta(x) -\langle r,\dot p_\vartheta(x)\rangle}\diff \mu(x)
=o(\norm{r}),
\]
as $r\to 0$ in $\mathbb{R}^d$.   
For the loss $\ell(t,\vartheta)=2\norm{P_t-P_\vartheta}_{\TV}=\norm{p_t-p_\vartheta}_{L^1}$, we then have
\(
\ell(t,\vartheta)=\norm{\langle t-\vartheta,\dot p_\vartheta\rangle}_{L^1}+R(t,\vartheta),
\)
where $R(t,\vartheta)=o(\norm{t-\vartheta})$ as $\norm{t-\vartheta}\to 0$.

Differentiability in quadratic mean, assumed in the BvM theorem (\Cref{B.v.M:}), implies $L^1$-differentiability, and moreover, $\dot p_\vartheta(x)=p_\vartheta(x)\ell'_\vartheta(x)$, where $\ell'_\vartheta$ is the quadratic mean derivative. Assume that $\vartheta \mapsto \dot p_\vartheta(x)$ is continuous at $\vartheta_0$ for all $x\in \mathcal{X}$, that 
\[
\int_{\mathcal{X}} \sup_{\vartheta \in B_\delta(\vartheta_0)}\norm{\dot p_\vartheta(x)}\,\diff \mu(x)<\infty
\]
for some $\delta>0$, and that for any compact $K\subset \mathbb{R}^d$ and $h\in \mathbb{R}^d$,
\[
\sup_{t\in K} \varepsilon^{-1} \abs{R(\varepsilon t +\vartheta_0,\varepsilon h+\vartheta_0)}\to 0
\quad\text{as }\varepsilon \to 0.
\]
A sufficient condition for the last property is that, for some $c>1$ and all $t,\vartheta \in B_\delta(\vartheta_0)$,
\[
R(t,\vartheta)\lesssim \norm{t-\vartheta}^c.
\]
Then condition~\ref{cor1a3} in Proposition~\ref{corollary1} holds with
\(
\ell_0(t,h)=\norm{\langle t-h,\dot p_{\vartheta_0}\rangle}_{L^1} \text{ and } p=1,
\)
{\rev
see Supplement~\ref{additional_details_ex3}.} Thus, if the other assumptions of \Cref{main_thm1} hold, the Bayes estimator under total variation loss satisfies
\[
n^{1/2}(\hat{\theta}_n-\vartheta_0) \stackrel{D}{\longrightarrow} \arg \min_{t \in \mathbb{R}^d}
\int_{\mathbb{R}^d }\norm{\langle t-h,\dot p_{\vartheta_0}\rangle}_{L^1}
\,\diff\mathcal{N}\!\left(Y,I_{\vartheta_0}^{-1}\right)(h),
\]
where $Y\sim\mathcal{N}_d\!\left(0,I_{\vartheta_0}^{-1}\right)$.
\end{example}

We next verify the conditions of \Cref{L1} for the multivariate normal location model, in which the induced loss behaves locally like the Euclidean distance and satisfies condition~\ref{Thm4.1a1} in \Cref{main_thm1} with $p=1$.

\begin{example}[Multivariate normal location model]\label[example]{b:norm}
Consider $\{P_\vartheta=\mathcal{N}_d(\vartheta,\Sigma): \vartheta\in \mathbb{R}^d\}$ with fixed $\Sigma\in \mathbb{R}^{d\times d}$.  
By Pinsker's inequality, for $t,\vartheta\in\mathbb{R}^d$,
\[
\norm{P_t-P_\vartheta}_{\TV}
\le \left(\frac{1}{2}\,\mathrm{KL}(P_t\mid P_\vartheta)\right)^{1/2}
= \left(\frac{1}{4}(t-\vartheta)^\top\Sigma^{-1}(t-\vartheta)\right)^{1/2}
\le  \frac{1}{2}\norm{\Sigma^{-1}}_{\F}^{1/2}\,\norm{t-\vartheta}.
\]
Fix $\vartheta_0 \in \mathbb{R}^d$ and choose $\delta>0$ such that 
\(
\frac{1}{2}\norm{\Sigma^{-1}}_{\F}^{1/2}\, \norm{t-\vartheta}
\le \delta  \norm{\Sigma^{-1}}_{\F}^{1/2} \le  \frac{1}{600},
\)
for all $t,\vartheta\in B_\delta(\vartheta_0)$.
Then, \citet[Theorem~1.8]{arbas2023} implies
\[
\frac{1}{200}\norm{t-\vartheta}
\ \le\ \norm{P_t-P_\vartheta}_{\TV} \ \le\ \frac{1}{2^{1/2}} \norm{t-\vartheta},
\quad t,\vartheta \in B_\delta(\vartheta_0).
\]
For simplicity, take $\Sigma={I}_d$ the identity matrix.  
The $L^1$-derivative of the density $p_\vartheta$ with respect to the Lebesgue measure $\lambda$ is
\(
\dot p_\vartheta(x) = (x-\vartheta)p_\vartheta(x),
\)
which follows from the mean value theorem and the dominated convergence theorem.  
The total variation loss admits the expansion
\(
\ell(t,\vartheta)=\norm{p_t-p_\vartheta}_{L^1}
=\norm{\langle t-\vartheta,\dot p_\vartheta\rangle}_{L^1}+R(t,\vartheta),
\)
where $\vartheta\mapsto \dot p_\vartheta(x)$ is continuous at $\vartheta_0$ for all $x$, and
\[
\int \sup_{\vartheta \in B_\delta(\vartheta_0)}\norm{\dot p_\vartheta(x)}\,\diff x \le \int \sup_{\vartheta \in B_\delta(\vartheta_0)}\norm{x-\vartheta} p_\vartheta(x)\diff x <\infty.
\]
In this case, we have
\begin{align*}
\ell_0(t,h)
&=\norm{\langle t-h,\dot p_{\vartheta_0}\rangle}_{L^1}
= \int \abs{\langle t-h, x-\vartheta_0\rangle} p_{\vartheta_0}(x)\,\diff x \\
&= \mathbb{E}_{Z\sim \mathcal{N}_d(0,\mathcal{I}_d)}\left[\abs{\langle t-h,Z\rangle}\right]
= \left(\frac{2}{\pi}\right)^{1/2}\,\norm{t-h}.
\end{align*}

A direct verification (see Supplement~\ref{ss:prf:s5}), using $L^1$-differentiability and uniform directional differentiability of the Gaussian density, shows that for any compact $K\subset \mathbb{R}^d$ and $h\in \mathbb{R}^d$,
\[
\sup_{t\in K} \abs{\varepsilon^{-1}\ell(\varepsilon t +\vartheta_0,\varepsilon h+\vartheta_0)-\ell_0(t,h)}
\ \to\ 0, \quad \varepsilon \to 0.
\]
Since $t \mapsto \int_{\mathbb{R}^d }(2\pi^{-1})^{1/2}\norm{t-h}\,\diff\mathcal{N}_d(Y,I_{\vartheta_0}^{-1})(h)$ is minimized at $t=Y$, we also obtain asymptotic efficiency for the Bayes estimator under total variation loss, i.e.
\(
n^{1/2}(\hat{\theta}_n-\vartheta_0) \stackrel{D}{\to} \mathcal{N}_d\!\bigl(0,I_{\vartheta_0}^{-1}\bigr).
\)
\end{example}

\subsubsection{1-Wasserstein distance}
For the particular Stein class $\mathcal{H}=\{f:\mathcal{X}\to \mathbb{R}:\abs{f(x)-f(y)}\le \norm{x-y} \text{ for any } x,y \}$ with compact $\mathcal{X}$, the induced integral probability metric in \eqref{e:stein} is the $1$-Wasserstein distance \cite[Theorem~1.14]{villani1}.  
If $\mathcal{X}$ is one-dimensional (not necessarily compact), the distance admits the distribution function representation \cite[Theorem~2.18]{villani1}:
For $P,Q$ with finite first absolute moments and distribution functions $F_P,F_Q$, it holds that 
\(
W_1(P,Q)=\int_{\mathbb{R}}\abs{F_P(x)-F_Q(x)}\,\diff x
= \norm{F_P-F_Q}_{L^1}.
\)

Let $\{P_\vartheta:\vartheta\in \Theta\subseteq \mathbb{R}^d\}$ be a parametric family with finite first absolute moments, and suppose that $F_\vartheta$ is $L^1$-differentiable at $\vartheta_0$ with derivative $\dot F_{\vartheta_0}:\mathcal{X}\to \mathbb{R}^d$, i.e.,
\[
\int_{\mathcal{X}} \abs{F_{\vartheta}(x)-F_{\vartheta_0}(x)-\langle\vartheta-\vartheta_0,\dot F_{\vartheta_0}(x)\rangle}\,\diff x
= o(\norm{\vartheta-\vartheta_0}),\quad \vartheta\to \vartheta_0.
\]
If the family is differentiable in quadratic mean, then under mild conditions $F_\vartheta$ is $L^1$-differentiable with
\(
\dot F_{\vartheta}(t)
= \int_{-\infty}^t \ell'_{\vartheta}(x)\,\diff P_{\vartheta}(x),
\)
where $\ell'_{\vartheta}$ is the quadratic mean derivative (\Cref{L1_diff_cdf} in Supplement~\ref{ss:lp:diff}).

For the loss $\ell(t,\vartheta)=W_1(P_t,P_\vartheta)$, its $L^1$-differentiability yields
\[
\ell(t,\vartheta)
= \|{\langle t-\vartheta,\dot F_{\vartheta}\rangle}\|_{L^1} + R(t,\vartheta),
\quad \abs{R(t,\vartheta)}=o(\norm{t-\vartheta}),
\]
and, as in \Cref{L1}, $\ell_0(t,h)=\|{\langle t-h,\dot F_{\vartheta_0}\rangle}\|_{L^1}$.
{The probability measure associated with the Bayes estimator, i.e.\ $P_{\hat{\theta}_n}$ is a population Wasserstein median over the space $\{P_\vartheta:\vartheta \in \Theta\}$ (see e.g.\ \citealp{CCE24}).}

\begin{example}[Gompertz model]\label{Ws1_ex}
Consider the Gompertz distribution $P_{\vartheta}$ with shape parameter $\vartheta>0$, which has a Lebesgue density
\(
p_\vartheta(x)=\vartheta\exp\!\left(\vartheta+x-\vartheta e^x\right),\; x>0.
\)
This distribution is often used in actuarial science to describe mortality processes of organisms \citep{gompertz}. The $1$-Wasserstein distance between Gompertz distributions of shape parameters $t,\vartheta>0$ is
\begin{align*}
\ell(t,\vartheta)
&= W_1(P_t,P_{\vartheta}) = \int_0^\infty \abs{e^{-t(e^x-1)}-e^{-\vartheta(e^x-1)}}\,\diff x \\
&= \mathbb{I}(t>\vartheta)\left(E_1(\vartheta)e^\vartheta- E_1(t)e^t\right)
    -\mathbb{I}(t\le\vartheta)\left(E_1(\vartheta)e^\vartheta- E_1(t)e^t\right),
\end{align*}
where 
\(
E_1(s)=\int_s^\infty u^{-1}{e^{-u}}\,\diff u.
\)
The function $t\mapsto\ell(t,\vartheta)$ is differentiable for $t\neq\vartheta$ but not at $t=\vartheta$.  
Since the density is continuously differentiable and the Fisher information is continuous, the model is differentiable in quadratic mean \cite[Lemma~7.6]{v}, and hence $F_\vartheta$ is $L^1$-differentiable (\Cref{L1_diff_cdf}  in Supplement~\ref{ss:lp:diff}). Thus,
\(
\ell(t,\vartheta)
= \abs{t-\vartheta} \int_0^\infty \abs{(e^x-1)e^{-\vartheta(e^x-1)}}\,\diff x
    + R(t,\vartheta), \text{ with } \abs{R(t,\vartheta)}=o(\abs{t-\vartheta}).
\)
\Cref{figw1} presents a contour plot of $(a,b)\mapsto W_1(P_a,P_b)$ and, for comparison, that of $(a,b)\mapsto \abs{a-b}$. As shown, the loss is not a function of $t-\vartheta$, and therefore falls outside standard quadratic approximations in the literature, but it is covered by our general framework.
\begin{figure}
    \centering
\includegraphics[width=0.4\textwidth, trim=10mm 13mm 9mm 18mm, clip]{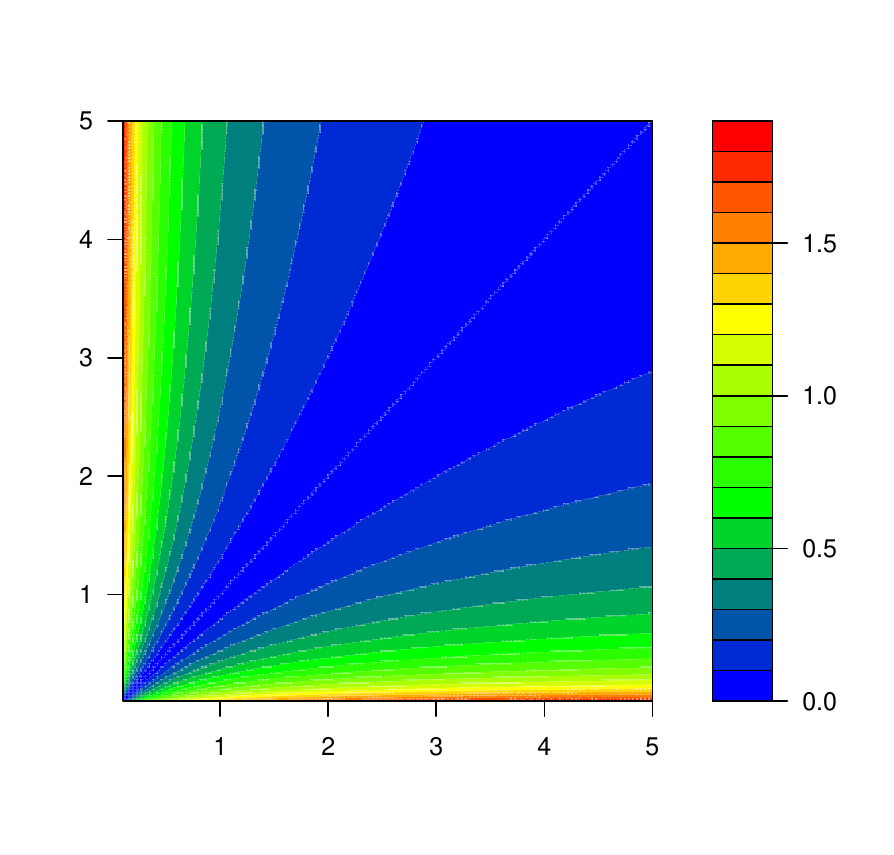}\qquad
\includegraphics[width=0.4\textwidth, trim=10mm 13mm 9mm 18mm, clip]{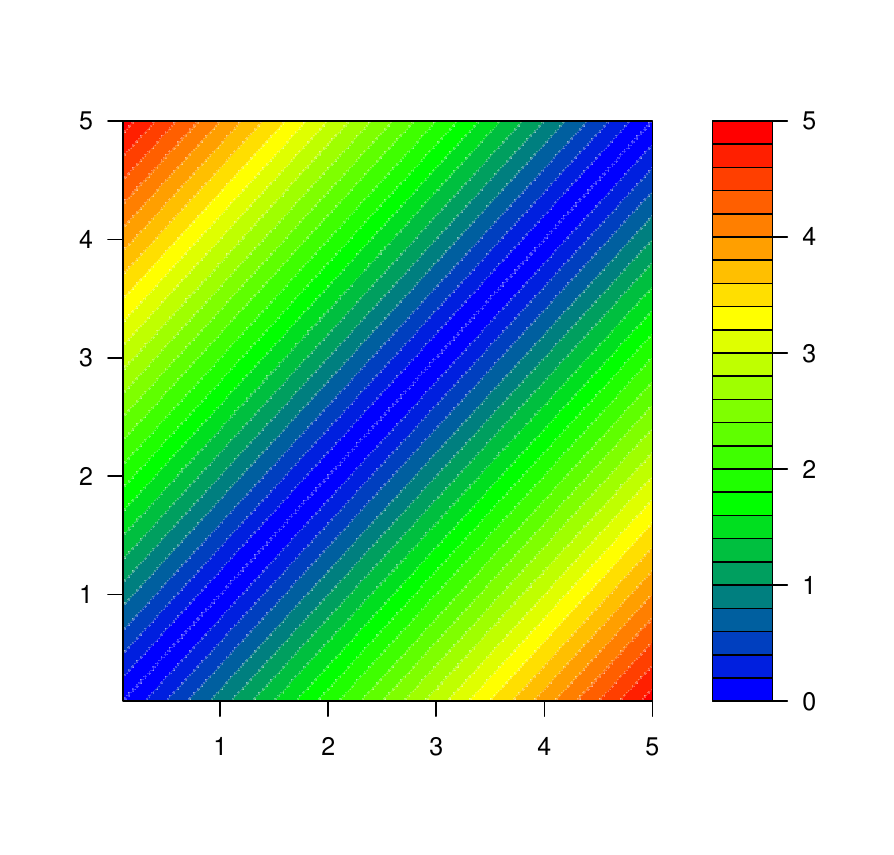}\\
    \caption{Contour plots of maps $(a,b)\mapsto W_1(P_a,P_b)$ (left) and $(a,b)\mapsto \abs{a-b}$ (right).    \label{figw1}}
\end{figure}

\end{example}

\subsection{Beyond our framework}\label{ss:beyond}
We now present an example that lies outside our framework.  
Let $\mathcal{X}=\{y_1,\dots,y_d\}$, where $y_j$ denotes the $j$th unit vector in $\mathbb{R}^d$, and consider the parametric family of multinomial distributions
\(
\{P_p = \mathrm{Mult}(1,p): p \in [0,1]^d, \ \|p\|_1 = 1\}
\)
on $\mathcal{X}$. We study an empirical Bayes model given by
\begin{align*}
    X_1,\dots,X_n \mid \rho=p \;\stackrel{\text{i.i.d.}}{\sim}\; P_p,\qquad
    \rho \;\sim\; \Pi_n,
\end{align*}
where the data-dependent prior satisfies
\[
\Pi_n(\{p\}) \;=\; \frac{1}{\sum_{i=1}^n w_i}\sum_{i=1}^n w_i \, \delta_{X_i}(p), 
\quad w_i = w_i(X_1,\ldots,X_n) = \prod_{j=1}^d X_{i,j}^{-\sum_{m=1}^n \mathbb{I}(X_m=y_j)}.
\]
It follows that the posterior is simply the empirical measure of the data. Moreover, the Bayes estimator under the squared 2-Wasserstein loss $\ell(t,\vartheta) = W_2^2(P_t,P_\vartheta)$ is
\[
\mathop{\arg \min}_{t \in [0,1]^d,\ \|t\|_1 = 1} \ \frac{1}{n}\sum_{i=1}^n W_2^2(t, X_i).
\]
That is, the Bayes estimator coincides with the empirical Wasserstein barycenter of $X_1,\dots,X_n$, where each $X_i$ is interpreted as a probability mass vector, if we restrict the support of all involved probability measures on $\mathcal{X}$.

{Here}, we cannot employ our framework, since the empirical prior $\Pi_n$ is not absolutely continuous with respect to the Lebesgue measure (so that the BvM theorem may fail). Nevertheless, this barycenter can be formulated as the solution of a finite-dimensional linear program, whose standard form (see Supplement~\ref{sss:beyond}) is
\[
    \min_{x \geq 0} c^\top x, 
    \quad \text{subject to } Ax = b,
\]
where $b = \mathbf{1}_{d}^\top \in \mathbb{R}^d$,  
\(
c = \bigl(1-N_1,\ \ldots,\ 1-N_{d-1},\ 0,\ldots,0\bigr)^\top \in \mathbb{R}^{2(d-1)+1},
\)
and
\[
A = \begin{pmatrix}
    \mathcal{I}_{d-1} & \mathcal{I}_{d-1} & \mathbf{0}_{d-1} \\
    \mathbf{1}_{d-1}^\top & \mathbf{0}_{d-1}^\top & 1
\end{pmatrix} \in \mathbb{R}^{d \times (2(d-1)+1)},
\]
with $N_j = n^{-1}\sum_{i=1}^n \mathbb{I}(X_i=y_j)$, $\mathbf{0}_d$ denoting the zero vector in $\mathbb{R}^d$, $\mathbf{1}_d$ the all-ones vector in $\mathbb{R}^d$, and $\mathcal{I}_d$ the $d \times d$ identity matrix.

The barycenter is then expressed as $t = (x_1,\ldots,x_{d-1},\, 1 - x_1 - \cdots - x_{d-1})$. Its limiting distribution (which is non-Gaussian, in general) may be studied via arguments analogous to those in \citet{KMZ22}, who {provided limit laws for linear programs with} random $b$ in contrast to the random $c$ encountered here. A detailed investigation of this problem is an intriguing direction for future research, but lies beyond the scope of this paper.

\section{Numerical experiments}\label{sec6}

In this section, we illustrate our theoretical convergence results through numerical simulations.  
We focus on the three intrinsic losses $\ell_{H},\ell_{\mathrm{KL}},\ell_{W_2}$ discussed in \Cref{ss:quad:loss}.

\subsection{Exponential--Gamma model} 
Consider the family of exponential distributions $P_t=\mathrm{Exp}(t)$ for $t>0$ with prior $\Pi=\mathrm{Gamma}(a,b)$.  
Then the posterior is 
\(
\theta \mid {\boldsymbol X}=(x_1,\ldots,x_n)  \sim \mathrm{Gamma}\!\left(a+n,\ b+\sum_{i=1}^n x_i\right).
\)
The three losses take the form
\[
\ell_{H}(t,\vartheta)=1-\frac{2(t\vartheta)^{1/2}}{t+\vartheta},\quad
\ell_{\mathrm{KL}}(t,\vartheta)=\log(\vartheta)-\log(t)+\frac{t}{\vartheta}-1, \quad
\ell_{W_2}(t,\vartheta)=2(t^{-1}-\vartheta^{-1})^2.
\]
In the simulation, we set $a=b=2$ and $\vartheta_0=2$.  \Cref{fig1} displays QQ-plots of the rescaled and centered estimators $(I_{\vartheta_0}n)^{1/2}(\hat{\theta}_n-\vartheta_{0})$ against the standard normal distribution.  As $n$ increases, the sample quantiles approach the theoretical normal quantiles, in agreement with \Cref{main_thm2}.

%\old
\iffalse
\begin{figure}[ht]
    \centering
    \includegraphics[width=.9\linewidth, trim={0 0.2cm 0.5cm 1cm}, clip]{qq_plots_hl.pdf}\\
    \includegraphics[width=.9\linewidth, trim={0 0.2cm 0.5cm 1cm}, clip]{qq_plots_ws.pdf}\\
    \includegraphics[width=.9\linewidth, trim={0 0.2cm 0.5cm 1cm}, clip]{qq_plots_kl.pdf}
    \caption{QQ-plots for $n^{1/2}(\hat{\theta}_n-\vartheta_{0})$ for $n\in\{10,10^2,10^4,10^6\}$ (left to right) based on $M=500$ Monte Carlo repetitions, where $\vartheta_{0}=2$ and  $\hat{\theta}_n$ is the Bayes estimator under losses $\ell_{H}$,  $\ell_{W_2}$ and   $\ell_{\mathrm{KL}}$ (top to bottom). The reference  corresponds to the standard normal distribution.}
    \label{fig1}
\end{figure}
\fi

%new
\begin{figure}[t]
    \centering
    \includegraphics[width=.95\linewidth]{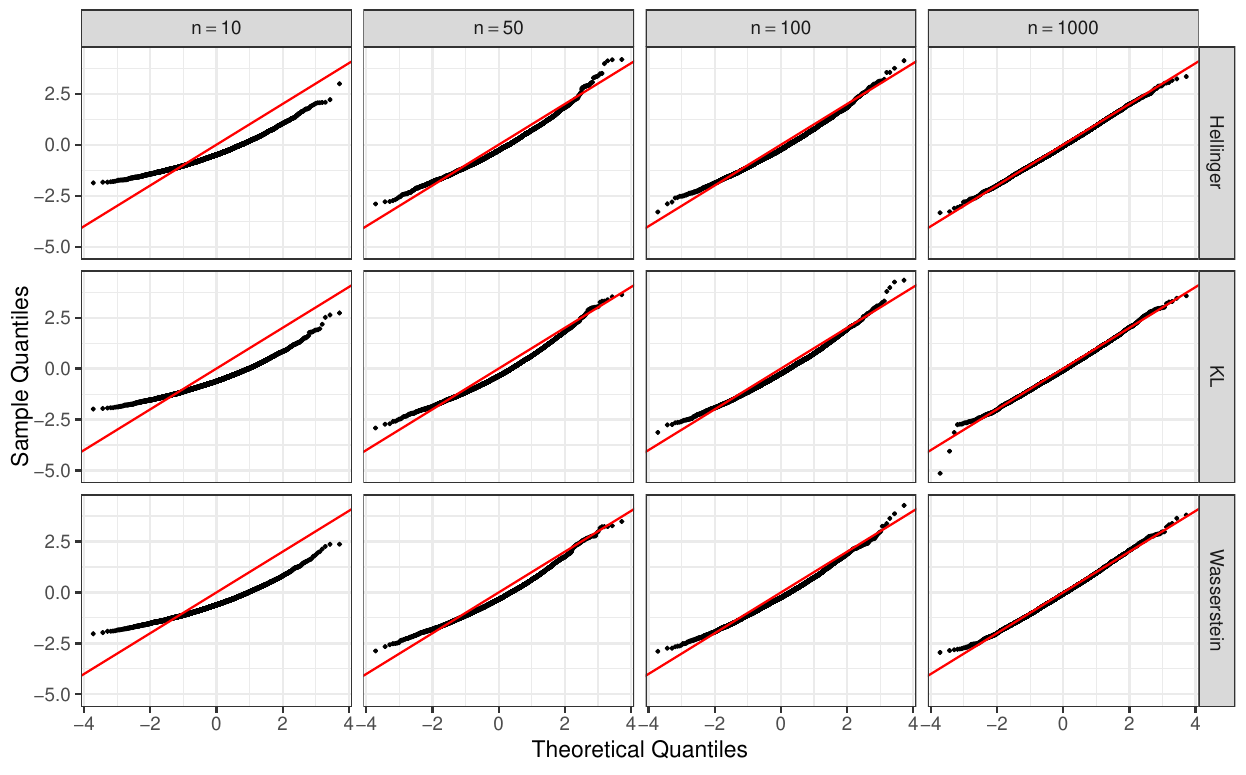}
    \caption{QQ-plots for  $(nI_{\vartheta_0})^{1/2}(\hat{\theta}_n-\vartheta_{0})$ {\rev for $n\in\{10,50,100,1000\}$ (left to right), based on $M=5000$} Monte Carlo repetitions under the Exponential--Gamma model. Here  $\hat{\theta}_n$ is the Bayes estimator under the Hellinger loss $\ell_{H}$, the  Kullback--Leibler (KL) loss $\ell_{\mathrm{KL}}$ and the 2-Wasserstein loss $\ell_{W_2}$ (top to bottom). The reference distribution (shown via theoretical quantiles) is the standard normal distribution. \label{fig1}}
\end{figure}

\subsection{Multinomial--Dirichlet model}\label{Mult_example}
Consider the multinomial model $\{P_p=\mathrm{Mult}(1,p):p\in[0,1]^d,\ \norm{p}_1=1\}$ on $\mathcal{X}=\{y_1,\dots,y_d\}$ with $y_j$ the $j$th unit vector in $\mathbb{R}^d$, as in \cref{ss:beyond}. Now we employ a Dirichlet prior $\mathrm{Dir}(\alpha)$, and then the posterior is $\mathrm{Dir}(\alpha+\sum_{i=1}^n x_i)$, given ${\boldsymbol X}=(x_1,\ldots,x_n)$.  

The squared 2-Wasserstein distance between two multinomial measures $P_p$ and $P_\vartheta$ can be expressed in terms of their coordinates (Lemma~\ref{lem:W2-mult} in Supplement~\ref{ss:prf:s6}).  
In particular,
\(
W_2^2(P_p,P_\vartheta) \ =\ 1 - \sum_{i=1}^d \min(p_i,\vartheta_i).
\)
Since $\{p\in[0,1]^d:\norm{p}_1=1\}$ has empty interior in $\mathbb{R}^d$, we reparametrize as
\(
\Theta = \left\{(p_1,\dots,p_{d-1})\in [0,1]^{d-1}:\sum_{i=1}^{d-1}p_i\le 1\right\},
\)
with \(\tilde{P}_a =  P_{(a,\,1-\norm{a}_1)}\) for $a \in \Theta$. The loss becomes
\[
\ell(a,b) = W_2^2(\tilde{P}_a,\tilde{P}_b)= \sum_{i=1}^{d-1}\bigl(a_i+b_i-\min\{a_i,b_i\}\bigr)
= \norm{a-b}_1,\; a,b\in\Theta,
\]
i.e., the 2-Wasserstein distance coincides with the total variation distance. The inverse of the Fisher information matrix at $\vartheta_0=(p_1,\dots,p_{d-1})\in \Theta^\circ$ is given by
 \[
 I_{\vartheta_0}^{-1}=\begin{pmatrix}
     p_1(1-p_1)&-p_1 p_2&\dots &-p_1p_{d-1}\\
     -p_2 p_1 & p_2(1-p_2)&\dots&-p_2 p_{d-1}\\
         \vdots&\vdots &\vdots& \vdots \\
     -p_{d-1}p_1& -p_{d-1}p_2&\dots &p_{d-1}(1-p_{d-1})
 \end{pmatrix}\; \in\; \mathbb{R}^{(d-1)\times(d-1)}.
 \]
The limiting process is then
\(
Z(t) = \int_{\mathbb{R}^{d-1}}\norm{t-h}_1\,\diff\mathcal{N}_{d-1}\!\left(Y,I^{-1}_{\vartheta_0}\right)(h),
\)
which is minimized at $t=Y$ {and} the Bayes estimator is the vector of component wise marginal posterior medians.  
If ${\boldsymbol X}=x$, the $k$-th marginal of the posterior $\mathrm{Dir}(\alpha+\sum_{i=1}^n x_i)$ is $\mathrm{Beta}(\tilde{\alpha}_k,\sum_{j\neq k}\tilde{\alpha}_j)$, with $\tilde{\alpha}_k=\alpha_k+\sum_{i=1}^n x_{i,k}$, so the $k$-th component of $\hat{\theta}_n$ is the median of this Beta distribution.

In the simulation, we set $d=3$, $\alpha=(1,1,1)$, and $\vartheta_0=(1/3,1/3)$.  
Figure~\ref{fig:mult} shows the 2-Wasserstein distance between the empirical distribution of $n^{1/2}(\hat{\theta}_n-\vartheta_0)$ and $\mathcal{N}_{d-1}(0,I_{\vartheta_0}^{-1})$ as a function of $n$. {\rev For each $n$,  the 2-Wasserstein distance between the empirical distribution of the scaled Bayes estimator and the limiting normal distribution was computed using the function \texttt{wasserstein} in the \texttt{R} package \texttt{transport} \cite{Schuhmacher2024}. The computation was based on $10^4$ drawn from each distribution and averaged over 
$M=100$ Monte Carlo repetitions, following the approach of \citet{SSZM19}.} One can see that the distance rapidly converges to zero. Since the 2-Wasserstein distance metrizes weak converge, this implies that the Bayes estimator, after proper scaling, converges in distribution to the normal limit $\mathcal{N}_{d-1}\!\left(0,I^{-1}_{\vartheta_0}\right)$.

%old
\iffalse
\begin{figure}[!h]
    \centering
    \includegraphics[width=0.5\linewidth,trim={0.1cm 0.5cm 0.8cm 1.8cm}, clip]{multinom.pdf}
    \caption{2-Wasserstein distance between the empirical distribution of $n^{1/2}(\hat{\theta}_n-\vartheta_0)$ and $\mathcal{N}_{d-1}(0,I_{\vartheta_0}^{-1})$ for $d=3$, $\alpha=(1,1,1)$ and $\vartheta_0=(1/3,1/3)$. Each point is based on $R=100$ Monte Carlo repetitions of calculating the Wasserstein distance based on $M=2000$ samples.}
    \label{fig:mult}
\end{figure}
\fi 

\begin{figure}
    \centering
    \includegraphics[width=0.75\linewidth]{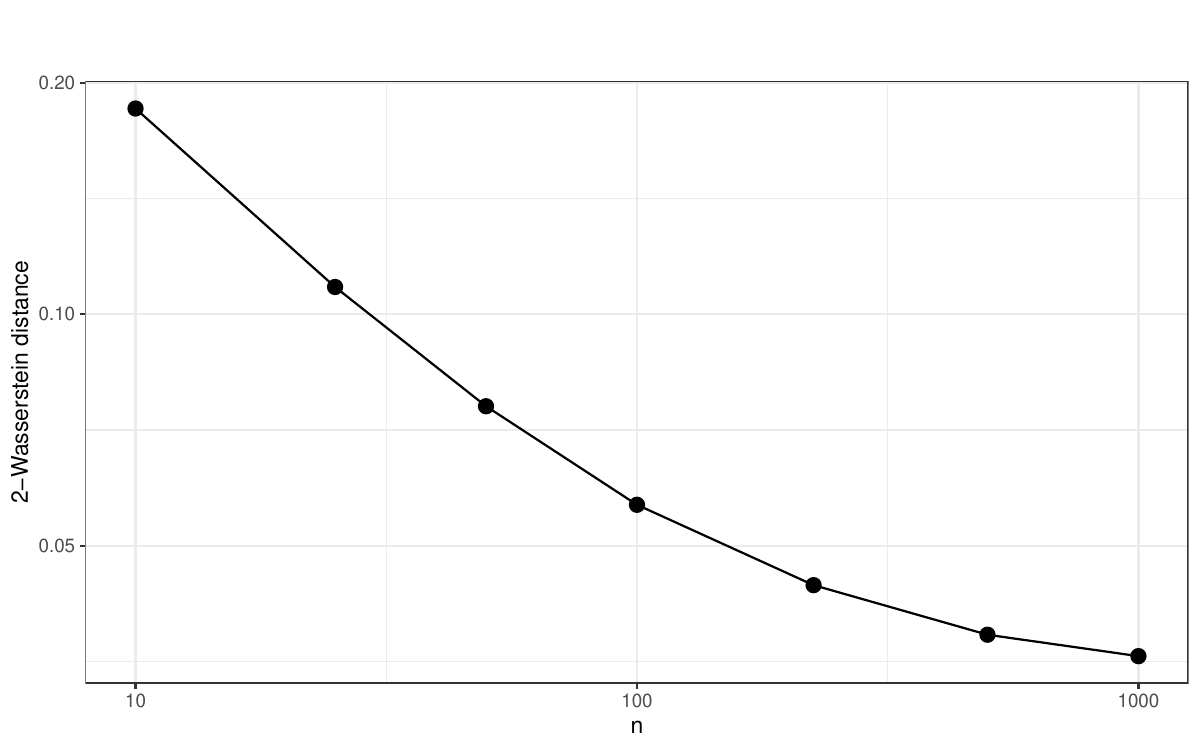}
    \caption{ {\rev 2-Wasserstein distance between the empirical distribution of $n^{1/2}(\hat{\theta}_n-\vartheta_0)$ and $\mathcal{N}_{d-1}(0,I_{\vartheta_0}^{-1})$ for $n \in \{10,25,50,100,225,500,1000\}$. Here $\hat{\theta}_n$ is the Bayes estimator, $d=3$, $\alpha=(1,1,1)$ and $\vartheta_0=(1/3,1/3)$. Each point represents the average over $M=100$ Monte Carlo repetitions, where the Wasserstein distance is computed from $10^4$ samples from each distribution. Both axes are shown on a logarithmic scale.} \label{fig:mult}}
\end{figure}

{\rev

\section{Discussion}\label[section]{s:extension}

In this paper, we study the weak convergence of the rescaled Bayes estimator
\(
\delta_n^{-1}(\hat\theta_n-\vartheta_0),
\)
with $\hat\theta_n$ defined through a general loss function $\ell(\cdot, \cdot)$, under the frequentist setting where observations
\(
X_1,\dots,X_n \stackrel{\text{i.i.d.}}{\sim} P_{\vartheta_0},
\)
for a fixed parameter $\vartheta_0 \in \Theta$, a subset of a metric vector space. The scaling $\delta_n$ is determined by the posterior contraction rate; in the regular parametric models considered here, $\delta_n=1/\sqrt{n}$. The proof strategy, rooted in the classical works of \citet{Lecam1953,chao,Ibragimov1973}, consists of two main steps:  
(a) establishing weak convergence of the posterior risk as a stochastic process indexed by the local parameter $t = \delta_n^{-1}(\vartheta-\vartheta_0)$, with $\vartheta$ ranging over a neighborhood of $\vartheta_0$; and  
(b) applying the argmax continuous mapping theorem (see e.g.~\cite[Theorem~5.56]{v} and \cite[Theorem~3.2.2]{vanwell}).

This framework is potentially applicable beyond the setting studied here and could provide a route to analyzing Bayes estimators in more general models. In particular, step~(a), namely the convergence of the posterior risk, depends crucially on an approximation to the posterior distribution, such as BvM-type results, or Laplace approximations. Such approximations have been established in a variety of settings, including non-regular models \citep{BoGr14,Ghosal1997}, misspecified models \citep{klein2012}, semiparametric models \citep{BiKl12,Cas12,Castillo2015}, nonparametric models \citep{CaNi14,ghosal_van_der_vaart_2017}, and high-dimensional models \citep{Kat25,Spo25}, see also \citep{Durante2024} for non-Gaussian approximations. To derive the convergence of posterior risk in step~(a) from the posterior approximations, and to establish uniform tightness (or an appropriate alternative) of $\delta_n^{-1}(\hat\theta_n-\vartheta_0)$, one additionally needs a strengthened form of posterior contraction, as in \Cref{weak_conv_prop1} (in Supplement~\ref{technical}). Thus, a systematic treatment of these extensions would require substantial additional work and lies beyond the scope of the present paper. We nevertheless briefly discuss how our analysis may extend to misspecified models.

Following \citet{klein2012}, suppose now that
\(
X_1,\dots,X_n \stackrel{\text{i.i.d.}}{\sim} P_0,
\)
where
\(
P_0 \notin \{P_\vartheta:\vartheta\in\Theta\subseteq\mathbb{R}^d\}.
\)
Define the pseudo-true parameter
\[
\vartheta_*=\arg\min_{\vartheta\in\Theta} \mathrm{KL}(P_0\mid P_\vartheta),
\]
where $\mathrm{KL}$ denotes the Kullback--Leibler divergence, and assume that this minimizer is unique. We further assume that the matrix
\[
V_{\vartheta_*}
=
\nabla_\vartheta^2 \mathrm{KL}(P_0\mid P_\vartheta)\big|_{\vartheta=\vartheta_*}
\]
is nonsingular. Under these conditions, Theorem~2.1 of \citet{klein2012}  yields a BvM-type result: the posterior distribution of $\sqrt{n}(\hat\theta_n-\vartheta_*)$ converges in total variation distance to a normal distribution with covariance matrix $V_{\vartheta_*}^{-1}$.

Moreover, we consider the loss function that is bounded by a polynomial, a stricter requirement than the exponential growth in \Cref{d:eg}. This may be viewed as a price of model misspecification. For such losses, we can extend \Cref{weak_conv_prop1} to this misspecified setup. As a consequence, we may establish, much as in \Cref{corollary1}, the convergence of the posterior risk in step~(a) to a nondegenerate limit process $Z(t)$ in $\ell^\infty(K)$ for every compact set $K\subseteq\mathbb{R}^d$, as well as the uniform tightness of $\sqrt{n}(\hat\theta_n-\vartheta_*)$, which is a key ingredient for step~(b). The limit process of posterior risks would take the form
\[
Z(t)=\int_{\mathbb{R}^d}\ell_0(t,h)\,\diff \mathcal{N}_d(Y,V^{-1}_{\vartheta_*})(h),
\]
where $\ell_0$ is a directional Hadamard derivative of the loss $\ell$, 
\(
Y\sim \mathcal{N}_d(0,J_{\vartheta_*})
\)
with
\[
J_{\vartheta_*}
=
V_{\vartheta_*}^{-1}
\left(
\int_{\mathcal{X}}
\ell_{\vartheta_*}'(x){\ell_{\vartheta_*}'(x)}^\intercal
\,\diff P_0(x)
\right)
V_{\vartheta_*}^{-1},
\]
and recall that $\ell_{\vartheta_*}'$ denotes the score function. This suggests that an extension of \Cref{main_thm1} to misspecified models should be feasible.  We leave a careful investigation of this extension, as well as of the other directions mentioned above, to future work.
}

\section*{Acknowledgement}
The authors would like to thank Johannes Schmidt-Hieber for helpful discussions. The authors gratefully acknowledge support from the Deutsche Forschungsgemeinschaft (DFG, German Research Foundation) Collaborative Research Center 1456. HL and AM further acknowledge funding from the DFG under Germany’s Excellence Strategy--EXC 2067/1-390729940. AM additionally acknowledges support from the DFG Research Unit 5381.

%\bibliographystyle{apalike}
%\bibliography{literature2}

\clearpage

\appendix

\section{Various types of differentiability}
In this supplement, we give an overview over different concepts of differentiability considered in this paper, and put them into perspective.

\subsection{Hadamard derivative}\label{ss:had}
Let $\phi: D\to E$ be a function between two Banach spaces $D$
and $E$. We say $\phi$ is $p$-th order \emph{Hadamard directionally differentiable} at $\vartheta_0\in D^\circ$ in direction $ h\in D$, if there exists a  mapping $ \phi^\prime_{\vartheta_0}:D \to E$ such that
\[
\frac{\phi(\vartheta_0+r_n h_n)-\phi(\vartheta_0)}{r_n^p}\to \phi^\prime_{\vartheta_0}(h) 
\]
for any sequences $r_n\to 0,r_n>0$ and $h_n\to h$ and for $p>0$. Here the derivative $\phi^\prime_{\vartheta_0}(\cdot)$ does not have to be linear, but is by definition continuous at $h$. If $p=1$ and $\phi^\prime_{\vartheta_0}(\cdot)$ is linear we say that $\phi$ is \emph{Hadamard differentiable} at $\vartheta_0\in D^\circ$ in direction $ h\in D$.  In \citet[Theorem~3.2] {Averbukh1968} it is shown that for any compact set $K\subset D$, the Hadamard differentiability at $\vartheta_0\in D^\circ$ in direction $ h$ for all $h \in K$ is equivalent to the \emph{compactly differentiability} at $\vartheta_0$, i.e.\ there exists a continuous, linear mapping $ \phi^\prime_{\vartheta_0}:D \to E$
with
\[
\sup_{h\in K}\norm{\frac{\phi(\vartheta_0+r h)-\phi(\vartheta_0)}{r}-\phi^\prime_{\vartheta_0}(h)  }_E \to 0,\quad \text{whenever }r\to 0.
\] Similarly, we say $\phi$ is $p$-th order \emph{compactly directionally differentiable} at $\vartheta_0$, if  there exists a mapping $ \phi^\prime_{\vartheta_0}:D \to E$ such that
\[ \sup_{h\in K}\norm{\frac{\phi(\vartheta_0+r h)-\phi(\vartheta_0)}{r^p}-\phi^\prime_{\vartheta_0}(h)  }_E \to 0,\quad \text{whenever }r\to 0,r>0.\]
Here derivative is not required to be linear. It can be shown that also the directional version of compact differentiability is equivalent to directionally Hadamard differentiability, if the derivative is assumed to be continuous.
\begin{proposition}[{\citealp[Proposition 3.3]{Shapiro1990}}]  Let $p>0$. The $p$-th order Hadamard directional differentiability of $\phi$ at $\vartheta_0$ in direction $h$ for all $h\in D$ implies $p$-th order compact directional differentiability of $\phi$ at $\vartheta_0$. Conversely, if $\phi$ is $p$-th order compactly directionally differentiability with derivative $\phi^\prime_{\vartheta_0}$, and if $\phi^\prime_{\vartheta_0}$ is sequentially continuous, then $\phi$ is
$p$-th order Hadamard directionally differentiable at $\vartheta_0$.
\end{proposition} The author only proves the case $p=1$, but the general case follows analogous. 

Another concept, which is stronger than Hadamard differentiability is Fréchet differentiability. We say $\phi$ is \emph{Fréchet differentiable} at $\vartheta_0\in D^\circ$, if if there exists a continuous and linear  mapping $ \phi^\prime_{\vartheta_0}:D \to E$ such that
\[\frac{\norm{\phi(\vartheta_0+ h)-\phi(\vartheta_0)-\phi^\prime_{\vartheta_0}(h)}_E}{\norm{h}_D }\to 0,\quad \text{whenever }\norm{h}_D\to 0.\]

\emph{Connection to condition \ref{cor1a3} in \Cref{corollary1}.} %\label[appendix]{connection}
In our setting we consider $D=\Theta\times \Theta$ for some Banach space $\Theta$ and are only interested in differentiability at the point $a_0=(\vartheta_0,\vartheta_0)$ for $\vartheta_0\in \Theta^\circ$. Consider a loss function $\ell:\Theta\times \Theta \to [0,\infty)$ and set $\phi(a)=\ell(a_1,a_2)$ for $a=(a_1,a_2)\in \Theta\times \Theta$. Condition \ref{cor1a3} from \Cref{corollary1} requires that for $\Theta\subseteq \mathbb{R}^d$ that there exists $ \phi^\prime_{\vartheta_0}:D \to [0,\infty)$ such that for all $h\in \mathbb{R}^d$ and compact set $K\subseteq \mathbb{R}^d$
\begin{align*}
    &\sup_{a=(t,h)\in K \times\{h\} }\abs{\frac{\phi(a_0+r a)-\phi(a_0)}{r^p}-\phi^\prime_{\vartheta_0}(a)  }
    &=\sup_{t\in K}\abs{\frac{\ell\left( \vartheta_{0}+rt,\vartheta_{0}+r h\right)}{r^{p}}-\ell_0(t,h)}
    \to 0,
\end{align*}as $r\to 0,r>0$.
We see that if $(a_1,a_2)\mapsto \ell(a_1,a_2)$ is $p$-th order Hadamard directionally differentiable at $(\vartheta_0,\vartheta_0)$, the loss function $\ell_0(t,h)$ satisfying condition \ref{cor1a3} from \Cref{corollary1}  is given by the $p$-th order Hadamard derivative at $(\vartheta_0,\vartheta_0)$ in direction $(t,h)$. As a consequence we also obtain that $(t,h)\mapsto\ell_0(t,h)$ is continuous in this case.  As we only require the first component of $a$ to range over a compact set and the second component is considered fixed, this condition is weaker than $p$-th order Hadamard directional differentiability and moreover, is only required at a point $a_0=(\vartheta_0,\vartheta_0)$ on the line $\{(a_1,a_2)\in \Theta\times \Theta: a_1=a_2\}$. Furthermore, we do not require the derivative to be linear.
We stress that the function $t\mapsto \ell(t,\vartheta)$ does not have to be differentiable for all $\vartheta$. This can be seen with the loss function $\ell(t,\vartheta)=\norm{t-\vartheta}$. Clearly, it satisfies the condition with $\ell_0=\ell$ and $p=1$, as $(\vartheta_1,\vartheta_2)\mapsto \ell(\vartheta_1,\vartheta_2)$ is first order Hadamard directionally differentiable at $(\vartheta_0,\vartheta_0)$ in direction $(t,h)$ for all $t,h$ and for any $\vartheta_0$ with nonlinear derivative
\[\ell_0(t,h)=\phi^\prime_{\vartheta_0}(t,h)=\norm{t-h},\]but $t\mapsto\ell(t,\vartheta)$ is not differentiable at $t=\vartheta$. Note that this differentiability only holds at points $(\vartheta_1,\vartheta_2)$ on the line $\{(a_1,a_2)\in \Theta\times \Theta: a_1=a_2\}$. This also illustrates that we do not require the derivative to be linear in $(t,h)$. 

So far it is not clear that $\ell_0$ is in general also a loss function. The following proposition gives insight into the nature of the derivative. It shows that the directional derivative of a loss function provides a loss function itself in terms of the directions, in which we take the derivative.
\begin{proposition}
    For $\Theta \subseteq \mathbb{R}^d$, let $\ell:\Theta \times \Theta\to [0,\infty)$ be a loss function such that for some $\vartheta\in \Theta^\circ$, $p>0$ and any $t,h\in \mathbb{R}^d$
    \[\abs{\frac{\ell\left( \vartheta_{0}+rt,\vartheta_{0}+r h\right)}{r^{p}}-\ell_0(t,h)}
    \to 0,\quad r\to 0,\quad r>0\]for some function $\ell_0$ on $\mathbb{R}^d\times \mathbb{R}^d$. Then, $\ell_0$ is nonnegative and $t=h$ implies $\ell_0(t,h)=0$. If there exists a function $\zeta:\mathbb{R}^d \times \mathbb{R}^d \to  [0,\infty)$, which satisfies
    \[\zeta(t,h)>0,\text{if }t\neq h,\quad \zeta(\lambda t,\lambda h)=\lambda^p \zeta(t,h)\text{ and }  \zeta(t+a,h+a)=\zeta(t,h),\quad \forall a,t,h\in \mathbb{R}^d,\lambda>0 \]and there exist $\delta>0$ and $c>0$ such that
    \[\ell(t,\vartheta)\ge c\zeta(t,\vartheta)\]for all $t,\vartheta\in B_\delta(\vartheta_0)$, then $\ell_0(t,h)=0$ implies $t=h$. 
\end{proposition}
\begin{proof}
    By assumption
    \[\lim_{r\to 0,r>0} \frac{\ell\left( \vartheta_{0}+rt,\vartheta_{0}+r h\right)}{r^{p}}=\ell_0(t,h),\]which yields nonnegativity of $\ell_0$. As $\ell$ is a loss function we have for $t=h$ that
    \[\frac{\ell\left( \vartheta_{0}+rt,\vartheta_{0}+r h\right)}{r^{p}}=0\]for all $r>0$. Hence, also the limit is zero, i.e.\ $\ell_0(t,h)=0$ for $t=h$. Assume now there exist $\delta>0$ and $c>0$ such that
    \[\ell(t,\vartheta)\ge c\zeta(t,\vartheta)\]for all $t,\vartheta\in B_\delta(\vartheta_0)$. Let $t\neq h$. Then for all $r>0$
    \[\ell_0(t,h)\ge\frac{\ell\left( \vartheta_{0}+rt,\vartheta_{0}+r h\right)}{r^{p}}-\abs{\frac{\ell\left( \vartheta_{0}+rt,\vartheta_{0}+r h\right)}{r^{p}}-\ell_0(t,h)}. \]The second term on the right hand side vanishes as $r\to 0$. Choose $r>0$ small enough such that $\vartheta_{0}+rt,\vartheta_{0}+r h\in B_\delta(\vartheta_0)$. Then,
     \[\ell_0(t,h)\ge\frac{\ell\left( \vartheta_{0}+rt,\vartheta_{0}+r h\right)}{r^{p}}-\abs{\frac{\ell\left( \vartheta_{0}+rt,\vartheta_{0}+r h\right)}{r^{p}}-\ell_0(t,h)}\ge  c\zeta(t,h)-o(1). \]As $t\neq h$, it holds $\zeta(t,h)>0$. Taking the limit $r\to 0$ yields $\ell_0(t,h)>0$.
\end{proof}A natural choice for $\zeta$ is the $p$-th power of the Euclidean distance $\zeta(t,\vartheta)=\norm{t-\vartheta}^p$. 

The view of $\ell_0$ as a $p$-th order directional derivative also gives a new perspective on the limit process from \Cref{main_thm1} and \Cref{corollary1}, namely
\[Z(t)=\int_{\mathbb{R}^d}\ell_0(t,h)\diff \mathcal{N}\left(Y,I^{-1}_{\vartheta_0}\right)(h),\]where $Y\sim \mathcal{N}\left(0,I^{-1}_{\vartheta_0}\right)$. Considering now that $\ell_0(t,h)$ is a directional derivative of $\ell$ in direction $(t,h)$ we see that $Z(t)$ is the expectation of the directional derivative taken with respect to direction $h$ and with respect to the random measure $\mathcal{N}\left(Y,I^{-1}_{\vartheta_0}\right)$. Thus, for fixed $t$ in some sense $Z(t)$ measures the expected slope ($p=1$), curvature ($p=2$) etc. of the loss function $\ell$ at $(\vartheta_0,\vartheta_0)$ in the random direction $(t,h)$.  As a consequence of \Cref{main_thm1} a Bayes estimator, centered and rescaled, will converge weakly to the direction $t\in \mathbb{R}^d$, such that the expected slope ($p=1$), curvature ($p=2$) etc. of $\ell$ at $(\vartheta_0,\vartheta_0)$ in direction $(t,h)$ is minimal. If this is in the direction $t=Y$, a Bayes estimator is asymptotically efficient.

\begin{example}
    Consider a family of Gaussian distributions with unknown mean and covariance matrix equal to the $d\times d$ identity matrix, denoted as $\mathcal{I}$,
    $\{P_\vartheta=\mathcal{N}(\vartheta,\mathcal{I}^{-1}):\vartheta\in \Theta=\mathbb{R}^d \}$ with $\vartheta_0=(0,\dots,0)$. Consider the loss function $\ell(t,\vartheta)=\norm{t-\vartheta}^2$. In this case it holds $I_{\vartheta_0}=\mathcal{I}$ and
    \[Z(t)=\int_{\mathbb{R}^d}\ell_0(t,h)\diff \mathcal{N}\left(Y,I^{-1}_{\vartheta_0}\right)(h)=d+\sum_{i=1}^d(t_i-Y_i)^2.\]For $d=2$ and $\vartheta_0=(0,0)$ the mapping $t\mapsto Z(t)$ is visualized for fixed $Y=(1,1)$. Here $Z$ measures the expected curvature of $\ell$ at $\vartheta_0$.

    \begin{figure}[ht]
        \centering
        \includegraphics[width=0.8\textwidth]{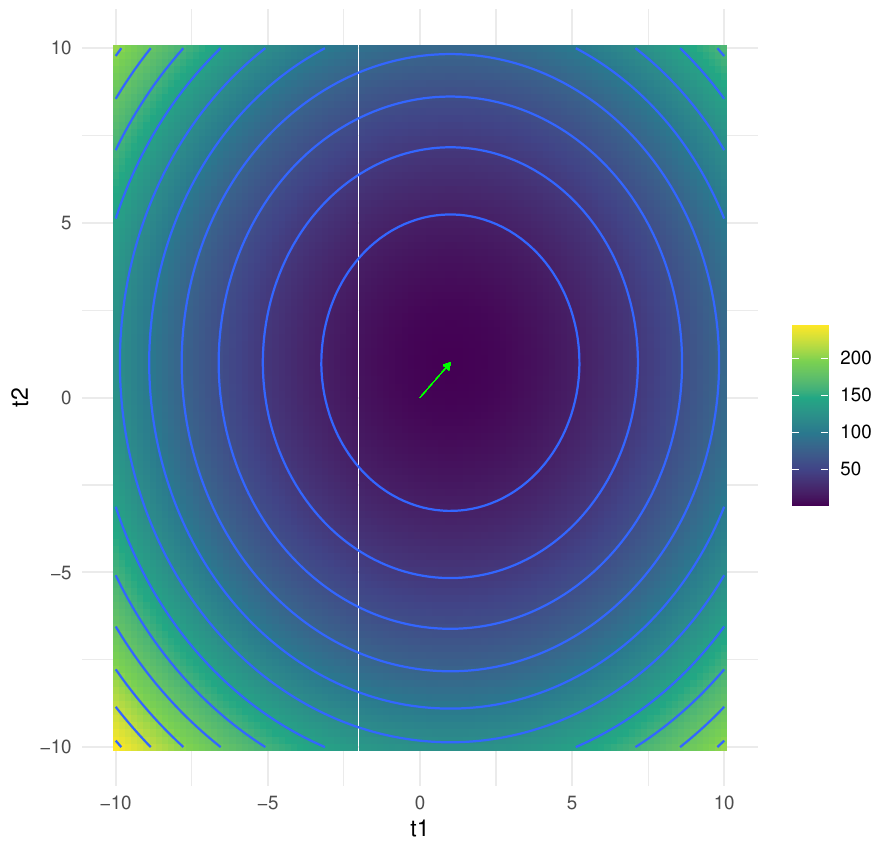}
        \caption{Plot of $(t_1,t_2)\mapsto Z(t_1,t_2)$ for $Y=(1,1)$. The green arrow is the vector $(1,1)$.}
        \label{fig:enter-label}
    \end{figure}
    \end{example}

\subsection{$L^p$-derivative} \label{ss:lp:diff}
For $p\in [1,\infty)$ and a measure space $(\mathcal{X},\mathcal{A},\mu)$
define the $L^p$ space
\[L^p=\{f:\mathcal{X}\to \mathbb{R}:f\text{ is }\mathcal{A}\text{-measurable},\int_{\mathcal{X}}\abs{f(x)}^p\diff \mu(x)<\infty\}.\]We define the seminorm
\[\norm{f}_{L^p}=\left(\int_{\mathcal{X}}\abs{f(x)}^p\diff \mu(x)\right)^{1/p},\quad f\in L^p.\]Let $\mathcal{M}=\{f_\vartheta\in L^p: \vartheta \in \Theta\}$, with $\Theta\subseteq \mathbb{R}^d$, be a parametric family of functions. We say the family $\mathcal{M}$ is \emph{$L^p$-differentiable} at $\vartheta_0\in \Theta^\circ$, if there exists
$\dot f_{\vartheta_0}:\mathcal{X}\to \mathbb{R}^d $ with $\norm{\dot f_{\vartheta_0}}\in L^p$ and
\[\norm{f_\vartheta - f_{\vartheta_0}- \s{\vartheta-\vartheta_0}{\dot f_{\vartheta_0}} }^p_{L^p}=o(\norm{\vartheta-\vartheta_0}^p),\]as $\vartheta\to \vartheta_0$. In this form we can see directly that this is the Fréchet differentiability in the space $L^p$. We consider three examples of $L^p$-differentiability: $L^2$-differentiability of the square root of a density, $L^1$-differentiability of a density and $L^1$-differentiability of a distribution function.

\emph{Differentiability in quadratic mean.}
 We say that a parametric family of probability distributions $\{P_\vartheta:\vartheta\in \Theta\}$, dominated by some measure $\mu$ on $\mathcal{X}$, is \emph{differentiable in quadratic mean} at $\vartheta_0 \in \Theta^\circ$, if there is a vector valued function $\ell'_{\vartheta_0}:\mathcal{X}\to \mathbb{R^d}$ such that
\[\int_{\mathcal{X}} \left[p^{1/2}_{\vartheta_0+h}(x)-p^{1/2}_{\vartheta_0}(x)-\frac{1}{2}\left\langle{h},{\ell'_{\vartheta_0}(x)p^{1/2}_{\vartheta_0}(x)}\right\rangle\right]^2\diff \mu(x)=o(\norm{h}^2),\quad h\to 0.\]Here $p_\vartheta=dP_\vartheta/d\mu$ and $\ell'_{\vartheta_0}$ is called the quadratic mean derivative. In this case, we define the Fisher information matrix at $\vartheta_0$ as
\[I_{\vartheta_{0}}=\int_{\mathcal{X}} \ell_{\vartheta_0}'(x)\left(\ell_{\vartheta_0}'(x)\right)^\intercal \diff P_{\vartheta_0}(x).\] Note that the differentiability in quadratic mean is the same concept as the $L^2$-differentiability of $\{p^{1/2}_\vartheta:\vartheta \in \Theta\}$. Differentiability in quadratic mean implies the existence of the Fisher matrix and
\[\int_{\mathcal{X}}\ell'_{\vartheta_0}(x)\diff P_{\vartheta_0}(x)=0,\]see \citet[Theorem~7.2]{v}. This type of differentiability is exactly the right concept of smoothness of the model required for the BvM theorem (\Cref{B.v.M:}).

\emph{$L^1$-differentiability.}
Differentiability in quadratic mean at $\vartheta_0$ implies that the $L^1$-differentiability of $p_\vartheta$ at $\vartheta_0$ and the derivative is given by $\dot p_{\vartheta_0}=\ell'_{\vartheta_0}p_{\vartheta_0}$, see \citet[Proposition~1.110]{liese}. So the differentiability in quadratic mean is a stronger property than $L^1$-differentiability of the density.

The following proposition gives conditions for $L^1$-differentiability of the distribution function $F_\vartheta(t)=\int_{-\infty}^t p_\vartheta(x)dx$. Under the condition that $\mu(\mathcal{X})<\infty$, the result was shown in Proposition B.2 from the supplement of \cite{Bernton2019}. We extend this by allowing $\mathcal{X}$ to be unbounded.
\begin{proposition}\label[proposition]{L1_diff_cdf}
    Let $\{P_\vartheta:\vartheta\in \Theta\subseteq \mathbb{R}^d\}$ be a parametric family on $\mathcal{X}\subseteq \mathbb{R}$ with densities $\{p_\vartheta:\vartheta\in \Theta\}$ with respect to some $\sigma$-finite measure $\mu$ on $\mathcal{X}$ and there exists $\delta>0$ such that the family is differentiable in quadratic mean at $\vartheta$ for any $\vartheta \in B_\delta(\vartheta_0)$ with derivative $\ell'_{\vartheta}$. Assume one of the two conditions:
    \begin{enumerate}[label=(\roman*)]
        \item $\mu(\mathcal{X})<\infty$.
    \\
    \item \label{propA3_2} For any $x\in \mathcal{X}$ the mapping $\vartheta\mapsto p_\vartheta(x)$ is continuously differentiable at $\vartheta_0\in \Theta^\circ$ with derivative $\dot p_{\vartheta_0}$.  \end{enumerate}Then, it holds for \( \dot F_{\vartheta_0}(x)=\int_{-\infty}^x\ell'_{\vartheta_0}(x)\diff P_{\vartheta_0}(x)\)  that
\[\int_{\mathcal{X}} \abs{F_{\vartheta}(x)-F_{\vartheta_0}(x)-\s{\vartheta-\vartheta_0}{\dot F_{\vartheta_0}(x)}}\diff \mu(x) =o(\norm{\vartheta-\vartheta_0}),\quad \norm{\vartheta-\vartheta_0}\to 0.\]
\end{proposition}
\begin{proof}
%If condition 2 holds, 7.6 Lemma from \cite{v} yields differentiability in quadratic mean with $\ell'_{\vartheta_0}(x)=\frac{\dot p_{\vartheta_0}(x)}{ p_{\vartheta_0}(x)}$. 
By differentiability in quadratic mean it follows from \citet[Proposition 1.110]{liese} that the $L^1$-derivative of $p_\vartheta$ at $\vartheta_0$ is given by $\dot p_{\vartheta_0}=\ell'_{\vartheta_0}p_{\vartheta_0}$.
    Define 
    \begin{align*}{\rev \mathcal{F}_\vartheta(t)} & =\int_{-\infty}^t \abs{p_\vartheta(x)- p_{\vartheta_0}(x)-\s{\vartheta-\vartheta_0}{\dot p_{\vartheta_0}(x)}}\diff \mu(x) \\
    &=\norm{\left(p_\vartheta- p_{\vartheta_0}-\s{\vartheta-\vartheta_0}{\dot p_{\vartheta_0}}\right)\mathbb{I}(\cdot \le t) }_{L^1}. \end{align*}Then, it holds 
\[\norm{\vartheta-\vartheta_0}^{-1}\int_{\mathcal{X}} \abs{F_{\vartheta}(t)-F_{\vartheta_0}(t)-\s{\vartheta-\vartheta_0}{\dot F_{\vartheta_0}(t)}}\diff \mu(t)\le \int_{\mathcal{X}}\frac{{\rev \mathcal{F}_\vartheta(t)}}{\norm{\vartheta-\vartheta_0}}\diff \mu(t).\]For any fixed $t\in \mathcal{X}$ 
\[0\le \lim_{\vartheta\to \vartheta_0}\frac{{\rev \mathcal{F}_\vartheta(t)}}{\norm{\vartheta-\vartheta_0}}\le \lim_{\vartheta\to \vartheta_0}\frac{\norm{p_\vartheta- p_{\vartheta_0}-\s{\vartheta-\vartheta_0}{\dot p_{\vartheta_0}} }_{L^1} }{\norm{\vartheta-\vartheta_0}}=0\]by $L^1$-differentiability of $p_\vartheta$ at $\vartheta_0$. If $\mathcal{X}$ is bounded, we are done. If condition \ref{propA3_2}  holds, we use the dominated convergence theorem with the integrable upper bound
\[\frac{{\rev \mathcal{F}_\vartheta(t)}}{\norm{\vartheta-\vartheta_0}}\le \sup_{\tilde \vartheta \in B_\delta(\vartheta_0)}\norm{ \norm{\dot p_{\tilde \vartheta}-\dot p_{\vartheta_0} }}_{L^1}\]by choosing $\delta>\norm{\vartheta-\vartheta_0}$ small enough. The upper bound is obtained by the mean value theorem. 
\end{proof}
As {\rev demonstrated} in Section~\ref{ss:stein}, {\rev$L^1$-differentiability of the distribution function}
 is closely related to the {\rev directional Hadamard differentiability} of the 1-Wasserstein distance.
\subsection{Differentiability of the parametric Wasserstein distance}\label[appendix]{diff_ws}
To illustrate that the loss function $\ell_{W_2}$ is indeed differentiable (with the usual notion of differentiability on Euclidean spaces), we state sufficient conditions for the existence of the first two derivatives. The third derivative can be obtained in an analogous fashion as the second under additional regularity assumptions, which allow to interchange order of integration and derivation. 

Let $\mathcal{W}_2(\mathcal{X})$ be the space of probability measure on $\mathcal{X}$ with finite second moment. {\rev For a parametric family of probability measures  $\{P_{\vartheta}:\vartheta\in \Theta\}\subset \mathcal{W}_2(\mathcal{X})$ that are absolutely continuous with respect to the Lebesgue measure,  we denote, for
any $t,\vartheta\in \Theta$, 
	\[
	\Phi(t,\vartheta):=\left\{\phi\in L^1(P_t):W_2^2(P_t,P_\vartheta)  =\int_{\mathcal{X}}\phi(x)\diff P_t(x)+\int_{\mathcal{X}}\phi^c(y)\diff P_\vartheta(y)\right\}	
	\]
	as the set of solutions to the dual formulation (see e.g.~\citealp{San15}), where the \emph{$c$-transform} is 
	\begin{equation}\label{e:c:transform}
	\phi^c(y)=\inf_{x\in \mathcal{X}}\norm{x-y}^2-\phi(x).
	\end{equation}
	The set $\Phi(t,\vartheta)$ is nonempty by \cite[Theorem 5.10]{villani2}. }
\begin{theorem}\label[theorem]{grad_ws}
	Let $k\in \mathbb{N}$, $\mathcal{X}\subseteq \mathbb{R}^k$ be a convex set, let $\Theta\subseteq\mathbb{R}^d$ be open and denote the Lebesgue measure on $\mathbb{R}^k$ as $\lambda$. Suppose that $\{P_{\vartheta}:\vartheta\in \Theta\}\subset \mathcal{W}_2(\mathcal{X})$ is a parametric family of probability measures with $P_\vartheta \ll \lambda$ for all $\vartheta\in \Theta$. We denote $p_\vartheta$ as the Lebesgue density of $P_\vartheta$ and assume that $\vartheta\mapsto p_\vartheta(x)$ is differentiable for any $x\in \mathcal{X}$ and $p_\vartheta(x)>0$ for any $x\in \mathcal{X}$ and $\vartheta\in \Theta$. For $t\in \Theta$ and $\delta>0$, let 
	\[
	B_\delta(t)=\left\{s\in \Theta:\norm{s-t}\le \delta\right\}.
	\]
	Furthermore, for some fixed $t_0\in \Theta$ and for any $t, \vartheta \in \Theta$, {\rev we assume that there exists a family of nonempty subsets $\Phi_0(t, \vartheta) \subseteq \Phi(t, \vartheta)$ such that:
	\begin{enumerate}[label=(\roman*)]
		\item\label{ThmA4_1}
		 For any sequence $t_n\to t_0$, there exists a sequence $\phi_{t_n}^\vartheta\in \Phi_0(t_n,\vartheta)$ and a fixed element $\phi_{t_0}^\vartheta \in \Phi_0(t_0,\vartheta)$ such that,  as $n\to \infty$, for $\lambda$-a.e.~$x\in\mathcal{X}$,
		\[ 
		\phi_{t_n}^\vartheta(x)- \phi_{t_0}^\vartheta(x)\longrightarrow 0.
		\] 
		\item\label{ThmA4_2} 
		There exists  $\delta>0$ such that, for all $i=1,\dots,d$, 
		\[ 
		\bigintssss_{\mathcal{X}} \sup_{t,\tilde{t}\in B_\delta(t_0)} \sup_{\phi \in \Phi_0(t, \vartheta)}\abs{\phi(x)\frac{\partial}{\partial v_i}p_{v}(x)\Big\vert_{v=\tilde{t}} }\diff x<\infty.
		\]
	\end{enumerate}}
	Then, the function $t\mapsto W_2^2(P_t,P_\vartheta)$ is differentiable at $t_0\in \Theta$ for any $\vartheta\in \Theta$ with derivative \[\nabla_t W_2^2(P_t,P_\vartheta)\vert_{t=t_0}=\int_{\mathcal{X}}\phi_{t_0}^\vartheta(x)\nabla_tp_t(x)\vert_{t=t_0}\diff x\]for any $\phi_{t_0}^\vartheta\in \Phi_0(t_0,\vartheta)$.
\end{theorem}
\begin{proof}
	Let $e_i$, $i=1\dots,d$, denote the unit vectors and set\[I_{i}(h):=\frac{W_2^2(P_{t_0+he_i},P_\vartheta)-W_2^2(P_{t_0},P_\vartheta)}{h},\quad h\neq0: t_0+he_i\in \Theta. \]Then, if the limit exists, it holds
	\[
	\lim\limits_{h\to 0}I_i(h)={\rev\frac{\partial}{\partial v_i} W_2^2(P_v,P_\vartheta)\Big\vert_{v=t_0}}.
	\] 
	Recall the dual formulation of the squared 2-Wasserstein distance \cite[Theorem~5.10]{villani2}. For any $\mu,\nu \in \mathcal{W}_2(\mathcal{X})$, it holds
	\[
	W_2^2(\mu,\nu)=\sup_{\phi\in L^1(\mu)}\left(\int_{\mathcal{X}}\phi(x)\diff \mu(x)+\int_{\mathcal{X}}\phi^c(y)\diff \nu(y)\right),
	\]
	{\rev where $\phi^c$ is the $c$-transform of $\phi$, defined in~\eqref{e:c:transform}. 
	By condition~\ref{ThmA4_1}, for any $i=1,\ldots,d$, there exist a sequence $\phi_{t_0+e_ih}^\vartheta\in \Phi_0(t_0+e_i h,\vartheta)\subset \Phi(t_0+e_i h,\vartheta)$ and a fixed 
	$\phi_{t_0}^\vartheta\in \Phi_0(t_0,\vartheta)$} such that
    \begin{equation}
         \phi_{t_0+e_ih}^\vartheta(x)-\phi_{t_0}^\vartheta(x)\to 0,\quad h\to 0, \label{Thm4.1.eq1}
    \end{equation}
    for $\lambda$-a.e.~$x\in \mathcal{X}$. {\rev By the dual formulation, we obtain
    \[ 
    W_2^2(P_{t_0},P_\vartheta)\ge \int_{\mathcal{X}}\frac{\phi_{t_0+e_ih}^\vartheta(x)p_{t_0}(x)}{h}\diff x
		+\int_{\mathcal{X}}\frac{\left(\phi_{t_0+e_ih}^\vartheta\right)^c(x)p_{\vartheta}(x)}{h}\diff x 
		\]
		and consequently}
	\begin{align*}
		I_i(h)&\le \int_{\mathcal{X}}\frac{\phi_{t_0+e_ih}^\vartheta(x)p_{t_0+e_ih}(x)}{h}\diff x
		+\int_{\mathcal{X}}\frac{\left(\phi_{t_0+e_ih}^\vartheta\right)^c(x)p_{\vartheta}(x)}{h}\diff x
		\\&\hspace{5.4cm}-\int_{\mathcal{X}}\frac{\phi_{t_0+e_ih}^\vartheta(x)p_{t_0}(x)}{h}\diff x
		-\int_{\mathcal{X}}\frac{\left(\phi_{t_0+e_ih}^\vartheta\right)^c(x)p_{\vartheta}(x)}{h}\diff x\\
		&=\int_{\mathcal{X}}\frac{\phi_{t_0+e_ih}^\vartheta(x)\left(p_{t_0+e_ih}(x)-p_{t_0}(x)\right)}{h}\diff x\\
		&=\int_{\mathcal{X}}\frac{\phi_{t_0}^\vartheta(x)\left(p_{t_0+e_ih}(x)-p_{t_0}(x)\right)}{h}\diff x+\int_{\mathcal{X}}\frac{\left(\phi_{t_0+e_ih}^\vartheta(x)-\phi_{t_0}^\vartheta(x)\right)\left(p_{t_0+e_ih}(x)-p_{t_0}(x)\right)}{h}\diff x.
	\end{align*}
	 The mean value theorem implies that for fixed $x\in \mathcal{X}$, $t_0\in \Theta$ and for $\abs{h}\le \delta$ it holds
	\[\abs{ \frac{p_{t_0+e_ih}(x)-p_{t_0}(x)}{h}}={\rev \abs{\frac{\partial}{\partial v_i}p_v(x)\Big\vert_{v=\tilde{t}} } }\]for some $\tilde{t}=t_0+e_ih_0$ and $\abs{h_0}\le \delta$. Denote $B_\delta(t_0)=\{s\in \Theta:\norm{s-t_0}\le \delta\}$. For any fixed $x\in \mathcal{X}$, we obtain the upper bound 
	\begin{align*} 
	&\abs{ \left(\phi_{t_0+e_ih}^\vartheta(x)-\phi_{t_0}^\vartheta(x)\right) \frac{p_{t_0+e_ih}(x)-p_{t_0}(x)}{h}}\le \sup_{t,\tilde{t}\in B_\delta(t_0)}{\rev\sup_{\phi \in \Phi_0(t, \vartheta)}\abs{\left(\phi(x)-\phi_{t_0}^\vartheta(x)\right) \frac{\partial}{\partial v_i}p_{v}(x) \Big \vert_{v=\tilde{t}} } }
    \\ &\le {\rev \sup_{t,\tilde{t}\in B_\delta(t_0)} \sup_{\phi \in \Phi_0(t, \vartheta)}\abs{\phi(x) \frac{\partial}{\partial v_i}p_{v}(x) \Big \vert_{v=\tilde{t}} } +\sup_{\tilde{t}\in B_\delta(t_0)} \abs{\phi^\vartheta_{t_0}(x) \frac{\partial}{\partial v_i}p_{v}(x) \Big \vert_{v=\tilde{t}} } }   \\
		&\le2 \sup_{t,\tilde{t}\in B_\delta(t_0)}{\rev \sup_{\phi \in \Phi_0(t, \vartheta)}\abs{\phi(x) \frac{\partial}{\partial v_i}p_{v}(x) \Big \vert_{v=\tilde{t}} } }. \end{align*} 
		By condition~\ref{ThmA4_2}, we can choose $\delta>0$ small enough such that the upper bound is integrable. Analogously, we obtain the following upper bound that is integrable, 
	\[ \abs{ \phi_{t_0}^\vartheta(x) \frac{p_{t_0+e_ih}(x)-p_{t_0}(x)}{h}}\le\sup_{t,\tilde{t}\in B_\delta(t_0)}{\rev \sup_{\phi \in \Phi_0(t, \vartheta)}\abs{\phi(x)  \frac{\partial}{\partial v_i}p_{v}(x) \Big\vert_{v=\tilde{t}} } }. \]  
	Consequently, Lebesgue's dominated convergence theorem and differentiability of the mapping  $\vartheta\mapsto p_\vartheta(x)$  imply
	\begin{align*}
	&\limsup_{h\to 0}I_i(h)\\
	\le&\lim\limits_{h\to 0}\Biggl[ \int_{\mathcal{X}}\frac{\phi_{t_0}^\vartheta(x)\left(p_{t_0+e_ih}(x)-p_{t_0}(x)\right)}{h}\diff x+\int_{\mathcal{X}}\frac{\left(\phi_{t_0+e_ih}^\vartheta(x)-\phi_{t_0}^\vartheta(x)\right)\left(p_{t_0+e_ih}(x)-p_{t_0}(x)\right)}{h}\diff x\Biggr] \\=&\int_{\mathcal{X}}\phi_{t_0}^\vartheta(x) {\rev \frac{\partial}{\partial v_i}p_v(x)\Big\vert_{v=t_0} }\diff x.
	\end{align*} 
	On the other hand, we have
	\begin{align*}
		I_i(h)&\ge \int_{\mathcal{X}}\frac{\phi_{t_0}^\vartheta(x)p_{t_0+e_ih}(x)}{h}\diff x
		+\int_{\mathcal{X}}\frac{\left(\phi_{t_0}^\vartheta\right)^c(x)p_{\vartheta}(x)}{h}\diff x
		\\&\hspace{4.8cm}-\int_{\mathcal{X}}\frac{\phi_{t_0}^\vartheta(x)p_{t_0}(x)}{h}\diff x
		-\int_{\mathcal{X}}\frac{\left(\phi_{t_0}^\vartheta\right)^c(x)p_{\vartheta}(x)}{h}\diff x\\
		&=\int_{\mathcal{X}}\frac{\phi_{t_0}^\vartheta(x)\left(p_{t_0+e_ih}(x)-p_{t_0}(x)\right)}{h}\diff x.
	\end{align*}The dominated convergence theorem yields
	\[\liminf_{h\to 0}I_i(h)\ge \int_{\mathcal{X}}\phi_{t_0}^\vartheta(x) {\rev \frac{\partial}{\partial v_i}p_v(x)\Big\vert_{v=t_0} }\diff x.\] Hence, we have derived that
	\[
	\nabla_t W_2^2(P_t,P_\vartheta)\vert_{t=t_0}=\int_{\mathcal{X}}\phi_{t_0}^\vartheta(x)\nabla_tp_t(x)\vert_{t=t_0}\diff x. 
	\]
    {This concludes the proof.}
\end{proof}
In \Cref{grad_ws}, condition~\ref{ThmA4_1} is a continuity condition on the dual solutions in the parameter. By \citet[Theorem~2]{goett}, the dual solutions are unique up to additive constants. Moreover, condition~\ref{ThmA4_2} implies
that \[	\bigintssss_{\mathcal{X}} {\rev \sup_{v\in B_\delta(t_0)} \abs{\frac{\partial}{\partial v_i}p_{v}(x)} } \diff x<\infty\] and consequently
\[\int_{ \mathcal{X}}\nabla_tp_t(x)\vert_{t=t_0}\diff x=0.\]Hence, the formula of the gradient of the Wasserstein distance is independent of the choice of the concrete potential and it holds {\rev 
\[
\nabla_t W_2^2(P_t,P_\vartheta)\vert_{t=t_0}=\int_{\mathcal{X}}\phi(x)\nabla_tp_t(x)\vert_{t=t_0}\diff x=\int_{\mathcal{X}}\psi(x)\nabla_tp_t(x)\vert_{t=t_0}\diff x
\] 
for any $\phi,\psi\in \Phi(t_0,\vartheta)$. The derivative of the squared 2-Wasserstein distance depends on $\vartheta$ only through $\phi_{t_0}^\vartheta\in \Phi(t_0,\vartheta)$. The collection of subsets $\{\Phi_0(t,\vartheta):t\in \Theta\}$ is typically chosen so that each set consists of potentials in $\Phi(t,\vartheta)$ with a fixed additive constant. In this case, $\Phi_0(t,\vartheta)$ is a singleton for every $t,\vartheta \in \Theta$.}

If the space $\mathcal{X}$ is a subset of the real line, we obtain a more concrete form for the second derivative and sufficient conditions for differentiability of the mapping $t\mapsto\nabla_t W_2^2(P_t,P_\vartheta)$. The next result is inspired by Theorem~1 from \cite{differentialwasser}, but does not assume that $\mathcal{X}$ is compact. We denote $\overline{\mathbb{R}}=\mathbb{R}\cup \{-\infty,\infty\}$ as the extended real numbers.

\begin{theorem}\label[theorem]{grad_ws_real} 
Assume that $\mathcal{X}=(a,b),$ where $a,b\in \overline{\mathbb{R}}$ and $a<b$. Let $\Theta\subseteq\mathbb{R}^d$ be open and $\{P_{\vartheta}:\vartheta\in \Theta\}\subset \mathcal{W}_2(\mathcal{X})$ be a parametric family of probability measures that are absolutely continuous with respect to  the Lebesgue measure on $\mathbb{R}$. The Lebesgue measure restricted to $\mathcal{X}$ is denoted as $\lambda$. Furthermore, suppose that for any $\vartheta\in \Theta$ we have $p_\vartheta(x)>0$ for $x\in \mathcal{X}$ and $p_\vartheta(x)=0$ for $x\notin\mathcal{X}$.
	We will denote $F_\vartheta$ as the cumulative distribution function of $P_\vartheta$, i.e.\ $F_\vartheta(x)=\int_{-\infty}^{x}p_\vartheta(y)\diff y$ and $F_\vartheta^{-1}$ as the generalized inverse of $F_\vartheta$, i.e.
	\[F_\vartheta^{-1}(u)=\inf\{x:F_\vartheta(x)\ge u\}.\]Let $T_t^\vartheta(x)$ denote the optimal transport map from $P_t$ to $P_\vartheta$. For $t\in \Theta$ and $\delta>0$ define
	\[B_\delta(t):=\{\vartheta \in \Theta:\norm{t-\vartheta}\le \delta\}. \]
	In addition, for some fixed $t_0\in \Theta$ we make the following regularity assumptions: \begin{enumerate}[label=(\roman*)]
		\item \label{ThmA5_1} The map $\vartheta\mapsto p_\vartheta(x)$ is differentiable for any $x\in \mathcal{X},\vartheta \in \Theta$ and for any $\vartheta\in \Theta$ it exists $\delta>0$ such that \[\int_{\mathcal{X}}{\rev  \sup_{v\in B_\delta(t_0)}  \abs{\frac{\partial}{\partial v_i}p_v(x)} }\diff x <\infty \]for all $i\in\{1,\dots,d\}$.
		\item \label{ThmA5_2} The map $\vartheta\mapsto F_\vartheta(x)$ is twice differentiable for any $x\in \mathcal{X},\vartheta \in \Theta$, and for any $\vartheta\in \Theta$ there exists $\eta>0$ such that\[\int_{\mathcal{X}}{\rev  \sup_{v\in B_\eta(t_0) }\abs{\left(T_{v}^\vartheta(x)-x\right)\left(\frac{\partial^2}{\partial v_i v_j}F_v(x)\right)+\left(\frac{\partial}{\partial v_i}T_v^\vartheta(x)\right)\left(\frac{\partial}{\partial v_j}F_v(x)\right)} }\diff x<\infty \]for all $i,j\in\{1,\dots,d\}$.
		\item \label{ThmA5_3} For any $\vartheta\in \Theta$ and $x_0\in \mathcal{X}$, it holds  {\rev
		\[\lim\limits_{x\to a}\nabla_tF_t(x)\vert_{t=t_0}\bigintssss_{x_0}^x \left(y-T_{t_0}^\vartheta(y)\right) \diff y=0 ,\;      \lim\limits_{x\to b}\nabla_tF_t(x)\vert_{t=t_0}\bigintssss_{x_0}^x \left(y-T_{t_0}^\vartheta(y)\right) \diff y=0.\]}
	\end{enumerate} 
	If the mapping $t\mapsto W_2^2(P_t,P_\vartheta)$ is differentiable for any $\vartheta\in \Theta$, {\rev with the derivative  \[\nabla_t W_2^2(P_t,P_\vartheta)\vert_{t=t_0}=\int_{\mathcal{X}}\phi_{t_0}^\vartheta(x)\nabla_tp_t(x)\vert_{t=t_0}\diff x\]for any $\phi_{t_0}^\vartheta\in \Phi(t_0,\vartheta)$,} then the second derivative at $t_0$ exists, and it holds
	\begin{align*}
		\nabla_t^2W_2^2(P_t,P_\vartheta)\vert_{t=t_0}&=2\bigg(\int_{\mathcal{X}}(T_{t_0}^\vartheta(x)-x)\nabla^2_t F_t(x)\vert_{t=t_0} \diff x\\&\qquad+\int_{\mathcal{X}}\left(\frac{\partial}{\partial x}T_{t_0}^\vartheta(x)\right)\frac{\nabla_t F_t(x)\vert_{t=t_0} \left(\nabla_t F_t(x)\vert_{t=t_0}\right)^\intercal}{p_{t_0}(x)}\diff x\bigg).
	\end{align*}
\end{theorem}

\begin{proof} 
By \citet[Theorem~2.18]{villani1}, the optimal transport map is given by $T_{t_0}^\vartheta(x)=F^{-1}_\vartheta(F_{t_0}(x))$. By \citet[Theorem~1.2]{mccan}, it holds, for any $\phi_{t_0}^\vartheta(x)\in \Phi(t_0,\vartheta)$,
	\begin{equation}
	\frac{\partial}{\partial x}\phi_{t_0}^\vartheta(x)=2(x-T_{t_0}^\vartheta(x)) \label{thm16:eq1}
	\end{equation}
for $P_{t_0}$-a.e.~$x\in \mathcal{X}$. As the density of $P_{t_0}$ is strictly positive on $\mathcal{X}$, we conclude that {\rev(\ref{thm16:eq1})}
 also holds $\lambda$-a.e.\ $x\in \mathcal{X}$. Hence, for any $C\in \mathbb{R}$ and {\rev $x_0\in \mathcal{X}$}, the function
	\begin{equation}{\rev x\mapsto \int_{x_0}^x 2\left(y-T_{t_0}^\vartheta(y)\right)\diff y+C, }\label{ta2_1}\end{equation} is included in the set $\Phi(t_0,\vartheta)$.
    {\rev Fix some $C\in \mathbb{R}$ and $x_0\in \mathcal{X}$ and set 
    $$
    \phi_{t_0}^\vartheta(x)=\int_{x_0}^x 2\left(y-T_{t_0}^\vartheta(y)\right)\diff y+C \;\in\; \Phi(t_0,\vartheta).
    $$}
   The first derivative is given by  
	\[\nabla_t W_2^2(P_t,P_\vartheta)\vert_{t=t_0}=\int_{\mathcal{X}}\phi_{t_0}^\vartheta(x)\nabla_tp_t(x)\vert_{t=t_0}\diff x.\]
	 Integration by parts yields
	\[\int_{\mathcal{X}}\phi_{t_0}^\vartheta(x)\nabla_tp_t(x)\vert_{t=t_0}\diff x=\left[\phi_{t_0}^\vartheta(x)\nabla_tF_t(x)\vert_{t=t_0}\right]_a^b-\int_{\mathcal{X}}2(x-T_{t_0}^\vartheta(x))\nabla_tF_t(x)\vert_{t=t_0}\diff x,\]where the first term on the right hand side is to be understood as the limit
	\[\left[\phi_{t_0}^\vartheta(x)\nabla_tF_t(x)\vert_{t=t_0}\right]_a^b=\lim\limits_{x\to b}\phi_{t_0}^\vartheta(x)\nabla_tF_t(x)\vert_{t=t_0}- \lim\limits_{x\to a}\phi_{t_0}^\vartheta(x)\nabla_tF_t(x)\vert_{t=t_0}. \]

{\rev By  \eqref{ta2_1} and condition~\ref{ThmA5_3}, it holds 
\begin{align*}
\lim\limits_{x\to b}\phi_{t_0}^\vartheta(x)\nabla_tF_t(x)\vert_{t=t_0}&=\lim\limits_{x\to b}\left(C \nabla_tF_t(x)\vert_{t=t_0}+\nabla_tF_t(x)\vert_{t=t_0}\int_{x_0}^x 2\left(y-F_\vartheta^{-1}\left(F_{t_0}(y)\right)\right)  \diff y\right)\\
		&=\lim\limits_{x\to b}C \nabla_tF_t(x)\vert_{t=t_0}.
		\end{align*}
        By condition~\ref{ThmA5_1}, we can interchange order of derivation and integration, which leads to
	\[\lim\limits_{x\to b}C \nabla_tF_t(x)\vert_{t=t_0}=\lim\limits_{x\to b}
	C\int_a^x\nabla_t p_t(y)\vert_{t=t_0}\diff y=C\int_{\mathcal{X}}\nabla_t p_t(y)\vert_{t=t_0}\diff y =\nabla_t\left(\int_{\mathcal{X}} p_{t}(y)\diff y\right)\Big\vert_{t=t_0}= 0.\]Here we used that $p_{t}$ integrates to 1. This shows that $\lim\limits_{x\to b}\phi_{t_0}^\vartheta(x)\nabla_tF_t(x)\vert_{t=t_0}=0$. 
	By  the dominated convergence theorem, we obtain
	\[\lim\limits_{x\to a} \nabla_t F_t(x)\vert_{t=t_0}=\lim\limits_{x\to a} \int_{a}^x\nabla_t p_t(y)\vert_{t=t_0}\diff y =0\]using the integrable upper bound 
	\[ \mathbb{I}(a< y\le x)\nabla_t p_t(y)\vert_{t=t_0}\le \nabla_t p_t(y)\vert_{t=t_0}.\] 
	By this convergence and condition \ref{ThmA5_3}, it holds \[\lim\limits_{x\to a}\phi_{t_0}^\vartheta(x)\nabla_tF_t(x)\vert_{t=t_0}=\lim\limits_{x\to a}\nabla_tF_t(x)\vert_{t=t_0} \int_{x_0}^x 2\left(y-T_{t_0}^\vartheta(y)\right)\diff y+ C \nabla_tF_t(x)\vert_{t=t_0} =0.
	\]}
 Thus, we have derived that 
	\[
	\int_{\mathcal{X}}\phi_{t_0}^\vartheta(x)\nabla_tp_t(x)\vert_{t=t_0}\diff x=-\int_{\mathcal{X}}2(x-T_{t_0}^\vartheta(x))\nabla_tF_t(x)\vert_{t=t_0}\diff x.
	\]
	The derivative with respect to  $x$ is 
	\[
	\frac{\partial}{\partial x}T_{t_0}^\vartheta(x)=\frac{\partial}{\partial y}F_\vartheta^{-1}(y)\vert_{y=F_{t_0}(x)}\cdot p_{t_0}(x)=\frac{p_{t_0}(x)}{p_\vartheta(F_\vartheta^{-1}(F_{t_0}(x)))}=\frac{p_{t_0}(x)}{p_\vartheta(T_{t_0}^\vartheta(x))}.
	\]
	On the other hand, taking the derivative with respect to  $t$ at $t=t_0$ yields
	\[\nabla_tT_t^\vartheta(x)\vert_{t=t_0}=\frac{\nabla_t F_t(x)\vert_{t=t_0}}{p_\vartheta(T_{t_0}^\vartheta(x))}=\left(\frac{\partial}{\partial x}T_{t_0}^\vartheta(x)\right)\frac{\nabla_t F_t(x)\vert_{t=t_0}}{p_{t_0}(x)}.\] Recall that condition~\ref{ThmA5_2} enables the dominated convergence theorem, c.f.\ Theorem~\ref{grad_ws}, so  we are allowed to interchange integral and derivative. This leads to
	\begin{align*}&\nabla_t^2W_2^2(P_t,P_\vartheta)\\=&2\left(\int_{\mathcal{X}}(T_{t_0}^\vartheta(x)-x)\nabla^2_t F_t(x)\vert_{t=t_0} \diff x+\int_{\mathcal{X}} \left(\nabla_tT_t^\vartheta(x) \vert_{t=t_0}\right)\left(\nabla_t F_t(x)\vert_{t=t_0}\right)^\intercal \diff x\right)
		\\= &2\left(\int_{\mathcal{X}}(T_{t_0}^\vartheta(x)-x)\nabla^2_t F_t(x)\vert_{t=t_0} \diff x+\int_{\mathcal{X}}\left(\frac{\partial}{\partial x}T_t^\vartheta(x)\right)\frac{\nabla_t F_t(x)\vert_{t=t_0} \left(\nabla_t F_t(x)\vert_{t=t_0}\right)^\intercal}{p_t(x)}\diff x\right).\end{align*}
\end{proof}

The conditions stated in \cite[Theorem~1]{differentialwasser} only consider a compact set $\mathcal{X}$ and require that the derivative of the optimal transport map is bounded. These conditions imply the conditions of  \Cref{grad_ws,grad_ws_real}. In addition, note that the conditions of  \Cref{grad_ws,grad_ws_real} are just sufficient and not necessary. 
\begin{example}[Laplace distribution]
	Consider the location-scale family of Laplace distributions
	\[\mathcal{M}=\{P_{(a,b)}=\text{Lap}(a,b):(a,b)\in \Theta= \mathbb{R}\times (0,\infty)\}.\]The density is not differentiable for all parameter values, but the squared 2-Wasserstein distance between two elements is given as
	\[W_2^2(P_{(a,b)},P_{(c,d)})=(a-c)^2+2(b-d)^2,\quad (a,b),(c,d)\in \Theta\] and is differentiable for all parameter values.
\end{example} To illustrate the fact that the conditions of \Cref{grad_ws,grad_ws_real} hold for typical models,
we give two examples.
\begin{example}[Pareto distribution]\label{ex:pare}
	Consider the parametric shape family of Pareto distributions
	\[\mathcal{M}=\{P_t=\mathrm{Pareto}(t,1):t\in \Theta=(2,\infty)\}\]on $\mathcal{X}=(1,\infty)$. For any $t,\vartheta \in \Theta$ let $p_t$ denote the density of $P_t$,
	$F_t$ the distribution function of $P_t$, and $T_t^\vartheta$ the optimal transport map, which transports $P_t$ to $P_\vartheta$ and {\rev $\Phi(t,\vartheta)$ the set of solutions} to the dual formulation of the Wasserstein distance.
	Then, it holds
	\begin{align*}
		&p_t(x)=\frac{t}{x^{t+1}},\; \frac{\partial}{\partial t}p_t(x)=-\frac{\log(x)t-1}{x^{t+1}},\\
		&F_t(x)=1-x^{-t},\; \frac{\partial}{\partial t}F_t(x)=\frac{\log(x)}{x^t},\; \frac{\partial^2}{\partial t^2}F_t(x)=\frac{\log(x)}{x^t}\\
		&T^\vartheta_t(x)=F_\vartheta^{-1}(F_t(x))=x^{t/\vartheta},\;  \frac{\partial}{\partial t}T_t^\vartheta(x)=\frac{\log(x)x^{t/\vartheta}}{\vartheta}, \\
        &{\rev \Phi(t,\vartheta)=\left\{x\mapsto C+x^2+\frac{2\vartheta}{t+\vartheta}\left(1-x^{1+t/\vartheta}\right):C\in \mathbb{R} \right\}}.
	\end{align*} It can be seen that the conditions of \Cref{grad_ws,grad_ws_real} are fulfilled {\rev for $\Phi_0(t,\vartheta)=\{x\mapsto x^2+\frac{2\vartheta}{t+\vartheta}\left(1-x^{1+t/\vartheta}\right) \}$}, since the logarithm grows slower than $x^t$ for all $t>0$ and\[\int_{1}^{\infty}\frac{\log(x)}{x^a}\diff x<\infty\]for $a>1$.
\end{example}

{\rev
\begin{example}[Beta distribution]\label{ex:beta}
    Consider the parametric shape family of Beta distributions
	\[
	\mathcal{M}=\{P_t=\mathrm{Beta}(t,1):t\in \Theta=(0,\infty)\}
	\]
	on $\mathcal{X}=(0,1)$. For any $t,\vartheta \in \Theta$, by $p_t$ and $F_t$ denote the density the distribution function of $P_t$, respectively. Let $T_t^\vartheta$  be the optimal transport map, which transports $P_t$ to $P_\vartheta$ and {\rev $\Phi(t,\vartheta)$ the set of solutions} to the dual formulation of the Wasserstein distance. Then, it holds
	\begin{align*}
		&p_t(x)=tx^{t-1},\; \frac{\partial}{\partial t}p_t(x)=x^{t-1}(1+t\log(x)),\\
		&F_t(x)=x^{t},\; \frac{\partial}{\partial t}F_t(x)=x^t\log(x),\; \frac{\partial^2}{\partial t^2}F_t(x)=x^t\log(x)^2, \\
		&T^\vartheta_t(x)=F_\vartheta^{-1}(F_t(x))=x^{t/\vartheta},\; \frac{\partial}{\partial t}T_t^\vartheta(x)=\frac{\log(x)x^{t/\vartheta}}{\vartheta}, \\ 
		&\Phi(t,\vartheta)=\{x\mapsto C+x^2-2x^{(t+\vartheta)/\vartheta}:C\in \mathbb{R}\}.
	\end{align*} 
	Thus, with the choice $\Phi_0(t,\vartheta)=\{x\mapsto x^2-2x^{(t+\vartheta)/\vartheta}\}$,  the conditions in \Cref{grad_ws,grad_ws_real} are fulfilled.
\end{example}
}
\section{Technical results}\label[supplement]{technical}
This supplement includes two technical results, which are consequences of the Bernstein--von Mises (BvM) theorem. 
The first result shows that the mass of the set $C_n^c=\left\{h\in H_n:\norm{h}\ge M_n\right\}$ with respect to  $P_{h_n\vert {\boldsymbol X}}$, where $(M_n)_{n\in \mathbb{N}}$ is any diverging sequence of positive numbers,  is asymptotically negligible. As a consequence, the posterior distribution asymptotically concentrates on sets of radius of the order $n^{-1/2}$.
\begin{proposition}\label[proposition]{weak_conv_prop1}
	Assume that $X_1,\dots,X_n\stackrel{\text{i.i.d.}}{\sim}P_{\vartheta_{0}}$ and that for any $\varepsilon>0$ there exists a sequence of tests $\phi_n:\mathcal{X}^n\to [0,1]$ such that
	\[\mathbb{E}_{P^n_{\vartheta_{0}}}\left[\phi_n({\boldsymbol X})\right]\to 0\qquad \text{and}\qquad \sup_{\vartheta\in \Theta:\norm{\vartheta-\vartheta_0}\ge \varepsilon}\mathbb{E}_{P_\vartheta^n}\left[1-\phi_n({\boldsymbol X})\right]\to 0,\]as $n\to \infty$.  Let $g_n: \mathbb{R}^d\to [0,\infty)$ be a sequence of Borel measurable functions and assume: \begin{enumerate}[label=(\roman*)]		\item \label{PropB1_1}There exist $N\in \mathbb{N},c_1>0$ and $0<c_2<1$ such that for all $n\ge N$
		\[ \int_{\Theta}g_n\left(n^{1/2}(\vartheta-\vartheta_{0})\right)\diff\Pi(\vartheta)\lesssim \exp\left(c_1n^{c_2}\right). \]
		\item \label{PropB1_2} There exist constants $c_3>0$ and $0<c_4<2$ such that for all $h\in \mathbb{R}^d $ \[g_n(h)\lesssim \exp\left(c_3 \norm{h}^{c_4}\right).\]
	\end{enumerate}
	Then, for any sequence $(M_n)_{n\in \mathbb{N}}$  of positive numbers such that $M_n\to \infty$, as $n\to \infty$, it holds
	\[\int_{C_n^c}g_n(h)\diff P_{h_n\vert {\boldsymbol X}}(h) \stackrel{P_{\vartheta_{0}}^n}{\to} 0, \]
	where $C_n:=\{h\in H_n:\norm{h}<  M_n\}.$ 
\end{proposition}

The second result yields under very general conditions the weak convergence of the stochastic process
 \[
 Z_n(t)=\int_{H_n}\ell(t,h)\diff P_{h_n\vert {\boldsymbol X}}(h)
 \]
 uniformly over any compact set.
 \begin{proposition} \label[proposition]{weak_conv_prop2}
	Let the assumption of the BvM theorem (\Cref{B.v.M:}) hold and let $\ell:\mathbb{R}^d\times \mathbb{R}^d\to [0,\infty)$ be a jointly measurable function. Suppose that\begin{enumerate}[label=(\roman*)]
		\item  For any compact sets $K,C\subset \mathbb{R}^d$
		\[\sup_{t\in K,h\in C}\ell(t,h)<\infty\]
		and for any $y\in \mathbb{R}^d$\[\int_{\mathbb{R}^d}\sup_{t\in K}\ell(t,h)\diff\mathcal{N}(y,I^{-1}_{\vartheta_{0}})(h)<\infty.\]
		\item The function $g:\mathbb{R}^d\to [0,\infty),h\mapsto \sup_{t\in K}\ell(t,h)$ satisfies the conditions from Proposition \ref{weak_conv_prop1}.
	\end{enumerate}
	Define the stochastic processes 
	\[
		Z_n(t):=\int_{H_n}\ell(t,h)\diff P_{h_n\vert {\boldsymbol X}}(h),\quad\text{ and }\quad
		Z(t):=\int_{\mathbb{R}^d}\ell(t,h)\diff \mathcal{N}\left(Y,I_{\vartheta_0}^{-1}\right)(h),
	\]
	where $Y\sim \mathcal{N}\left(0,I_{\vartheta_0}^{-1}\right)$ is a normal distributed random variable. Then, it holds that $Z_n\stackrel{D}{\to}_{P_{\vartheta_{0}}^n} Z$, as $n\to \infty$, in $\ell^\infty(K)$ for any compact set $K\subset\mathbb{R}^d$. 
\end{proposition}
The stated two propositions can be proven in the a similar way as Theorem~10.8 in \cite{v}, but hold for functions even with exponential growth and a proof of them is thus omitted. They are crucial for the proofs of our asymptotic results, such as  \Cref{consistency,main_thm1}.

\section{Proofs and technical derivations}
In this supplement, we state the proofs of the results from the main text.
\subsection{Proofs for \Cref{sec1_sub3}}

\begin{proof}[Proof of Proposition \ref{unif_well_implied}]
    In the first case  \ref{i:uwi:i},  it holds for $a = \min(1,2^{(1-p)}) \in (0,1]$,
   {\rev \[\ell(t,\vartheta_{0})\le a^{-1}(\ell(t,\vartheta)+\ell(\vartheta_{0},\vartheta)) \]}
by the triangle inequality of $\rho$ and the fact that for $b_1,b_2\ge 0$ and $p >0$, $(b_1+b_2)^p\le 2^{\max\{p-1,0\}}(b_1^p+b_2^p)$. By rewriting this inequality we find that
\[\ell(t,\vartheta)\ge a \ell(t,\vartheta_{0})-\ell(\vartheta_{0},\vartheta).\]
As a consequence,
\[\inf_{t,\vartheta:\norm{t-\vartheta_0}>M,\norm{\vartheta-\vartheta_0}\le \delta}\ell(t,\vartheta)\ge  a\inf_{t:\norm{t-\vartheta_0}>M}\ell(t,\vartheta_{0})-\sup_{\vartheta:\norm{\vartheta-\vartheta_0}\le \delta}\ell(\vartheta_{0},\vartheta).\]By assumption the map $\vartheta\mapsto\ell(\vartheta_{0},\vartheta)$ is continuous at $\vartheta_{0}$ and $\ell(\vartheta_{0},\vartheta_0)=0$. 
Hence, we can choose $\delta>0$ small enough such that \[a\inf_{t:\norm{t-\vartheta_0}>M}\ell(t,\vartheta_{0})-\sup_{\vartheta:\norm{\vartheta-\vartheta_0}\le \delta}\ell(\vartheta_{0},\vartheta)>0.\]

 In the second case  \ref{i:uwi:ii}, we will show that for any $M>0$ the choice $\delta=M/2$ works. Consider $t,\vartheta\in \Theta$ such that $\norm{t-\vartheta_0}>M$ and $\norm{\vartheta-\vartheta_0}\le M/2$. We {\rev obtain by the reverse triangle inequality} that
\[\norm{t-\vartheta}=\norm{(t-\vartheta_0)-(\vartheta -\vartheta_0)}\ge \norm{t-\vartheta_0}-\norm{\vartheta-\vartheta_0}\ge M-\frac{M}{2}=\frac{M}{2} .\]By monotonicity of $g$ we find that
\[\inf_{t,\vartheta:\norm{t-\vartheta_0}>M,\norm{\vartheta-\vartheta_0}\le \delta}\ell(t,\vartheta)\ge g\left(\frac{M}{2}\right)>0.\]
\end{proof}
 
\begin{proof}[Proof of Proposition \ref{consistency}] We begin by showing that under both conditions \ref{i:sep}  and \ref{i:conv}, for any $t\in \Theta$ there exists $\delta>0$ such that
\begin{equation}
  \int_{B_\delta(\vartheta_{0})}\ell(t,\vartheta)\diff P_{\theta\vert {\boldsymbol X}}(\vartheta) \stackrel{P^n_{\vartheta_0}} {\to} \ell(t,\vartheta).\label{prop4_post_exp}
\end{equation}Define $\tilde{\theta}_n=n^{-1/2}\Delta_{n,\vartheta_0}+\vartheta_{0}$, where $\Delta_{n,\vartheta_0}=n^{-1/2}\sum_{i=1}^nI^{-1}_{\vartheta_{0}}\ell_{\vartheta_{0}}'(X_i)$. The law of large numbers implies that $\tilde{\theta}_n\stackrel{P^n_{\vartheta_{0}}}{\longrightarrow}\vartheta_{0}$. Moreover, it holds by the BvM theorem that
	\[\norm{P_{\theta\vert {\boldsymbol X}}-\mathcal{N}(\tilde{\theta}_n,n^{-1}I^{-1}_{\vartheta_{0}})}_{\TV}=\norm{P_{h_n\vert {\boldsymbol X}}-\mathcal{N}(\Delta_{n,\vartheta_0},I^{-1}_{\vartheta_{0}})}_{\TV}=o_{P^n_{\vartheta_0}}(1).\]The first equality can be seen by applying the formula of the total variation distance in terms of densities and using a change of variables. We have for any $\delta>0$
	\begin{align*}\int_{B_\delta(\vartheta_{0})}\ell(t,\vartheta)\diff P_{\theta\vert {\boldsymbol X}}(\vartheta)	=\int_{B_\delta(\vartheta_{0})}\ell(t,\vartheta)\diff \mathcal{N}(\tilde{\theta}_n,n^{-1}I^{-1}_{\vartheta_{0}})(\vartheta)+R_n(t)
		,
	\end{align*}where $R_n(t)=\int_{B_\delta(\vartheta_{0})}\ell(t,\vartheta)\diff P_{\theta\vert {\boldsymbol X}}(\vartheta)-\int_{B_\delta(\vartheta_{0})}\ell(t,\vartheta)\diff \mathcal{N}(\tilde{\theta}_n,n^{-1}I^{-1}_{\vartheta_{0}})(\vartheta)$. By continuity of the loss function $\ell$ and choosing $\delta>0$ sufficiently small the BvM theorem yields\[\abs{R_n(t)}\lesssim\norm{P_{\theta\vert {\boldsymbol X}}-\mathcal{N}(\tilde{\theta}_n,n^{-1}I^{-1}_{\vartheta_{0}})}_{\TV} =o_{P^n_{\vartheta_0}}(1).\]It is left to show that
	\[\int_{B_\delta(\vartheta_{0})}\ell(t,\vartheta)\diff \mathcal{N}(\tilde{\theta}_n,n^{-1}I^{-1}_{\vartheta_{0}})(\vartheta)\stackrel{P^n_{\vartheta_{0}}}{\longrightarrow}\ell(t,\vartheta_{0}).\]For $\delta>0$ define $\ell_\delta(\vartheta):=\mathbb{I}(\norm{\vartheta-\vartheta_{0}}\le \delta)\ell(t,\vartheta)$. For some constant $c_\delta>0$ it holds that \[\sup_{\vartheta\in \Theta}\ell_\delta(\vartheta)\le c_\delta.\] Define the sequence of maps
	\[\mathcal{I}_n:\mathbb{R}^d\to [0,\infty),\;\mu \mapsto \int_{\mathbb{R}^d}\ell_\delta(\vartheta)\diff \mathcal{N}(\mu,n^{-1}I^{-1}_{\vartheta_{0}})(\vartheta).\]Let $(\mu_n)_{n \in \mathbb{N}}\subseteq \mathbb{R}^d$ be a sequence, which converges to some $\mu\in B_\delta(\vartheta_{0})^\circ$. By the substitution $z=n^{1/2}(\vartheta-\mu_n)$ we obtain
	\[\mathcal{I}_n(\mu_n)=\int_{ \mathbb{R}^d}\ell_\delta(\mu_n+n^{-1/2}z)\diff \mathcal{N}(0,I^{-1}_{\vartheta_{0}})(z).\]From the dominated convergence theorem and the continuity of $\ell$ it follows that 
	\[\lim\limits_{n\to \infty}\mathcal{I}_n(\mu_n)=\int_{ \mathbb{R}^d}\ell_\delta(\mu)\diff \mathcal{N}(0,I^{-1}_{\vartheta_{0}})(z)=\ell_\delta(\mu)=\ell(t,\mu).\]The extended continuous mapping theorem \cite[Theorem~1.11.1]{vanwell} yields
	\[\int_{B_\delta(\vartheta_{0})}\ell(t,\vartheta)\diff \mathcal{N}(\tilde{\theta}_n,n^{-1}I^{-1}_{\vartheta_{0}})(\vartheta)=\mathcal{I}_n(\tilde{\theta}_n)\stackrel{P^n_{\vartheta_{0}}}{\longrightarrow}\ell_\delta(\vartheta_{0})=\ell(t,\vartheta_{0}).\]

Let condition \ref{i:sep} hold. As the Bayes estimator minimizes the posterior risk for almost every observation, the following upper bound is valid for any $M>0$:
	\begin{align*}
		\prob{\norm{\hat{\theta}_n-\vartheta_0}>M}&\le \prob{\inf_{t\in \Theta:\norm{t-\vartheta_0}>M}\int_{\Theta}\ell(t,\vartheta)-\ell(\vartheta_0,\vartheta)\diff P_{\theta\vert {\boldsymbol X}}(\vartheta)\le 0}. 
	\end{align*}We begin by showing that
	\[\int_{\Theta}\ell(\vartheta_0,\vartheta)\diff P_{\theta\vert {\boldsymbol X}}(\vartheta)=o_{P_{\vartheta_0}^n}(1).\]
	For any $\delta>0$ it holds that
	\begin{align*}
		\int_\Theta\ell(\vartheta_0,\vartheta)\diff P_{\theta\vert {\boldsymbol X}}(\vartheta)&= \int_{\vartheta:\norm{\vartheta-\vartheta_0}> \delta}\ell(\vartheta_0,\vartheta)\diff P_{\theta\vert {\boldsymbol X}}(\vartheta)+\int_{\vartheta:\norm{\vartheta-\vartheta_0}\le \delta}\ell(\vartheta_0,\vartheta)\diff P_{\theta\vert {\boldsymbol X}}(\vartheta).
	\end{align*}Choosing $\delta>0$ small enough, we have by \eqref{prop4_post_exp} 
 \[\int_{\vartheta:\norm{\vartheta-\vartheta_0}\le \delta}\ell(\vartheta_0,\vartheta)\diff P_{\theta\vert {\boldsymbol X}}(\vartheta) \stackrel{P^n_{\vartheta_0}}{\to}\ell(\vartheta_0,\vartheta_0)=0.\] 
 	By the assumptions we can find $c_2>0$ and $0<b<2$ such that
	\[\ell(\vartheta_{0},\vartheta)\lesssim e^{c_2\norm{\vartheta}^b+c_2\norm{\vartheta_{0}}^b}\lesssim e^{c_2\norm{\vartheta}^b},\quad \forall \vartheta \in \Theta.\]As a consequence, there exists $c_3>0$ such that
	\[\ell(\vartheta_0,\vartheta_0+n^{-1/2}h)\lesssim e^{c_3\norm{n^{-1/2}h}^b}\le e^{c_3\norm{h}^b}.\]Proposition \ref{weak_conv_prop1} yields that
	\begin{align*}
		\int_{\vartheta:\norm{\vartheta-\vartheta_0}> \delta}\ell(\vartheta_0,\vartheta)\diff P_{\theta\vert {\boldsymbol X}}(\vartheta)=\int_{h:\norm{h}> \delta n^{1/2}}\ell(\vartheta_0,\vartheta_0+n^{-1/2}h)\diff P_{h_n\vert {\boldsymbol X}}(h)=o_{P^n_{\vartheta_0}}(1).
	\end{align*}
	Define $\eta:=\inf_{t,\vartheta:\norm{t-\vartheta_0}>M,\norm{\vartheta-\vartheta_0}\le \delta}\ell(t,\vartheta)$. By Assumption \ref{mon}, $\eta$ is strictly positive. We obtain for any sufficiently small $\delta>0$
	\begin{align*}
		&\inf_{t\in \Theta:\norm{t-\vartheta_0}>M}\int_{\Theta}\ell(t,\vartheta)-\ell(\vartheta_0,\vartheta)\diff P_{\theta\vert {\boldsymbol X}}(\vartheta)
		\\&\ge\inf_{t\in \Theta:\norm{t-\vartheta_0}>M}\int_{\vartheta:\norm{\vartheta-\vartheta_{0}}\le \delta}\ell(t,\vartheta)\diff P_{\theta\vert {\boldsymbol X}}(\vartheta)-\int_{\Theta}\ell(\vartheta_0,\vartheta)\diff P_{\theta\vert {\boldsymbol X}}(\vartheta)\\
		&\ge\eta P_{\theta\vert {\boldsymbol X}}\left(\{\vartheta:\norm{\vartheta-\vartheta_{0}}\le \delta\}\right)-o_{P^n_{\vartheta_{0}}}(1).
	\end{align*} By Proposition \ref{weak_conv_prop1}, for $g_n(h)=1$ and Slutsky's lemma it holds that
	\[P_{\theta\vert {\boldsymbol X}}\left(\{\vartheta:\norm{\vartheta-\vartheta_{0}}\le \delta\}\right)=P_{h_n\vert{\boldsymbol X}}\left(\{h:\norm{h}\le \delta n^{1/2}\}\right)\stackrel{P^n_{\vartheta_{0}}}{\to}1,\quad n\to \infty.\]
	Thus, the  probability that $\norm{\hat{\theta}_n-\vartheta_0}>M$ holds converges to zero for any $M>0$, as $\inf_{t\in \Theta:\norm{t-\vartheta_0}>M}\int_{\Theta}\left(\ell(t,\vartheta)-\ell(\vartheta_0,\vartheta)\right)\diff P_{\theta\vert {\boldsymbol X}}(\vartheta)$ is bounded from below by a random variable, which converges weakly to a strictly positive random variable. 
	
	  This shows the result under condition \ref{i:sep}. Consider now condition \ref{i:conv}.
 The fact that $t\mapsto \ell(t,\vartheta_0)$ is convex and $\ell$ is a loss function implies that $t\mapsto \ell(t,\vartheta_0)$ has a well-separated minimum at $\vartheta_0$, i.e.\ for any $M>0$
 \[\inf_{t\in \Theta:\norm{t-\vartheta_0}>M}\ell(t,\vartheta_0)>0.\]If we can show that for any fixed $t\in \Theta$ we have
	\begin{equation}Z_n(t)\stackrel{P^n_{\vartheta_{0}}}{\longrightarrow}\ell(t,\vartheta_{0}) \label{t.s.},\end{equation}where $Z_n(t)=\int_{\Theta}\ell(t,\vartheta)\diff P_{\theta\vert {\boldsymbol X}}(\vartheta)$, by convexity this pointwise convergence can be strengthened to uniform convergence on compacta in probability; see \citet[Theorem~II.1]{Andersen1982} or \citet[Section~1]{Pollard1991}. For any $\delta>0$ it holds that
	\[\int_{B_\delta(\vartheta_{0})^C}\ell(t,\vartheta)\diff P_{\theta\vert {\boldsymbol X}}(\vartheta)=\int_{h:\norm{h}>\delta \sqrt{n}}\ell(t,\vartheta_0+n^{-1/2}h)\diff P_{H_n\vert {\boldsymbol X}}(h)=o_{P^n_{\vartheta_{0}}}(1),\]by Proposition \ref{weak_conv_prop1}, as for some constants $c>0$ and $0<b<2$
	\[ \ell(t,\vartheta_0+n^{-1/2}h)\lesssim e^{c\norm{t}^b+c\norm{n^{-1/2}h}^b} \le e^{c\norm{t}^b+c\norm{h}^b}\lesssim e^{c\norm{h}^b},\]for all $n\in \mathbb{N}$. 
 Hence,
	\begin{align*}Z_n(t)&=\int_{B_\delta(\vartheta_{0})}\ell(t,\vartheta)\diff P_{\theta\vert {\boldsymbol X}}(\vartheta)+o_{P^n_{\vartheta_{0}}}(1).
	\end{align*}Combined with \eqref{prop4_post_exp}, we obtain $Z_n(t)\stackrel{P^n_{\vartheta_0}}{\to}\ell(t,\vartheta_0)$.
 For an arbitrary $M>0$ let $\alpha= M/(M+\norm{\hat{\theta}_n-\vartheta_{0}})$ and $T_n=\alpha\hat{\theta}_n+(1-\alpha)\vartheta_{0}$. By the convexity of $\ell(\cdot,\vartheta)$ and the definition of $\hat{\theta}_n$ it holds that
	\[Z_n(T_n)\le\alpha Z_n(\hat{\theta}_n)+(1-\alpha)Z_n(\vartheta_{0})\le Z_n(\vartheta_{0}).\]Consistency of $T_n$ follows, as  $\norm{T_n-\vartheta_{0}}< M$, and thus for any $\varepsilon>0$
	\begin{align*}
		\prob{\norm{T_n-\vartheta_{0}}>\varepsilon}&\le \prob{\ell(T_n,\vartheta_{0})\ge \inf_{t\in \Theta:\norm{t-\vartheta_{0}}>\varepsilon}\ell(t,\vartheta_{0})}\\
		&\le \prob{\ell(T_n,\vartheta_{0})+Z_n(\vartheta_{0})-Z_n(T_n)-\ell(\vartheta_{0},\vartheta_{0})\ge \inf_{t\in \Theta:\norm{t-\vartheta_{0}}>\varepsilon}\ell(t,\vartheta_{0})}\\
		&\le  \prob { 2\sup_{t\in \Theta:\norm{t-\vartheta_{0}}\le M}\abs{Z_n(t)-\ell(t,\vartheta_{0})} \ge \inf_{t\in \Theta:\norm{t-\vartheta_{0}}>\varepsilon}\ell(t,\vartheta_{0})},
	\end{align*}where the upper bound converges to zero, as $n\to \infty$. By the continuous mapping theorem we conclude that
	\[\norm{\hat{\theta}_n-\vartheta_{0}}=\frac{M \norm{T_n-\vartheta_{0}}}{M-\norm{T_n-\vartheta_{0}}}=o_{P^n_{\vartheta_{0}}}(1).\]
\end{proof} 

\subsection{Proofs for \Cref{sec2}}

\begin{proof}[Proof of \Cref{main_thm1}] 
We begin the proof by showing that it holds, for any $M_n\to \infty$, \begin{equation}\int_{ h:\norm{h}> M_n}n^{p/2}\ell(\vartheta_{0},\vartheta_{0}+n^{-1/2}h)\diff P_{h_n\vert {\boldsymbol X}}(h)\stackrel{P^n_{\vartheta_{0}}}{\to}0,\quad n\to \infty.\label{eq_divergence}\end{equation} If we can show that there exist constants $\gamma_1>0$ and $0<\gamma_2<2$ such that
\[\ell(\vartheta_{0},\vartheta)\lesssim \norm{\vartheta_{0}-\vartheta}^pe^{\gamma_1\norm{\vartheta_{0}}^{\gamma_2}+\gamma_1\norm{\vartheta}^{\gamma_2} }\]for all $\vartheta\in \Theta$, then (\ref{eq_divergence}) follows by Proposition \ref{weak_conv_prop1}, because for some $c,a>0$
\[n^{p/2}\ell(\vartheta_{0},\vartheta_{0}+n^{-1/2}h)\lesssim n^{p/2}\norm{\frac{h}{n^{1/2}}}^pe^{c\norm{n^{-1/2}h}^{\gamma} }e^{\gamma_1\norm{\vartheta_0}^{\gamma_2}}\lesssim e^{(c+a)\norm{h}^{\gamma_2}} .\]
By condition~\ref{Thm4.1a1}, we can find $\varepsilon>0$ such that \[c_1\norm{t-\vartheta}^p\le \ell(t,\vartheta)\le c_2\norm{t-\vartheta}^p\]for all $t,\vartheta\in B_\varepsilon(\vartheta_{0})$. In particular, this yields for $\vartheta\in B_\varepsilon(\vartheta_{0})$ and $c_4>0,0<c_5<2$,
\[ \ell(\vartheta_{0},\vartheta)\lesssim \norm{\vartheta-\vartheta_{0}}^p\le \norm{\vartheta-\vartheta_{0}}^pe^{c_4\norm{\vartheta_{0}}^{c_5}+c_4\norm{\vartheta}^{c_5}}.\] 
By condition~\ref{Thm4.1a3} it exists $c_3>0,c_4>0$ and $0<c_5<2$ such that
\[ \ell(\vartheta_{0},\vartheta)\le c_3  e^{c_4\norm{\vartheta_{0}}^{c_5}+c_4\norm{\vartheta}^{c_5}},\quad \forall \vartheta \in \Theta.\]Then, for $\vartheta\in B_\varepsilon(\vartheta_{0})^C$, we have
\[ \ell(\vartheta_{0},\vartheta)\lesssim  e^{c_4\norm{\vartheta_{0}}^{c_5}+c_4\norm{\vartheta}^{c_5}}\le \frac{\norm{\vartheta-\vartheta_{0}}^p}{\varepsilon^p}e^{c_4\norm{\vartheta_{0}}^{c_5}+c_4\norm{\vartheta}^{c_5}}.\] As a consequence, (\ref{eq_divergence}) follows.

	If we can show that $n^{1/2}(\hat{\theta}_n-\vartheta_{0})$ is uniformly tight, then the assertion~follows from the argmax continuous mapping theorem \cite[Corollary 5.58]{v}. Let $M_n\to \infty$ be any diverging sequence. We will show that for some $r\ge 1$, to be determined later in the proof, the term $\prob{n^{1/2}\norm{\hat{\theta}_n-\vartheta_0}>rM_n}$ converges to zero, as $n\to \infty$, which implies that $n^{1/2}(\hat{\theta}_n-\vartheta_0)$ is uniformly tight. By assumption, the Bayes estimator is consistent, i.e.\ $\hat{\theta}_n\stackrel{P^n_{\vartheta_0}}{\longrightarrow}\vartheta_0$.   We find that for any $\varepsilon>0,r\ge 1$
	\[
	\prob{n^{1/2}\norm{\hat{\theta}_n-\vartheta_0}>rM_n}\le\prob{rM_n n^{-1/2}<\norm{\hat{\theta}_n-\vartheta_0}\le \varepsilon}+\prob{\norm{\hat{\theta}_n-\vartheta_0}>\varepsilon}.
	\]
	The right term converges to zero by consistency. We find $\varepsilon>0$ small enough and $c_1,c_2>0$ such that
	\[c_1\norm{t-\vartheta}^p\le \ell(t,\vartheta)\le c_2 \norm{t-\vartheta}^p\]for any $t,\vartheta\in B_{\varepsilon}(\vartheta_0)$. Define the sequence of mappings
	\[\tau_n:\Theta\to H_n,\quad\vartheta\mapsto n^{1/2}(\vartheta-\vartheta_{0}).\] For $n\in \mathbb{N}$ such that $rM_nn^{-1/2}\le \varepsilon$ we have for $t,h\in H_n$ with $\norm{t}\le \varepsilon n^{1/2},\norm{h}\le M_n$ that $\tau_n^{-1}(t),\tau_n^{-1}(h)\in B_{\varepsilon}(\vartheta_0)$.
	Define $\ell_n:H_n\times H_n\to [0,\infty),\,(t,h)\mapsto n^{p/2}\ell(\tau_n^{-1}(t),\tau_n^{-1}(h))$. Consequently, for $t,h\in H_n$ such that $\tau_n^{-1}(t),\tau_n^{-1}(h)\in B_{\varepsilon}(\vartheta_0)$, we find that
	\[c_1\norm{t-h}^p\le \ell_n(t,h)\le c_2 \norm{t-h}^p.\] Because $n^{1/2}(\hat{\theta}_n-\vartheta_{0})$ minimizes the mapping $t\mapsto Z_n(t)$ almost surely, it holds
	\[\prob{rM_nn^{-1/2}<\norm{\hat{\theta}_n-\vartheta_0}\le \varepsilon}\le \prob{\inf_{t\in H_n:rM_n<\norm{t}\le n^{1/2}\varepsilon}Z_n(t)-Z_n(0)\le 0}.\]This term converges to zero, as $n\to \infty$, if we can show that $\inf_{t\in H_n:rM_n<\norm{t}\le n^{1/2}\varepsilon}Z_n(t)-Z_n(0)$ is bounded from below by a random variable, which weakly converges to an almost surely strictly positive random variable. By the previous considerations we obtain for any $t\in H_n$ with $rM_n<\norm{t}\le n^{1/2}\varepsilon$
	\begin{align*}
		&Z_n(t)-Z_n(0)=\int_{H_n}\left(\ell_n(t,h)-\ell_n(0,h)\right)\diff P_{h_n\vert {\boldsymbol X}}(h)\\
		=\, &\int_{h:\norm{h}\le M_n}\left(\ell_n(t,h)-\ell_n(0,h)\right)\diff P_{h_n\vert {\boldsymbol X}}(h)+\int_{h:\norm{h}>M_n}\left(\ell_n(t,h)-\ell_n(0,h)\right)\diff P_{h_n\vert {\boldsymbol X}}(h)\\
		\ge\, &\int_{h:\norm{h}\le M_n}\left(\ell_n(t,h)-\ell_n(0,h)\right)\diff P_{h_n\vert {\boldsymbol X}}(h)-\int_{h:\norm{h}>M_n}\ell_n(0,h)\diff P_{h_n\vert {\boldsymbol X}}(h).
	\end{align*}We already have shown that
	\[\int_{h:\norm{h}>M_n}\ell_n(0,h)\diff P_{h_n\vert {\boldsymbol X}}(h)=o_{P^n_{\vartheta_0}}(1).\]Furthermore, we have
	\begin{align*}
		\int_{h:\norm{h}\le M_n}\left(\ell_n(t,h)-\ell_n(0,h)\right)\diff P_{h_n\vert {\boldsymbol X}}(h)&\ge\int_{h:\norm{h}\le M_n}\left(c_1\norm{t-h}^p-c_2\norm{h}^p\right)\diff P_{h_n\vert {\boldsymbol X}}(h)\\
		&\ge P_{h_n\vert {\boldsymbol X}}(\{h:\norm{h}\le M_n\})\left[ c_1(r-1)^pM_n^p-c_2M_n^p\right]\\
		&\ge  P_{h_n\vert {\boldsymbol X}}(\{h:\norm{h}\le M_n\})M^p\left[ c_1(r-1)^p-c_2\right]
	\end{align*}for some fixed $M>0$, by choosing $r>1$ such that $\eta:=c_1(r-1)^p-c_2>0$ and large enough $n$ such that $M_n\ge M$. We conclude by Proposition \ref{weak_conv_prop1} that
	\[P_{h_n\vert {\boldsymbol X}}(\{h:\norm{h}\le M_n\})M^p\left[ c_1(r-1)^p-c_2\right]\stackrel{D}{\rightarrow}_{P^n_{\vartheta_0}}M^p\left[ c_1(r-1)^p-c_2\right].\]As this limit is strictly positive and holds for any sequence $(M_n)_{n\in \mathbb{N}}$ with $M_n\to \infty$, uniform tightness follows.
\end{proof}

\begin{proof}[Proof of Lemma \ref{weak_conv_prop3}] Note that by conditions~\ref{prop5a1} and \ref{prop5a3} we can apply \Cref{weak_conv_prop1} with $g_n(h):=\sup_{t\in K}\ell_n(t,h)$. Hence, it holds for any $M_n\to \infty$  \begin{equation}\sup_{t\in K}\int_{ h:\norm{h}> M_n}\ell_n(t,h)\diff P_{h_n\vert {\boldsymbol X}}(h)=o_{P^n_{\vartheta_0}}(1).\label{tozero}\end{equation}Moreover, condition \ref{prop5a3} implies that $\sup_{t\in K}\int_{ \mathbb{R}^d}\ell_0(t,h)\diff \mathcal{N}(y,I_{\vartheta_{0}}^{-1})(h)<\infty$ for any $t,y\in \mathbb{R}^d$, as $I_{\vartheta_{0}}$ is positive definite.
	Define the stochastic processes for any $M>0$\begin{align*}
		Z_{n,M}(t)&:=\int_{ h\in H_n:\norm{h}\le M}\ell_n(t,h)\diff P_{h_n\vert {\boldsymbol X}}(h),
		\\
		W_{n,M}(t)&:=\int_{h\in \mathbb{R}^d:\norm{h}\le M}\ell_n(t,h)\diff\mathcal{N}(\Delta_{n,\vartheta_0},I^{-1}_{\vartheta_{0}})(h),\\
		W_M(t)&:=\int_{h\in \mathbb{R}^d:\norm{h}\le M}\ell_0(t,h)\diff\mathcal{N}(Y,I^{-1}_{\vartheta_{0}})(h).
	\end{align*}
 
 For any $t\in K,h\in C=\{h:\norm{h}\le M\}$, we have for some constants $c_3>0,c_4>0,0<c_5<2$
\[\sup_{t\in K,h\in C}\ell_n(t,h)\le c_3\sup_{t\in K,h\in C}\exp\left(c_4\norm{t}^{c_5} +c_4 \norm{h}^{c_5}\right)\lesssim1 \]by the extreme value theorem in real analysis and condition~\ref{prop5a3}.
 Consequently, it holds{\rev
	\begin{align*}
		&\norm{Z_{n,M}-W_{n,M}}_{\infty,K}
        \\=\,&\sup_{t\in K}\abs{ \int_{ h\in H_n:\norm{h}\le M}\ell_n(t,h)\diff P_{h_n\vert {\boldsymbol X}}(h)
       -\int_{h\in \mathbb{R}^d:\norm{h}\le M}\ell_n(t,h)\diff\mathcal{N}(\Delta_{n,\vartheta_0},I^{-1}_{\vartheta_{0}})(h)} \\
       \lesssim\, & \abs{P_{h_n\vert {\boldsymbol X}}(\{h\in \mathbb{R}^d:\norm{h}\le M\}) - \mathcal{N}(\Delta_{n,\vartheta_0},I^{-1}_{\vartheta_{0}})(\{h\in \mathbb{R}^d:\norm{h}\le M\}) }
       \\ \lesssim\, &\norm{P_{H_n\vert {\boldsymbol X}}-\mathcal{N}(\Delta_{n,\vartheta_0},I^{-1}_{\vartheta_{0}})}_{\TV}=o_{P^n_{\vartheta_{0}}}(1),
	\end{align*}}where the last equality follows by the BvM theorem. 
	Define the operator
	\[\mathcal{I}:\mathbb{R}^d\to \ell^\infty(K),\mu \mapsto\left(t\mapsto\int_C \ell_0(t,h)\diff \mathcal{N}(\mu,I^{-1}_{\vartheta_{0}})(h)\right)\]and the sequence of operators
	\[\mathcal{I}_n:\mathbb{R}^d\to \ell^\infty(K),\mu \mapsto\left(t\mapsto\int_C \ell_n(t,h)\diff \mathcal{N}(\mu,I^{-1}_{\vartheta_{0}})(h)\right).\]
	Denote $f_\mu=\diff \mathcal{N}(\mu,I^{-1}_{\vartheta_{0}})/\diff\lambda^d$, where $\mu\in \mathbb{R}^d$ and consider an arbitrary sequence $\mu_n\to \mu$. It holds
	for $\norm{\cdot}_{\infty,K}$ the sup norm on $\ell^\infty(K)$ that
	\begin{align*}
		&\norm{\mathcal{I}_n(\mu_n)-\mathcal{I}(\mu)}_{\infty,K}\\
		\le&\int_C \sup_{t\in K}\abs{\ell_n(t,h)-\ell_0(t,h)}f_{\mu_n}(h) \diff h+\sup_{t\in K,h\in C}\ell_0(t,h)\int_C \abs{f_{\mu_n}(h)-f_\mu(h)} \diff h\\
		\lesssim&\int_C \sup_{t\in K}\abs{\ell_n(t,h)-\ell_0(t,h)} \diff h+\int_C \abs{f_{\mu_n}(h)-f_\mu(h)} \diff h,
	\end{align*} This upper bound converges to zero by {\rev condition~\ref{prop5a2}} and the dominated convergence theorem.
	By the central limit theorem $\Delta_{n,\vartheta_0}\stackrel{D}{\to}_{P^n_{\vartheta_{0}}}Y\sim \mathcal{N}(0,I^{-1}_{\vartheta_{0}})$. As a consequence, the extended continuous mapping theorem  \cite[Theorem~1.11.1]{vanwell} implies that \[\mathcal{I}_n(\Delta_{n,\vartheta_0})=W_{n,M}\stackrel{D}{\longrightarrow}_{P^n_{\vartheta_{0}}}W_M=\mathcal{I}(Y),\quad n\to \infty.\]
	Slutsky's lemma yields
	\[Z_{n,M}\stackrel{D}{\longrightarrow}W_M,\quad n\to \infty\]for any $M>0$. Moreover, we find for $c_4>0,0<c_5<2$ and $\eta_2=\sup_{t\in K}e^{c_4\norm{t}^{c_5}}$ and any $y\in \mathbb{R}^d$ that by condition~\ref{prop5a3}
	\begin{align*}
		\sup_{t\in K}\abs{Z(t)-W_M(t)}&{\rev =\sup_{t\in K} \int_{h:\norm{h}> M}\ell_0(t,h)\diff\mathcal{N}(Y,I^{-1}_{\vartheta_{0}})(h) } 
        \\ &\lesssim \eta_2\int_{h:\norm{h}> M}e^{c_4\norm{h}^{c_5}}\diff \mathcal{N}(Y,I^{-1}_{\vartheta_{0}})(h).
	\end{align*}For $\lambda^d$-a.e.\ $y\in \mathbb{R}^d$  the dominated convergence theorem yields
	\[\int_{h:\norm{h}> M}e^{c_4\norm{h}^{c_5}}\diff \mathcal{N}(y,I^{-1}_{\vartheta_{0}})(h)\to 0,\quad M\to \infty.\] Hence,
	\(W_M\stackrel{\mathbb{P}}{\to}Z,\, M\to \infty\) in $\ell^\infty(K)$. Conclude that there exists $M_n\to \infty$ such that $Z_{n,M_n}\stackrel{D}{\to}Z$ in $\ell^\infty(K)$, as $n\to \infty$. 
	Finally, by (\ref{tozero})
	\[\norm{Z_n-Z_{n,M_n}}_{\infty,K}=\sup_{t\in K}\int_{ h:\norm{h}> M_n}\ell_n(t,h)\diff P_{h_n\vert {\boldsymbol X}}(h)=o_{P^n_{\vartheta_0}}(1),\]which yields by Slutsky's Lemma
	\[Z_n\stackrel{D}{\longrightarrow}_{P^n_{\vartheta_{0}}}Z\]in $\ell^\infty(K)$.
\end{proof}

\begin{proof}[Proof of \Cref{corollary1}]
	We define the sequences of loss functions 
	$$
	\ell_n(t,h):=n^{p/2}\ell(n^{-1/2}t+\vartheta_{0},n^{-1/2}h+\vartheta_{0}),
	$$
	{\rev
	and will employ \Cref{weak_conv_prop3}. To this end, note that requirement~\ref{prop5a2} in \Cref{weak_conv_prop3} follows directly from condition~\ref{cor1a3} here. Due to compactness of $K$, there exists $N >0$ such that  $t/n^{1/2}+\vartheta \in B_{\delta}(\vartheta_0)$ for all $t\in K,n\ge N$. 
By condition~\ref{cor1a2},  we have
\[
\int_{\Theta}\sup_{t\in K}\ell_n\!\bigl(t,n^{1/2}(\vartheta-\vartheta_{0})\bigr)\,\mathrm{d} \Pi(\vartheta)\le n^{p/2}\int_{\Theta}\sup_{t\in B_\delta(\vartheta_0)}\ell(t,\vartheta)\,\mathrm{d} \Pi(\vartheta),\quad n\ge N.
\]
This shows requirement~\ref{prop5a1} in \Cref{weak_conv_prop3}.
}\label{rev:prop3.3}
 
    We will next show that for any compact set $K\subset \mathbb{R}^d$ there exist some $0<a<2$ and $\xi>0$ such that for $n$ large enough
\[\ell_n(t,h)\lesssim e^{\xi\norm{t}^a+\xi \norm{h}^a},\quad t\in K,h\in H_n,\]
i.e.~requirement~\ref{prop5a3} in \Cref{weak_conv_prop3}. 
By condition~\ref{cor1a1}, there exist $p>0$ and some sufficiently small $\varepsilon>0$ such that
	\[\ell(t,\vartheta)\lesssim \norm{t-\vartheta}^p\]for all $t,h\in B_\varepsilon(\vartheta_{0})$. Moreover, by condition~\ref{cor1a4}, we have for some $\delta>0$ that there exist $c_2>0,c_3>0,0<c_4<2$ such that 
	\begin{equation}
	\ell(t,\vartheta)\le c_2e^{c_3\norm{t}^{c_4}+c_3\norm{\vartheta}^{c_4}},\quad t\in B_\delta(\vartheta_0),\vartheta\in \Theta.\label{cor1_eq1}
	\end{equation} 
	Without loss of generality, we can assume that $\delta < \varepsilon$; otherwise consider $\delta^\prime=\min(\delta,\varepsilon/2)$ instead. Let $N>0$ be large enough such that for $\eta:=\sup_{t\in K}\norm{t}\le \delta n^{1/2}/2$ for all $n\ge N$. We separate the range of $h$ between $A_n:=\{h\in H_n:\sup_{t\in K}\norm{t-h}\le \delta n^{1/2}/2\}$ and $B_n:=\{h\in H_n:\sup_{t\in K}\norm{t-h}\ge\delta n^{1/2}/2\}$. For any $n\ge N$  and $h\in A_n$, we have
	\[ 
	\frac{\norm{h}}{n^{1/2}}\le \frac{\norm{t}}{n^{1/2}}+\frac{\norm{t-h}}{n^{1/2}}\le \delta<\varepsilon,\quad t\in K,
	\]
	i.e.\ $n^{-1/2}h+\vartheta_{0}\in B_\varepsilon(\vartheta_0)$ and $n^{-1/2}t+\vartheta_{0}\in B_\varepsilon(\vartheta_0)$. This yields, for any $t\in K$, $h\in A_n$ and come constant $c>0$,
	\[
	n^{p/2}\ell(n^{-1/2}t+\vartheta_{0},n^{-1/2}h+\vartheta_{0})\lesssim n^{p/2} \frac{\norm{t-h}^p}{n^{p/2}}\le c e^{c_3\norm{t}^{c_4}+c_3\norm{h}^{c_4}}. 
	\]
	Now consider any $h\in B_n$. We obtain with $c_3>0,0<c_4<2$ in \eqref{cor1_eq1} and some $\xi_0,\xi_1>0$,
	\begin{align*}
		n^{p/2}\ell(n^{-1/2}t+\vartheta_{0},n^{-1/2}h+\vartheta_{0})&\lesssim n^{p/2}e^{c_3(\norm{n^{-1/2}t+\vartheta_{0}}^{c_4}+\norm{n^{-1/2}h+\vartheta_{0}}^{c_4})} \\
		&\lesssim n^{p/2}e^{c_3\xi_0 (\norm{t}^{c_4}+\norm{h}^{c_4})}\\
		&\le 2^p \delta^{-p} \left(\sup_{t\in K}\norm{t-h}^p\right) e^{c_3\xi_0 (\norm{t}^{c_4}+\norm{h}^{c_4})}\\
		&\lesssim e^{c_3\xi_0\xi_1(\norm{t}^{c_4}+\norm{h}^{c_4})}
	\end{align*}for any $t\in K$. Thus, for any $t\in K$ and any $h\in H_n$ there exist some $\xi>0,0<a<2$ such that
\[
\ell_n(t,h)\lesssim e^{\xi(\norm{t}^a+\norm{h}^a)}.
\]
Now the assertion follows directly from \Cref{weak_conv_prop3}. 
\end{proof}

\subsection{Technical details for \Cref{b:integer:order}}\label[appendix]{ss:detail:io}

By continuity, there exists $\varepsilon>0$ such that for all $\vartheta \in B_\varepsilon(\vartheta_0)$, 
 \[
 D^k(\vartheta)[u]>0.
 \]
 Let $S^{d-1}=\{u\in \mathbb{R}^d:\norm{u}=1\}$ and set 
    \[
    0<\lambda_{\text{min} }:=\inf_{\vartheta\in B_\varepsilon(\vartheta_0),u\in S^{d-1}}D^k(\vartheta)[u]\le\lambda_{\text{max} }:=\sup_{\vartheta\in B_\varepsilon(\vartheta_0),u\in S^{d-1}}D^k(\vartheta)[u] <\infty.
    \]
As $D^k(\vartheta)[u]$ is homogeneous in $u\in \mathbb{R}^d$ of order $k$, we have
    \[\lambda_{\text{min} }\norm{u}^k\le \norm{u}^kD^k(\vartheta)[u/\norm{u}]=D^k(\vartheta)[u]\le \lambda_{\text{max} }\norm{u}^k\]for any $\vartheta\in  B_{\varepsilon}(\vartheta_0)$. 
     By a Taylor expansion of order $k$, 
    \[
    \ell(t,\vartheta)=D^k(\vartheta)[t-\vartheta]+\xi(t,\vartheta),
    \]
    where the remainder term $\xi(t,\vartheta)$ satisfies, for some constant $M>0$, that 
    \[
    \abs{\xi(t,\vartheta)}\le M\norm{t-\vartheta}^{k+1},
    \]
    for all $t,\vartheta \in B_\varepsilon(\vartheta_0)$. Thus, for $t,\vartheta\in B_\varepsilon(\vartheta_0)$
    \begin{align*}
        \ell(t,\vartheta)&\ge \lambda_{\text{min} }\norm{t-\vartheta}^k-M\norm{t-\vartheta}^{k+1}\\
        &\ge \norm{t-\vartheta}^k(\lambda_{\text{min} }-M\norm{t-\vartheta})\\
        &\ge \norm{t-\vartheta}^k(\lambda_{\text{min} }-2M\varepsilon).
    \end{align*}
Choosing $\varepsilon<\lambda_{\text{min} }/(2M)$ yields the lower bound. The upper bound is obtained, for $\varepsilon<1$, by
    \begin{align*}
       \ell(t,\vartheta)&\le   \lambda_{\text{max} }\norm{t-\vartheta}^k+M\norm{t-\vartheta}^{k+1}\\
       &\le \norm{t-\vartheta}^k(\lambda_{\text{max} }+2M).
    \end{align*}
    
\subsection{Technical details for Section \ref{loc_q}}
{\rev\subsubsection{Detailed derivation for equation~\eqref{e:up:lo:quad}} 
\label{sss:technical}
The upper bound follows from
\begin{align*}
\ell(t,\vartheta) &= \langle t-\vartheta,\, A(\vartheta)(t-\vartheta) \rangle + \xi(t,\vartheta) \\
&\lesssim \|t-\vartheta\|^2 \|A(\vartheta)\|_{\mathsf{F}} + \|t-\vartheta\|^{2} \|t-\vartheta\|^{c_1-2} \lesssim \|t-\vartheta\|^2,
\end{align*}
using the continuity of $\vartheta\mapsto\|A(\vartheta)\|_{\F}$ and $(t,\vartheta)\mapsto\|t-\vartheta\|^{c_1-2}$, together with the extreme value theorem {in real analysis.}

For the lower bound, let $\lambda_{\min}(\vartheta)$ denote the smallest eigenvalue of $A(\vartheta)$. By continuity of {\rev $\vartheta \mapsto A(\vartheta)$} and positive definiteness of $A(\vartheta_0)$, there exists $\delta>0$ such that
\[
\eta_\delta := \inf_{\|\vartheta-\vartheta_0\|\le \delta} \lambda_{\min}(\vartheta) > 0.
\]
Then \( \ell(t,\vartheta) = \langle t-\vartheta,\, A(\vartheta)(t-\vartheta) \rangle + \xi(t,\vartheta) \ge \eta_\delta \|t-\vartheta\|^2 - |\xi(t,\vartheta)|. \) Since
\[
\|t-\vartheta\|^2\big(\eta_\delta - \|t-\vartheta\|^{c_1-2}\big)
= \eta_\delta \|t-\vartheta\|^2 - \|t-\vartheta\|^{c_1}
\lesssim \eta_\delta \|t-\vartheta\|^2 - |\xi(t,\vartheta)|,
\]
and $\eta_\delta$ is nonincreasing in $\delta$, we can choose $\varepsilon>0$ and $\eta>0$ such that $\eta_\delta\ge\eta>0$ for all $\delta<\varepsilon$. If $\delta<\varepsilon$ and $t,\vartheta\in B_\delta(\vartheta_0)$, then
\[
\|t-\vartheta\|^2\big(\eta_\delta - \|t-\vartheta\|^{c_1-2}\big)
\ge \|t-\vartheta\|^2\big(\eta - (2\delta)^{c_1-2}\big),
\]
and choosing $\delta$ small enough so that $\eta - (2\delta)^{c_1-2} > 0$ yields the lower bound in~\eqref{e:up:lo:quad}.
}

\subsubsection{Proof of \Cref{main_thm2}}

We apply \Cref{main_thm1} with $\ell_0=\s{t-h}{A(\vartheta_{0})(t-h)}$, which yields the result. It therefore suffices to verify requirement~\ref{Thm4.1a2} in \Cref{main_thm1}. 

To this end, we invoke \Cref{corollary1}, {\rev for which requirement~\ref{cor1a1} follows from \Cref{assumption3}, particularly, from equation~\eqref{e:up:lo:quad}, whose derivation is given above.} Requirements~\ref{cor1a2} and \ref{cor1a4} in \Cref{corollary1} are satisfied by~\Cref{assumption3,assumption2}. 
It remains to verify {\rev requirement~\ref{cor1a3} in \Cref{corollary1}}, namely, for any $h\in \mathbb{R}^d$,
\[
\sup_{t\in  K}\abs{\frac{\ell(t\varepsilon+\vartheta_0,h\varepsilon+\vartheta_0)}{\varepsilon^2}-\ell_0(t,h)}\to 0,\quad \varepsilon\to 0.
\]
For any $\delta>0$ there exists $\varepsilon>0$ suffiiciently small such that $t\varepsilon+\vartheta_0, h\varepsilon+\vartheta_0 \in B_\delta(\vartheta_0)$ for all $t\in K$. Choose $\delta>0$ small enough and some $c_1>2$, such that $\abs{\xi(t,\vartheta)}\lesssim \norm{t-\vartheta}^{c_1}$ for all $t,\vartheta \in B_\delta(\vartheta_0)$. Then
\begin{align*}
    &\sup_{t\in  K}\abs{\frac{\ell(t\varepsilon+\vartheta_0,h\varepsilon+\vartheta_0)}{\varepsilon^2}-\ell_0(t,h)}\\
    =\,&\sup_{t\in  K}\abs{\s{t-h}{A(h\varepsilon+\vartheta_0)(t-h) } +\frac{1}{\varepsilon^2} \xi(t\varepsilon+\vartheta_0,h\varepsilon+\vartheta_0)-\ell_0(t,h)}\\
    \le\,&  \sup_{t\in  K}\norm{t-h}^2\norm{A(\vartheta_{0}+h\varepsilon)-A(\vartheta_{0}))}_{\F} + \sup_{t\in  K}\frac{1}{\varepsilon^2}\abs{\xi(t\varepsilon+\vartheta_0,h\varepsilon+\vartheta_0)}\\
    \lesssim\,&\norm{A(\vartheta_{0}+h\varepsilon)-A(\vartheta_{0}))}_{\F} +\varepsilon^{c_1-2} \sup_{t\in  K}\norm{t-h}^{c_1}.\end{align*}This upper bound converges to zero, as $\varepsilon\to 0$ by continuity of the components of $A$ and the Frobenius norm.
\hfill\qedsymbol

\subsection{Technical derivations for \Cref{sec5}}\label[appendix]{ss:prf:s5}
\subsubsection{Additional details in \Cref{L1}}
\label{additional_details_ex3}
{\rev 
The reverse triangle inequality and the Cauchy--Schwarz inequality yield uniform convergence as follows
\begin{align*}
   &\sup_{t\in K} \abs{\varepsilon^{-1}\ell(\varepsilon t +\vartheta_0,\varepsilon h+\vartheta_0)-\ell_0(t,h)}\\
   \le\; &\sup_{t\in K} \abs{\norm{\s{t-h}{\dot p_{\vartheta_0+\varepsilon h}}}_{L^1}-\norm{\s{t-h}{\dot p_{\vartheta_0}}}_{L^1}}+\sup_{t\in K} \varepsilon^{-1}\abs{R(\varepsilon t +\vartheta_0,\varepsilon h+\vartheta_0)}\\
   \le\; & \sup_{t\in K} \norm{\s{t-h}{\dot p_{\vartheta_0+\varepsilon h}-\dot p_{\vartheta_0}}}_{L^1}+o(1)
   \le \left(\sup_{t\in K}\norm{t-h}\right) \norm{\dot p_{\vartheta_0+\varepsilon h}-\dot p_{\vartheta_0}}_{L^1}+o(1),
\end{align*}
which converges to zero, as $\varepsilon\to 0$, by continuity.  }
\subsubsection{Additional details in \Cref{b:norm}}

We show the claim in \Cref{b:norm}: For any compact $K\subset \mathbb{R}^d$ and any $h\in \mathbb{R}^d$
    \begin{align*}
   \sup_{t\in K} \abs{\varepsilon^{-1}\ell(\varepsilon t +\vartheta_0,\varepsilon h+\vartheta_0)-\ell_0(t,h)}\to 0,\qquad \text{as}\quad\varepsilon \to 0.
\end{align*}
For ease of notation, let $t_\varepsilon=t\varepsilon+\vartheta_0,h_\varepsilon=h\varepsilon+\vartheta_0$ and $g_\vartheta(t)=\norm{p_t-p_\vartheta-\s{t-\vartheta}{\dot p_{\vartheta}}}_{L^1}$. Note that $L^1$-differentiability implies that $g_h(t)=o(\norm{t-h})$, as $t\to h$. Thus, for any $t\in K$, where $K\subset \mathbb{R}$ is any compact set, it holds
    \begin{align*}
        \abs{\frac{\ell(t_\varepsilon,h_\varepsilon)}{\varepsilon}-\ell_0(t,h)}&\le \varepsilon^{-1}\norm{p_{t_\varepsilon}-p_{h_\varepsilon}-\s{t_\varepsilon-h_\varepsilon}{\dot p_{\vartheta_0}}}_{L^1}\\
        &=\varepsilon^{-1}\norm{p_{t_\varepsilon}-p_{\vartheta_0}-\s{t_\varepsilon-\vartheta_0}{\dot p_{\vartheta_0}}+p_{\vartheta_0}-p_{h_\varepsilon}-\s{\vartheta_0-h_\varepsilon}{\dot p_{\vartheta_0}}}_{L^1}
        \\
        &\le\frac{ g_{\vartheta_0}(t_\varepsilon)+g_{\vartheta_0}(h_\varepsilon) }{\varepsilon}.
    \end{align*}By $L^1$-differentiability
    \[\frac{ g_{\vartheta_0}(h_\varepsilon) }{\varepsilon} \to 0,\quad \varepsilon\to 0.\] Thus, the only thing left to show for the uniform convergence over all $t\in K$ is that also
    \[
    \sup_{t\in K}\frac{ g_{\vartheta_0}(t_\varepsilon) }{\varepsilon}=\sup_{t\in K}\frac{\norm{p_{t_\varepsilon}-p_{\vartheta_0}-\s{t_\varepsilon-\vartheta_0}{\dot p_{\vartheta_0}}}_{L^1}}{\varepsilon}\to 0.
    \] 
    For an arbitrarily fixed $x\in \mathcal{X}$, denote 
    \[f_\varepsilon(t)=\abs{\frac{p_{\vartheta_0+\varepsilon t}(x)-p_{\vartheta_0}(x) }{\varepsilon}-\s{t}{\dot p_{\vartheta_0}(x)} }.\] Note that $f_\varepsilon(t)$ converges to zero for any $t$, by differentiability of the normal density in the parameter. If we can show that it also holds uniformly in $t$, we conclude that also $\sup_{t\in K}\varepsilon^{-1} g_{\vartheta_0}(t_\varepsilon)\to 0$, by the dominated convergence theorem. The required property is directional differentiability of $\vartheta \mapsto p_\vartheta(x)$ at $\vartheta_0$ in direction $t$ uniformly over all $t\in K$. This follows from the fact that the normal density is continuously differentiable in the parameter: For $\eta=\sup_{t\in K}\norm{t}<\infty$ and $0\le \alpha \le 1$, the mean value theorem yields
    \[f_\varepsilon(t)=\abs{\s{t}{\dot p_{\vartheta_0+\eta \alpha t}(x)-\dot p_{\vartheta_0}(x)   }}\le \eta \sup_{\vartheta\in B_{\eta \varepsilon}(\vartheta_0)} \norm{\dot p_{\vartheta}(x)- \dot p_{\vartheta_0}(x)}.\]The upper bound converges to zero, as $\varepsilon\to 0$, by continuity.

{
\subsubsection{Additional details in \Cref{ss:beyond}}\label[appendix]{sss:beyond}
Note that
\[
\frac{1}{n}\sum_{i=1}^n W_2^2(t,X_i)=\sum_{j=1}^d N_jW_2^2(t,y_j).
\]
By Lemma~\ref{lem:W2-mult}, we have $W_2^2(t,y_j)=\sum_{k=1}^{d-1}\abs{t_k-y_{j,k}}$, which yields
\begin{align*}
    &\min_{t\in [0,1]^d:\norm{t}_1=1}\frac{1}{n}\sum_{i=1}^n W_2^2(t,X_i)\\
    =\;&\min_{(t_1,\dots,t_{d-1})\in [0,1]^{d-1}:\sum_{k=1}^{d-1}t_k\le 1}\sum_{j=1}^d N_j\sum_{k=1}^{d-1}\abs{t_k-y_{j,k}}
    \\=\;&\min_{(t_1,\dots,t_{d-1})\in [0,1]^{d-1}:\sum_{k=1}^{d-1}t_k\le 1} \sum_{k=1}^{d-1}\left(N_k(1-t_k) +\sum_{j=1,\dots,d;j\neq k} N_jt_k\right).
\end{align*}
Note that $\sum_{j=1}^d N_j=1$. Then, it holds
\begin{align*}
    \sum_{k=1}^{d-1}\left(N_k(1-t_k)+\sum_{j=1,\dots,d;j\neq k} N_jt_k\right)&=\sum_{k=1}^{d-1}\bigl(N_k(1-t_k)+(1-N_k)t_k\bigr)\\
    &=\sum_{k=1}^{d-1}N_k+\sum_{k=1}^{d-1}(1-2N_k)t_k.
\end{align*}
We can write the optimization problem as a standard form linear program
\[\min_{x\ge 0}c^\intercal x, \quad \text{s.t.}\;Ax=b,\]where
\begin{align*}
    x=\left(\begin{matrix}
        t_1\\
        \vdots\\
        t_{d- 1}\\
        u_1\\
         \vdots\\
        u_{d- 1}\\
        v
    \end{matrix}\right)\in \mathbb{R}^{2(d-1)+1},\; c=\left(\begin{matrix}
        1-N_1\\
        \vdots\\
1-N_{d- 1}\\
        0\\
         \vdots\\
        0
    \end{matrix}\right)\in \mathbb{R}^{2(d-1)+1},\; b=\left(\begin{matrix}
        1\\
        \vdots\\
        1
    \end{matrix}\right)\in \mathbb{R}^{d}
\end{align*}and
\[A
=\left(\begin{matrix}
    \mathcal{I}_{d-1}&   \mathcal{I}_{d-1}& \mathbf{0}_{d-1}\\
    \mathbf{1}_{d-1}^\intercal& \mathbf{0}_{d-1}^\intercal &1
\end{matrix}\right)\in \mathbb{R}^{d\times 2(d-1)+1}.
    \]
In addition, an equivalent optimization problem is given by the dual problem
\[\min_{\lambda \in \mathbb{R}^d}-b^\intercal \lambda, \quad \text{s.t.} \; A^\intercal \lambda \le c ,\]which can be rewritten, by replacing $\lambda \in \mathbb{R}^d$ by $u-v\in \mathbb{R}^d$, where $u,v\ge 0$, into a standard form as
\[\min_{u,v,z\ge 0}(b^\intercal,-b^\intercal )\left(\begin{matrix}
    u\\
    v
\end{matrix}\right), \quad \text{s.t.} \; (A^\intercal,-A^\intercal,\mathcal{I}_{d}) \left(\begin{matrix}
    u\\
    v\\
    z
\end{matrix}\right) = c .
\]
}

 \subsection{Technical derivations for \Cref{sec6}}\label[appendix]{ss:prf:s6}

\begin{lemma}\label{lem:W2-mult}
Let $\mathcal{X}=\{y_1,\dots,y_d\}$ where $y_j$ is the $j$th unit vector in $\mathbb{R}^d$, and let $P_p$ and $P_\vartheta$ be multinomial measures supported on $\mathcal{X}$ with parameters $p,\vartheta\in[0,1]^d$ satisfying $\sum_{i=1}^d p_i = \sum_{i=1}^d \vartheta_i = 1$.  Then the squared 2-Wasserstein distance between $P_p$ and $P_\vartheta$ is 
\[
W_2^2(P_p,P_\vartheta) = 1 - \sum_{i=1}^d \min(p_i,\vartheta_i).
\]
\end{lemma}

\begin{proof}
The optimal transport problem is
\[
\min_{\pi\in[0,1]^{d\times d}} \ \mathrm{tr}(\pi^\intercal C)
\quad\text{subject to }\quad \pi\mathbf{1} = p,\quad \pi^\intercal\mathbf{1} = \vartheta,
\]
where $\mathbf{1}=(1,\dots,1)^\intercal$ is the vector with all components equal to one, and $C\in \mathbb{R}^{d\times d}$ is the cost matrix with $C_{ij} = \mathbb{I}(i\neq j)$, corresponding to squared Euclidean distance between $y_i$ and $y_j$. 

Let ${I}_d$ be the $d\times d$ identity matrix. Then, we have, for any feasible $\pi \in [0,1]^{d\times d}$,
\[ 
\tr(\pi^\intercal C)=1- \tr(\pi^\intercal {I}_d)=1-\sum_{i=1}^d \pi_{i,i}.
\]
Denote $d_i=\pi_{i,i}$ and the slack variables $s_{i,j}=\pi_{i,j}$ for $i \neq j$, we can rewrite the optimization problem as
\begin{align*}
    &\min_{0\le d_i\le 1}  1-\sum_{i=1}^d d_i\\
    \text{subject to } &d_i+\sum_{j\neq i}s_{i,j}=p_i, d_j+\sum_{i\neq j}s_{i,j}=\vartheta_j, s_{i,j}\ge 0, i,j=1,\dots,d.
\end{align*}As the objective value is independent of the $s_{i,j}$, the problem is equivalent to
\begin{align*}
    &  1-\max_{0\le d_i\le 1}\sum_{i=1}^d d_i\\
    \text{subject to }&d_i\le p_i, d_j\le \vartheta_j, i,j=1,\dots,d.
\end{align*}
This problem is solved explicitly by
\[ 
W_2^2(P_p,P_\vartheta) =1-\sum_{i=1}^d \min(p_i,\vartheta_i),
\]
which shows the assertion. 
\end{proof}
\end{document}